\documentclass{amsart}
\usepackage{amssymb}
\usepackage{graphicx}
\usepackage{amscd}
\usepackage{url}

\allowdisplaybreaks
\numberwithin{figure}{section}
\numberwithin{table}{section}
\newtheorem{thm}{Theorem}[section]

\newtheorem{lem}[thm]{Lemma}
\newtheorem{cor}[thm]{Corollary}
\newenvironment{remark}[1]{\medskip \noindent {\it Remark.} #1}{}

\newcommand\tr{\operatorname{tr}}

\newcommand\grad{\operatorname{grad}}
\renewcommand\div{\operatorname{div}}
\newcommand\curl{\operatorname{curl}}

\newcommand\gap{\operatorname{gap}}
\newcommand\vol{\mathsf{vol}}

\newcommand\R{\mathbb{R}}
\newcommand\eps{\operatorname\epsilon}
\newcommand\x{\times}

\newcommand\B{{\mathcal B}}
\newcommand\Bfrak{\mathfrak B}
\newcommand\D{d}

\renewcommand\H{\mathcal H}

\newcommand\Hfrak{\mathfrak H}
\newcommand\id{\operatorname{id}}

\newcommand\Lin{\mathcal L}

\renewcommand\P{{\mathcal P}}
\newcommand\Th{{\mathcal T}_h}
\newcommand\T{\mathcal T}

\newcommand\V{{\mathbb V}}

\newcommand\Zfrak{\mathfrak Z}
\newcommand\Sym{{\mathbb S}}
\newcommand\Hdiv{H(\div;\Omega;\Sym)}

\newcommand{\<}{\langle}
\renewcommand{\>}{\rangle}
\newcommand{\0}{\mathaccent23}

\newcommand\Alt{\operatorname{Alt}}
\DeclareMathOperator{\sign}{sign}

\begin{document}
\title[Finite element exterior calculus]{Finite element
exterior calculus: \\
from Hodge theory to numerical stability}

\author{Douglas N. Arnold}
\address{School of Mathematics,
University of Minnesota, Minneapolis, MN 55455}
\email{arnold@umn.edu}
\thanks{}
\author{Richard S. Falk}
\address{Department of Mathematics,
Rutgers University, Piscataway, NJ 08854}
\email{falk@math.rutgers.edu}
\thanks{}
\author{Ragnar Winther}
\address{Centre of Mathematics for Applications
and Department of Informatics,
University of Oslo, 0316 Oslo, Norway}
\email{ragnar.winther@cma.uio.no}
\thanks{}
\date{August 12, 2009}
\thanks{The work of the first author was supported in part by NSF grant
DMS-0713568.}

\thanks{The work of the second author was supported in part by NSF grant
DMS-0609755.}

\thanks{The work of the third author was supported by the Norwegian
Research Council}

\thanks{The authors would like to thank Marie Rognes for her help
with some of the computations appearing in Section~\ref{subsec:examples}.}

\dedicatory{}

\begin{abstract}
This article reports on the confluence of two streams of research,
one emanating from the fields of numerical analysis and scientific computation,
the other from topology and geometry.
In it we consider the numerical discretization of partial differential equations
that are related to differential complexes so that de~Rham cohomology
and Hodge theory are key tools for exploring the well-posedness of
the continuous problem.  The discretization methods we consider are
finite element methods, in which a variational or weak formulation of
the PDE problem is approximated by restricting the trial subspace to
an appropriately constructed piecewise
polynomial subspace.  After a brief introduction to finite element
methods, we develop an abstract Hilbert space framework for analyzing the stability
and convergence of such discretizations.  In this framework,
the differential complex is represented by a complex of
Hilbert spaces and stability is obtained
by transferring Hodge
theoretic structures that ensure well-posedness of the continuous problem
from the continuous level to the discrete.
We show stable discretization is achieved if the finite element spaces
satisfy two hypotheses: they can be arranged into a
subcomplex  of this Hilbert complex, and there exists a bounded
cochain projection from that complex to the subcomplex.
In the next part of the paper, we consider
the most canonical example of the abstract theory, in which the Hilbert complex
is the de~Rham complex of a domain in Euclidean space.
We use the Koszul complex to construct two families of finite element differential forms,
show that these can be arranged in subcomplexes of the de~Rham complex in numerous ways,
and for each construct a bounded cochain
projection.   The abstract theory therefore applies to give the stability and convergence
of finite element approximations of the Hodge Laplacian.
Other applications are considered as well, especially the elasticity
complex and its application to the equations of elasticity.
Background material is included to make the presentation self-contained for a variety
of readers.
\end{abstract}

\maketitle

\tableofcontents

\section{Introduction}\label{sec:intro}
Numerical algorithms for the solution of partial differential equations
are an essential tool of the modern world.   They are applied in
countless ways every day in problems as varied as the design of aircraft,
prediction of climate, development of cardiac devices, and modeling
of the financial system.  Science, engineering, and technology depend
not only on efficient and accurate algorithms for approximating the
solutions of a vast and diverse array of differential equations which
arise in applications, but also on mathematical analysis of the behavior
of computed solutions, in order to validate the computations, determine
the ranges of applicability, compare different algorithms, and point the
way to improved numerical methods.  Given a partial differential equation (PDE)
problem, a numerical algorithm approximates the solution by the solution
of a finite dimensional problem which can be implemented and solved on
a computer.  This discretization depends on a parameter---representing,
for example, a grid spacing, mesh size, or time step---which can be
adjusted to obtain a more accurate approximate solution, at the cost
of a larger finite dimensional problem.  The mathematical analysis
of the algorithm aims to describe the relationship between the true
solution and the numerical solution as the discretization parameter
is varied.  For example, at its most basic, the theory attempts to
establish convergence of the discrete solution to the true solution in
an appropriate sense as the discretization parameter tends to zero.

Underlying the analysis of numerical methods for PDEs is the
realization that convergence depends on consistency and stability of
the discretization.  Consistency, whose precise definition depends
on the particular PDE problem and the type of numerical method,
aims to capture the idea that the operators and data defining the
discrete problem are appropriately close to those of the true problem
for small values of the discretization parameter.
The essence of stability is that the discrete
problem is well-posed, uniformly with respect to the discretization
parameter.  Like well-posedness of the PDE problem itself, stability
can be very elusive.  One might think that well-posedness of the PDE
problem, which means invertibility of the operator, together with
consistency of the discretization, would imply invertibility of the
discrete operator, since invertible operators between a pair of Banach
spaces form an open set in the norm topology.  But this reasoning is
bogus.  Consistency is not and cannot be defined to mean norm
convergence of the discrete operators to the PDE operator, since the
PDE operator, being an invertible operator between infinite dimensional
spaces, is not compact and so is not the norm limit of finite dimensional
operators.  In fact, in the first part of the preceding century, a
fundamental, and initially unexpected, realization was made: that a
consistent discretization of a well-posed problem need not be stable
\cite{cfl,vonneumann-goldstine,charney-fjortoft-vonneumann}.  Only for
very special classes of problems and algorithms does well-posedness
at the continuous level transfer to stability at the discrete level.
In other situations, the development and analysis of stable, consistent
algorithms can be a challenging problem, to which a wide array of
mathematical techniques have been applied, just as for establishing
the well-posedness of PDEs. 

In this paper we will consider PDEs that are related to differential
complexes, for which de~Rham cohomology and Hodge theory are key tools
for exploring the well-posedness of the continuous problem.  These are
linear elliptic PDEs, but they are a fundamental component of problems
arising in many mathematical models, including parabolic, hyperbolic,
and nonlinear problems.  The finite element exterior calculus, which
we develop here, is a theory which was developed to capture the key
structures of de~Rham cohomology and Hodge theory at the discrete level
and to relate the discrete and continuous structures, in order to obtain
stable finite element discretizations.

\subsection{The finite element method}\label{subsec:fem}
The finite element method, whose development as an approach to
the computer solution of PDEs began over 50 years ago and is still
flourishing today, has proven to be one of the most important technologies
in numerical PDEs.  Finite elements not only provide a methodology to
develop numerical algorithms for many different problems, but also a
mathematical framework in which to explore their behavior.  
They are based on a weak or variational form of the
PDE problem, and fall into the class of Galerkin methods, in which the
trial space for the weak formulation is replaced by a finite dimensional
subspace to obtain the discrete problem.  For a finite element method, this
subspace is a space of piecewise polynomials defined by specifying three
things: a simplicial decomposition of the domain, and, for each simplex,
a space of polynomials called the shape functions, and a set of degrees
of freedom for the shape functions, i.e., a basis for their dual space,
with each degree of freedom associated to a face of some dimension of
the simplex.  This allows for efficient implementation of the global
piecewise polynomial subspace, with the degrees of freedom determining
the degree of interelement continuity.

For readers unfamilar with the finite element method, we introduce some basic ideas
by considering the approximation of the simplest two-point
boundary value problem
\begin{equation}\label{1d}
-u^{\prime\prime}(x) = f(x), \quad -1 < x < 1, \qquad u(-1) = u(1) =0.
\end{equation}
Weak solutions to this problem are sought in the Sobolev space $H^1(-1,1)$
consisting of functions in $L^2(-1,1)$ whose first derivative
also belong to $L^2(-1,1)$.
Indeed, the solution $u$ can be characterized as the minimizer of the energy
functional
\begin{equation}\label{min}
J(u):= \frac{1}{2} \int_{-1}^1 |u^{\prime}(x)|^2 \, dx - \int_{-1}^1
f(x) u(x) \, dx
\end{equation}
over the space $\0 H^1(-1,1)$ (which consists of $H^1(-1,1)$ functions
vanishing at $\pm 1$), or, equivalently, as the solution of the weak problem:
Find $u \in \0H^1(-1,1)$ such that
\begin{equation}\label{wf}
\int_{-1}^1 u^{\prime}(x) v^{\prime}(x) \, dx = \int_{-1}^1 f(x) v(x) \, dx,
\quad v \in \0H^1(-1,1).
\end{equation}
It is easily seen by integrating by parts that a smooth solution
of the boundary value problem satisfies the weak formulation and that
a solution of the weak formulation which posesses appropriate smoothness
will be a solution of the boundary value problem.

Letting $V_h$ denote a finite dimensional subspace of $\0 H^1(-1,1)$, called
the \emph{trial space}, we may define an approximate solution $u_h \in V_h$ as
the minimizer of the functional $J$ over the trial space (the classical Ritz
method), or, equivalently by Galerkin's method, in which $u_h \in V_h$ is
determined by requiring that the variation given in the weak problem hold only
for functions in $V_h$, i.e., by the equations
\begin{equation*}
\int_{-1}^1 u_h^{\prime}(x) \, v^{\prime}(x) \, dx
= \int_{-1}^1 f(x) \, v(x) \, dx, \quad  v \in V_h.
\end{equation*}
By choosing a basis for the trial space $V_h$, the Galerkin method reduces to
a linear system of algebraic
equations for the coefficients of the expansion of $u_h$ in terms of the basis
functions.  More specifically, if we write $u_h = \sum_{j=1}^M c_j
\phi_j$, where the functions $\phi_j$ form a basis for the trial space, then
the Galerkin equations hold if and only if $A c=b$ where the coefficient matrix
of the linear system is given by 
 $A_{ij} =\int_{-1}^1 \phi_j^{\prime} \, \phi_i^{\prime} \, dx$ and
$b_i=\int_{-1}^1 f \phi_i \, dx$. Since this is a square
system of linear equations, it is nonsingular if and only if the only
solution for $f=0$ is  $u_h=0$.  This follows
immediately by choosing $v =u_h$.

The simplest finite element method is obtained applying Galerkin's method
with the trial space $V_h$ consisting of all element of $\0H^1(-1,1)$ which
are linear on each subinterval of some chosen mesh $-1=x_0<x_1<\cdots <x_N=1$
of the domain $(-1,1)$.  Figure~\ref{fg:simplest} compares the exact solution
$u=\cos(\pi x/2)$ and the finite element solution $u_h$ in the case of a uniform mesh with
$N=14$ subintervals.  The derivatives are compared as well.
For this simple problem, $u_h$ is simply the orthogonal projection of $u$
into $V_h$ with respect to the inner product defined by the left hand side of \eqref{wf},
and the finite element method gives good approximation even with a fairly coarse mesh.
Higher accuracy can easily be obtained by
using a finer mesh or piecewise polynomials of higher degree.
\begin{figure}[htb]
\centerline{\raise-.05in\hbox{\includegraphics[height=1.85in]{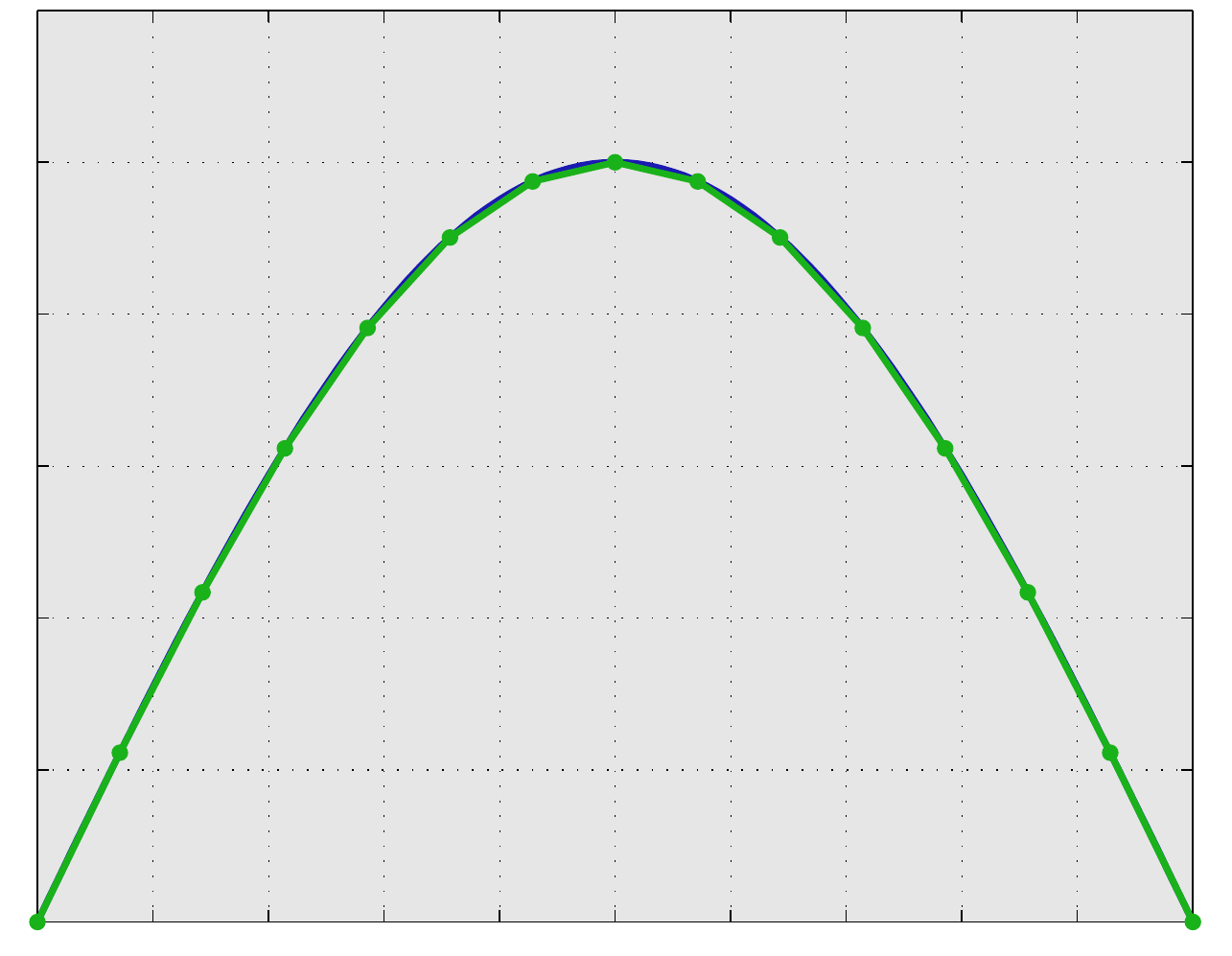}}\quad
\includegraphics[height=1.8in]{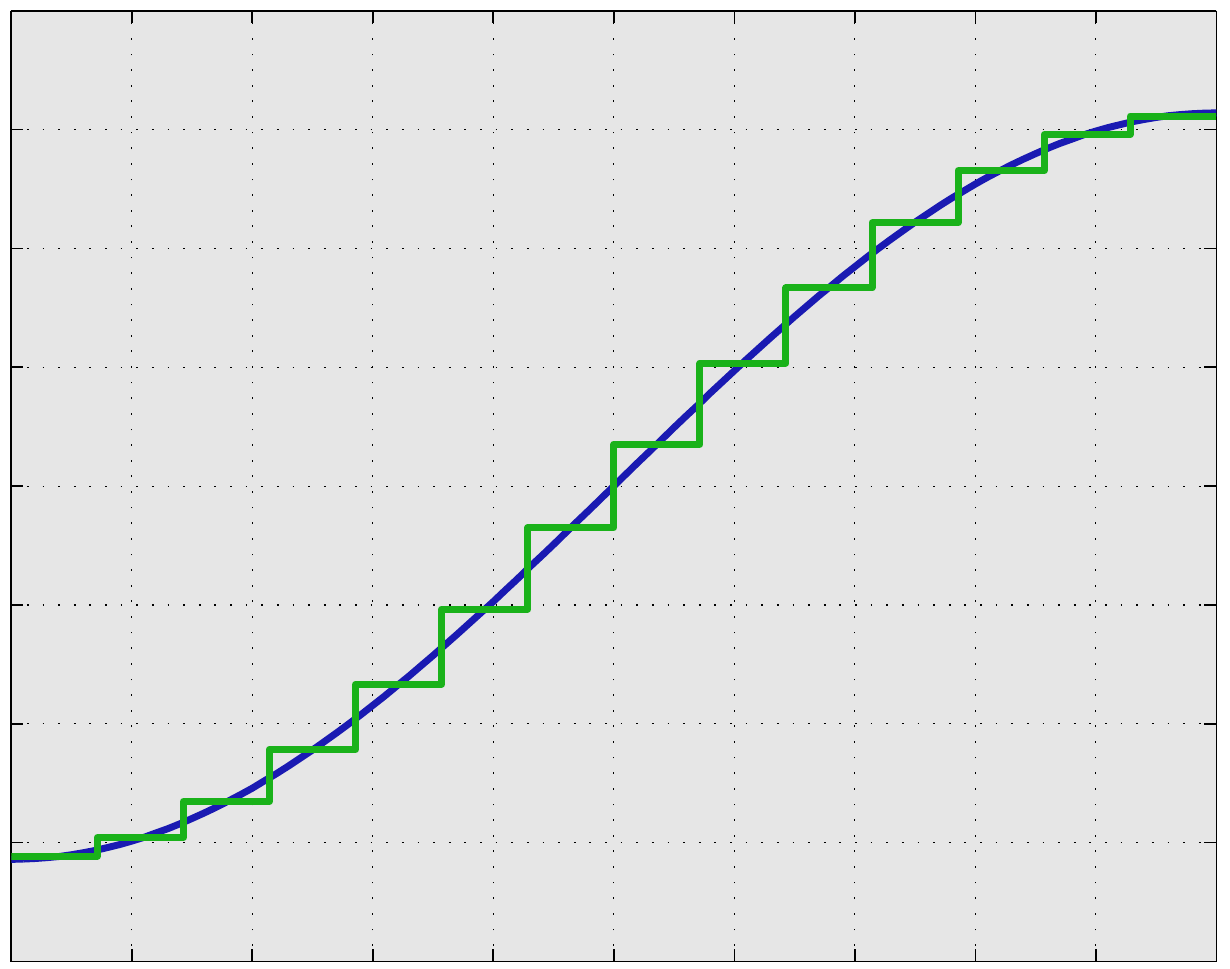}}
\caption[]{Approximation of
$-u^{\prime\prime}=f$ by the simplest finite element method.  The
left plot shows $u$ and the right plot shows $-u'$, with the exact solution
in blue and the finite element solution in green.}
\label{fg:simplest}
\end{figure}

The weak formulation \eqref{wf} associated to minimization of the functional \eqref{min}
is not the only variational formulation that can be used for discretization, and in
more complicated situations other formulations may bring important advantages.
In this simple situation, we may, for example, start by
writing the differential equation $-u^{\prime\prime} = f$ as
the first order system
\begin{equation*}
\sigma = - u^{\prime}, \quad \sigma^{\prime} =f.
\end{equation*}
The pair $(\sigma, u)$ can then be characterized variationally as the unique
critical point of the functional
\begin{equation*}
I(\sigma, u) = \int_{-1}^1 (\frac{1}{2} \sigma^2
- u \sigma^{\prime}) \, dx + \int_{-1}^1 f u \, dx
\end{equation*}
over $H^1(-1,1) \times L^2(-1,1)$. Equivalently, the pair is
the solution of the weak formulation: Find $\sigma \in H^1(-1,1),
u \in L^2(-1,1)$ satisfying
\begin{gather*}
\int_{-1}^1 \sigma \tau \, dx - \int_{-1}^1 u \tau^{\prime} \, dx
= 0, \quad \tau \in H^1(-1,1),
\\
\int_{-1}^1 \sigma^{\prime} v \, dx =  \int_{-1}^1 f v \, dx,
\quad v \in L^2(-1,1).
\end{gather*}
This is called a \emph{mixed formulation} of the boundary value problem.  Note that
for the mixed formulation, the Dirichlet boundary condition is implied
by the formulation, as can be seen by integration by parts.  Note also that in
this case the solution is a saddle point of the functional $I$, not an extremum:
$I(\sigma,v)\le I(\sigma,u)\le I(\tau,u)$ for $\tau\in H^1(-1,1)$, $v\in
L^2(-1,1)$.

Although the mixed formulation is perfectly well-posed, it may easily lead to
a discretization which is not.  If we apply Galerkin's method with a simple choice
of trial subspaces $\Sigma_h\subset H^1(-1,1)$ and $V_h\subset L^2(-1,1)$, we obtain
a finite-dimensional linear system, which, however, may be singular, or may become increasingly \emph{unstable} as
the mesh is refined. This concept will be formalized and
explored in the next section, but the result of such instability is clearly
visible in simple computations.  For example,
the choice of continuous piecewise linear functions
for both $\Sigma_h$ and $V_h$ leads to a singular linear
system.  The choice of continuous piecewise linear
functions for $\Sigma_h$ and piecewise constants for $V_h$ leads to stable
discretization and good accuracy.
However choosing
piecewise quadratics for $\Sigma_h$ and piecewise constants for $V_h$ gives a
nonsingular system but unstable approximation (see
\cite{Brezzi-Bathe} for further discussion of this example).  The dramatic difference between the stable and unstable
methods can be seen in Figure~\ref{fg:f0}.
\begin{figure}[htb]
\centerline{\raise-.05in\hbox{\includegraphics[height=1.85in]{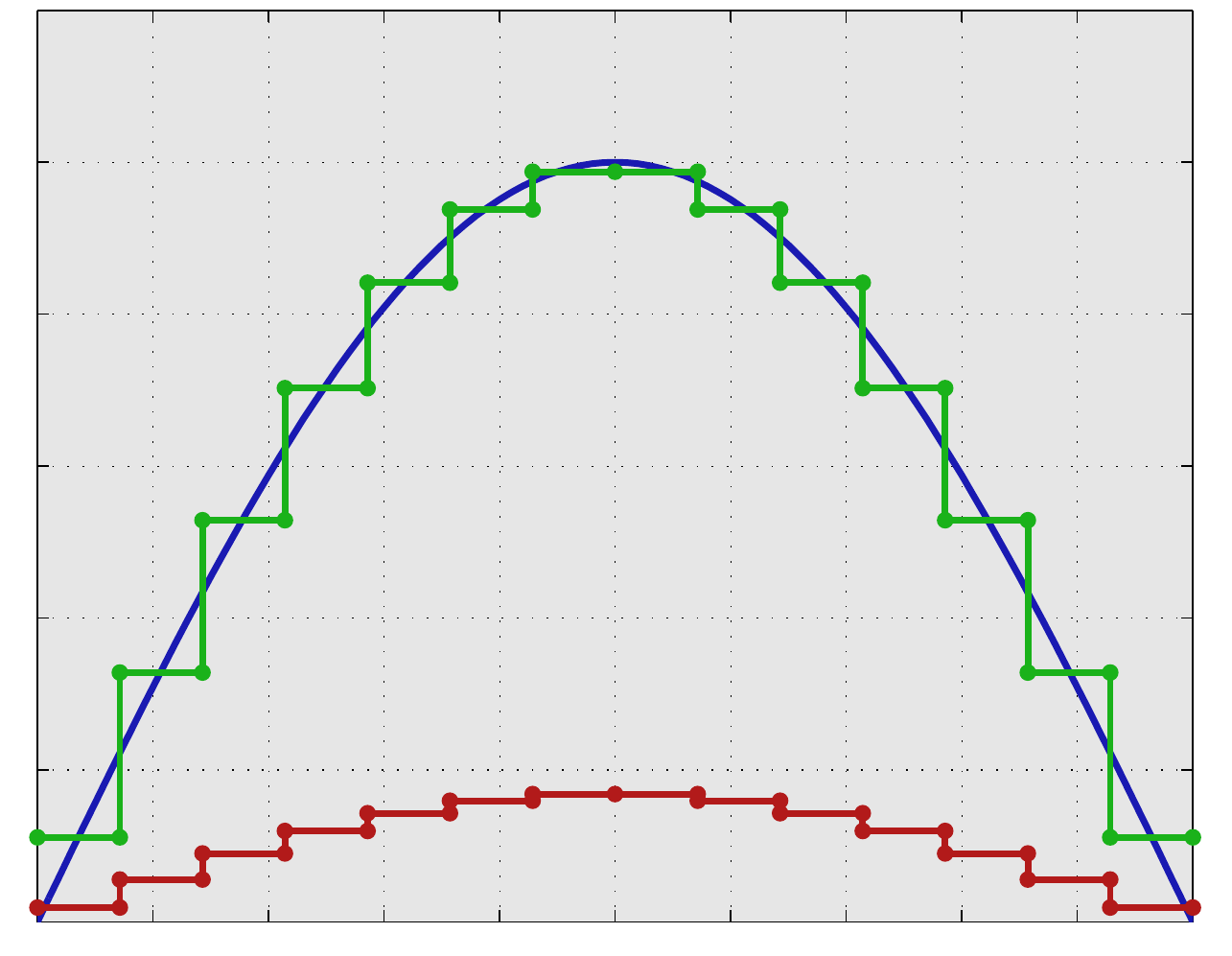}}\quad
\includegraphics[height=1.8in]{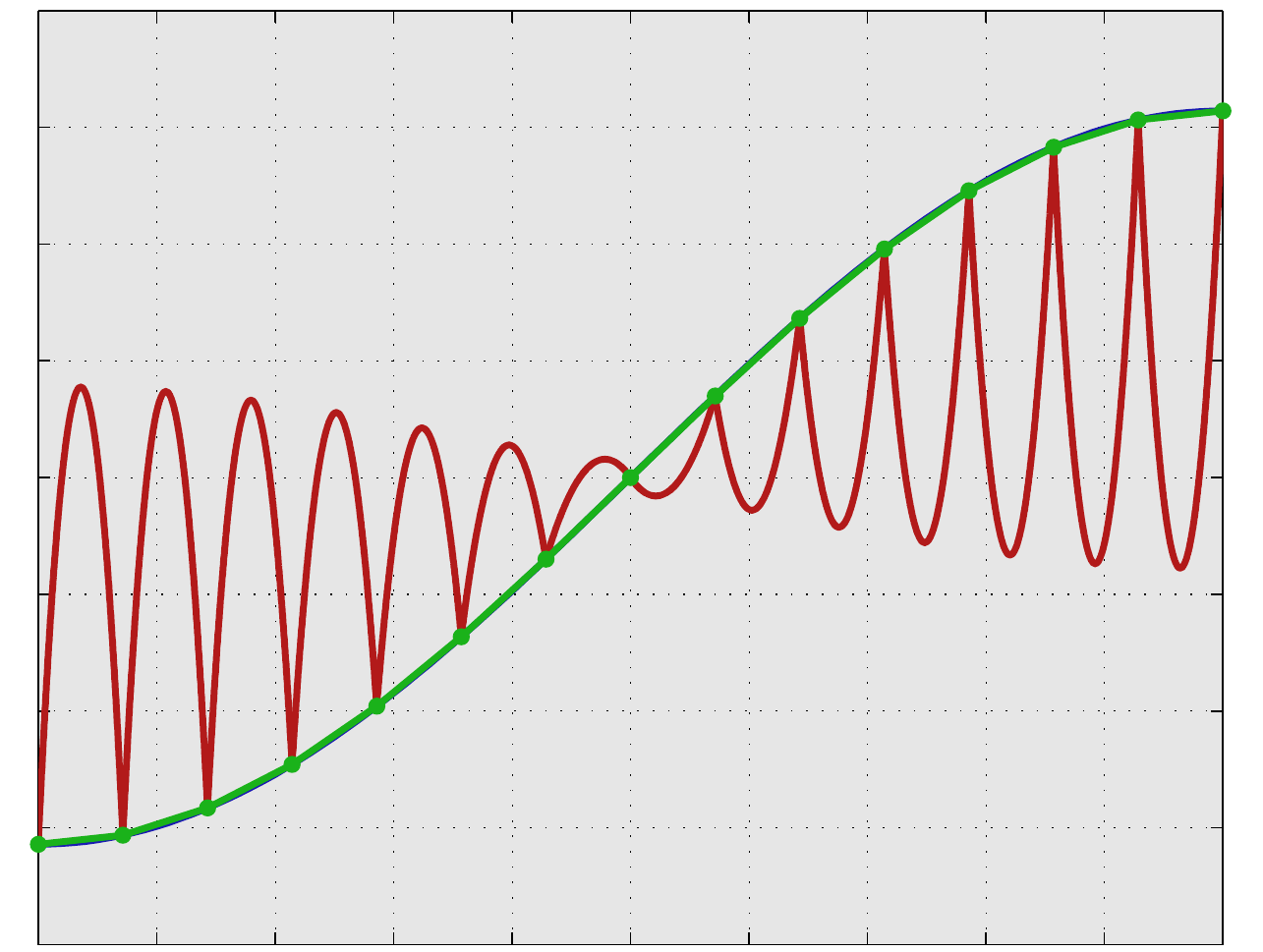}}
\caption[]{Approximation of the mixed formulation for 
$-u^{\prime\prime}=f$ in one dimension with two choices of elements, piecewise constants for
$u$ and piecewise linears for $\sigma$ (a stable method, shown in green), or
piecewise constants for $u$ and piecewise quadratics for $\sigma$ (unstable, shown in red).  The
left plot shows $u$ and the right plot shows $\sigma$, with the exact solution
in blue.  (In the right plot, the blue curve essentially coincides with the green
curve and hence is not visible.)}
\label{fg:f0}
\end{figure}

In one dimension, finding stable pairs of finite dimensional subspaces for the
mixed formulation of the two-point boundary value problem is easy.  For any
integer $r \ge 1$, the combination of continuous piecewise polynomials of
degree at most $r$ for $\sigma$ and arbitrary piecewise polynomials of degree
at most $r-1$ for $u$ is stable as can be verified via elementary means (and
which can be viewed as a very simple application
of the theory presented in this paper).  In higher
dimensions, the problem of finding stable combinations of elements is
considerably more complicated.  This is discussed in Section~\ref{subsubsec:mixedlap}
below.  In particular, we shall see that the choice of continuous
piecewise linear functions for $\sigma$ and piecewise constant functions for
$u$ is not stable in more than one dimension.  However stable element choices are known for this
problem and again may be viewed as a simple application of the finite element
exterior calculus developed in this paper.

\subsection{The contents of this paper}\label{subsec:abstract}
The brief introduction to the finite element method just given will be continued in
Section~\ref{sec:femethod}.  In particular, there we formalize the notions of
consistency and stability and establish their relation to convergence.  We shall also give
several more computational examples. While seemingly simple, some of these
examples may be surprising even to specialists, and they illustrate the difficulty
in obtaining good methods and the need for a theoretical framework in which
to understand such behaviors.

Like the theory of weak solutions of PDEs, the theory of finite element
methods is based on functional analysis and takes its most natural form in
a Hilbert space setting.  In Section~\ref{sec:HCA} of this paper, we
develop an abstract Hilbert space framework which
captures key elements of  Hodge theory, and can be used to explore
the stability of finite element methods.  The most basic object in
this framework is a cochain complex of Hilbert spaces, referred to as
a \emph{Hilbert complex}. Function spaces of such complexes will occur
in the weak formulations of the PDE problems we consider,
and the differentials will be differential operators entering into the
PDE problem. The most canonical example of a Hilbert complex is the
$L^2$ de~Rham complex of a Riemannian manifold, but it is a far more
general object with other important realizations.
For example, it allows
the definition of spaces of harmonic forms and the proof that they are
isomorphic to the cohomology groups. A Hilbert complex includes
enough structure to define an abstract Hodge Laplacian, defined from
a variational problem with a saddle point structure. However, for
these problems to be well--posed, we need the additional property of a
\emph{closed} Hilbert complex, i.e., that the range of the differentials
are closed.

In this framework, the finite element spaces used to compute
approximate solutions are represented by finite dimensional subspaces
of the spaces in the closed Hilbert complex.  We identify two key
properties of these subspaces: first, they should combine
to form a \emph{subcomplex} of the Hilbert complex, and, second, there should
exist a \emph{bounded cochain projection} from the Hilbert complex to this
subcomplex.  Under these hypotheses and a minor consistency condition,
it is easy to show that the subcomplex inherits the cohomology of the
true complex, i.e., that the cochain projections induce an isomorphism
from the space of harmonic forms to the space of discrete harmonic
forms, and to get an error estimate on the difference between a
harmonic form and its discrete counterpart.  In the applications, this
will be crucial for stable approximation of the PDEs.  In fact, a main
theme of finite element exterior calculus is that the same two
assumptions, the subcomplex property and the existence of a bounded
cochain projection, are the natural hypotheses to establish the
stability of the corresponding discrete Hodge Laplacian, defined by
the Galerkin method.

In Section~\ref{sec:deRham} we look in more depth at the canonical
example of the de~Rham complex for a bounded domain in Euclidean space,
beginning with a brief summary of exterior calculus.  We interpret
the de~Rham complex as a Hilbert complex and discuss the PDEs
most closely associated with it.  This brings us to the topic of
Section~\ref{sec:FEDF}, the construction of finite element de~Rham
subcomplexes, which is the heart of finite element exterior calculus and
the reason for its name.  In this section, we construct finite element
spaces of differential forms---piecewise polynomial spaces defined via a
simplicial decomposition and specification of shape functions and degrees
of freedom---which combine to form a subcomplex of the $L^2$ de~Rham
complex admitting a bounded cochain projection.  First we construct the
spaces of polynomial differential forms used for shape functions, relying
heavily on the Koszul complex and its properties, and then we construct
the degrees of freedom.  We next show that the resulting finite element
spaces can be efficiently implemented, have good approximation properties,
and can be combined into de~Rham subcomplexes.  Finally, we construct
bounded cochain projections, and, having verified the hypotheses of the
abstract theory, draw conclusions for the finite element approximation of
the Hodge Laplacian.

In the final two sections of the paper, we make other applications of
the abstract framework.  In the last section, we study a differential
complex we call the elasticity complex, which is quite different from the
de~Rham complex.  In particular, one of its differentials is a partial
differential operator of second order.  Via the finite element exterior
calculus of the elasticity complex, we have obtained the first stable
mixed finite elements using polynomial shape functions
for the equations of elasticity, with important
applications in solid mechanics.

\subsection{Antecedents and related approaches}\label{subsec:ant}
We now discuss some of the antecedents of finite element exterior
calculus and some related approaches.  While the first comprehensive
view of finite element exterior calculus, and the first use of that
phrase, was in the 2006 paper \cite{acta}, this was certainly
not the first intersection of finite element theory and Hodge theory.
In 1957, Whitney \cite{whitney} published his complex of Whitney forms,
which is, in our terminology, a finite element de~Rham subcomplex.
Whitney's goals were geometric.  For example, he used these forms in
a proof of de~Rham's theorem identifying the cohomology of a manifold
defined via differential forms (de~Rham cohomology) with that defined
via a triangulation and cochains (simplicial cohomology).  With the
benefit of hindsight, we may view this, at least in principle, as an early
application of finite elements to reduce the computation of a quantity of
interest defined via infinite dimensional function spaces and operators,
to a finite dimensional computation using piecewise polynomials on a
triangulation.  The computed quantities are the Betti numbers of the
manifold, i.e., the dimensions of the de~Rham cohomology spaces. For
these integer quantities, issues of approximation and convergence do
not play much of a role.  The situation is different in the 1976 work
of Dodziuk \cite{dodziuk} and Dodziuk and Patodi \cite{dodziuk-patodi},
who considered the approximation of the Hodge Laplacian on a Riemannian
manifold by a combinatorial Hodge Laplacian, a sort of finite difference
approximation defined on cochains with respect to a triangulation.
The combinatorial Hodge Laplacian was defined in \cite{dodziuk} using
the Whitney forms, thus realizing the finite difference operator as a
sort of finite element approximation. A key result in \cite{dodziuk}
was a proof of some convergence properties of the Whitney forms.
In \cite{dodziuk-patodi} the authors applied them to show that the
eigenvalues of the combinatorial Hodge Laplacian converge to those of
the true Hodge Laplacian.  This is indeed a finite element convergence
result, as the authors remark.  In 1978, M\"uller \cite{muller} further
developed this work and used it to prove the Ray--Singer conjecture.
This conjecture  equates a topological invariant defined in terms of
the Riemannian structure with one defined in terms of a triangulation,
and was the original goal of \cite{dodziuk,dodziuk-patodi}. (Cheeger
\cite{cheeger} gave a different, independent proof of the Ray--Singer
conjecture  at about the same time.) Other spaces of finite element
differential forms have appeared in geometry as well, especially
the differential graded algebra of piecewise polynomial forms on a
simplicial complex introduced by Sullivan \cite{sullivan73,sullivan77}.
Baker \cite{baker} calls these Sullivan--Whitney forms, and, in an early
paper bringing finite element analysis techniques to bear on geometry,
gives a numerical analysis of their accuracy for approximating the
eigenvalues of the Hodge Laplacian.

Independently of the work of the geometers, during the 1970s and 1980s
numerical analysts and computational engineers reinvented various
special cases of the Whitney forms and developed new variants of them
to use for the solution of partial differential equations on two and
three-dimensional domains.  In this work, naturally, implementational
issues, rates of convergence, and sharp estimates played a more prominent
role than in the geometry literature.  The pioneering paper of Raviart and
Thomas \cite{Raviart-Thomas}, presented at a finite element conference in
1975, proposed the first stable finite elements for solving the scalar
Laplacian in two dimensions using the mixed formulation.  The mixed
formulation involves two unknown fields, the scalar-valued solution, and
an additional vector-valued variable representing its gradient. Raviart
and Thomas
 proposed a family
of pairs of finite element spaces, one for each polynomial degree.
As was realized later, in the lowest degree case the space they
constructed for the vector-valued variable is just the space of Whitney
1-forms, while they used piecewise constants, which are Whitney 2-forms,
for the scalar variable.  For higher degrees, their elements are the
higher order Whitney forms.  In three-dimensions, the introduction of
Whitney 1- and 2-forms for finite element computations and their higher
degree analogues was made by N\'ed\'elec \cite{Nedelec1} in 1980, while
the polynomial mixed elements which can be viewed as Sullivan--Whitney
forms were introduced as finite elements by  Brezzi, Douglas, and Marini
\cite{Brezzi-Douglas-Marini}  in 1985 in two dimensions, and then  by
N\'ed\'elec \cite{Nedelec1} in 1986 in three dimensions.

In 1988 Bossavit, in a paper in the IEE Transactions on Magnetics
\cite{Bossavit}, made the connection between Whitney's forms used by
geometers and some of the mixed finite element spaces that had been
proposed for electromagnetics, inspired in part by Kotiuga's Ph.D.\
thesis in electrical engineering  \cite{kotiuga-thesis}. Maxwell's
equations are naturally formulated in terms of differential
forms, and the computational electromagnetics community
developed the connection between mixed finite elements and
Hodge theory in a number of directions.  See, in particular,
\cite{boffi,Demkowicz-Monk-Rachowicz-Vardapetyan,Hiptmair,hiptmair-pier,HiptmairActa,monk-book}.

The methods we derive here are examples of compatible discretization
methods, which means that at the discrete level they reproduce,
rather than merely approximate, certain essential structures of the
continuous problem.  Other examples of compatible discretization
methods for elliptic PDEs are mimetic finite difference methods
\cite{bochev-hyman,brezzi-lipnikov-shashkov} including covolume
methods \cite{nicolaides-trapp} and the discrete exterior calculus
\cite{desbrun-et-al}.  In these methods, the fundamental object used to
discretize a differential $k$-form is typically a simplicial cochain,
i.e., a number is assigned to each $k$-dimensional face of the mesh
representing the integral of the $k$-form over the face.  This is more
of a finite difference, rather than finite element, point of view,
recalling the early work of Dodziuk on combinatorial Hodge theory.
Since the space of $k$-dimensional simplicial cochains is isomorphic to
the space of Whitney $k$-forms, there
is a close relationship between these methods and the simplest methods of the finite element
exterior calculus.  In some simple cases, the methods even coincide.
In contrast to the finite element approach, these cochain-based approaches
do not naturally generalize to higher order methods.  Discretizations of exterior
calculus and Hodge theory have also been used for purposes other than solving partial
differential equations. For example, discrete forms which are
identical or closely related to cochains or the corresponding Whitney
forms play an important role in geometric modeling, parameterization, and computer graphics.
See for example \cite{fsdh07,ggt06,gu-yau,wwtds06}.

\subsection{Highlights of the finite element exterior calculus}\label{subsec:high}
We close this introduction by highlighting some of the features that are
unique or much more prominent in the finite element exterior calculus
than in earlier works.
\begin{itemize}
 \item We work in an abstract Hilbert space setting that captures the fundamental
structures of the Hodge theory of Riemannian manifolds, but applies more generally.
In fact, the paper proceeds in two parts, first the abstract theory for Hilbert complexes,
and then the application to the de~Rham complex and Hodge theory and other applications.
 \item Mixed formulations based on saddle point variational principles
play a prominent role.  In particular, the algorithms we use to
approximate the Hodge Laplacian are based on a mixed formulation, as is
the analysis of the algorithms.  This is in contrast to the approach
in the geometry literature, where the underlying variational principle
is a minimization principle.  In the case of the simplest elements,
the Whitney elements, the two methods are equivalent.  That is, the
discrete solution obtained by the mixed finite element method using
Whitney forms, is the same as obtained by Dodziuk's combinatorial
Laplacian.  However the different viewpoint leads naturally to
different approaches to the analysis.  The use of Whitney forms for
the mixed formulation is obviously a \emph{consistent} discretization,
and the key to the analysis is to establish \emph{stability} (see the next
section for the terminology).  However, for the minimization
principle, it is unclear whether Whitney forms provide a consistent
approximation, because they do not belong to the domain of the
exterior coderivative, and, as remarked in \cite{dodziuk-patodi}, this
greatly complicates the analysis.  The results we obtain are both more
easily proven and sharper.
\item
Our analysis is based on two main properties of the subspaces used to discretize
the Hilbert complex.  First, they can
be formed into subcomplexes, which is a key assumption in
much of the work we have discussed.  Second, there exist a bounded cochain projection
from the Hilbert complex to the subcomplex.  This is a new feature.  In previous work,
a cochain projection often played a major role, but it was not bounded, and
the existence of bounded cochain projections was not realized.  In fact, they exist
quite generally (see Theorem~\ref{converse}), and we review the construction for the
de~Rham complex in Section~\ref{subsec:bcp}.
\item
Since we are interested in actual numerical computations, it is important that our spaces
be efficiently implementable.  This is not true for all piecewise polynomial spaces.
As explained in the next section,
finite element spaces are a class of piecewise polynomial spaces that can be implemented
efficiently by local computations thanks to the existence of degrees of freedom,
and the construction of degrees of freedom and local bases is an important part of the
finite element exterior calculus.
\item
For the same reason, high order piecewise polynomials are of great importance, and all the
constructions and analysis of finite element exterior calculus carries through for
arbitrary polynomial degree.
\item
A prominent aspect of the finite element exterior calculus is the role of
two families of spaces of polynomial differentials form, $\P_r\Lambda^k$
and $\P_r^-\Lambda^k$, where the index $r\ge 1$ denotes the polynomial
degree and $k\ge 0$ the form degree.  These are the shape functions
for corresponding finite element spaces of differential $k$-forms
which include, as special cases, the Lagrange finite element family,
and most of the stable finite element spaces that have been used
to define mixed formulations of the Poisson or Maxwell's equations.
The space $\P_1^-\Lambda^k$ is the classical space of Whitney $k$-forms.
The finite element spaces based on $\P_r\Lambda^k$ are the spaces
of Sullivan--Whitney forms.
We show that for each polynomial degree $r$, there are $2^{n-1}$ ways to
form these spaces in de~Rham subcomplexes for a domain in $n$ dimensions.
The unified treatment of the spaces $\P_r\Lambda^k$
and $\P_r^-\Lambda^k$, particularly their connections via the Koszul complex, is new to the
finite element exterior calculus.
\end{itemize}
The finite element exterior calculus unifies many of the finite element
methods that have been developed for solving PDEs arising in fluid
and solid mechanics, electromagnetics, and other areas.  Consequently,
the methods developed here have been widely implemented and applied in
scientific and commercial programs such as GetDP \cite{getdp}, FEniCS
\cite{fenics}, EMSolve \cite{emsolve}, deal.II \cite{deal.ii}, Diffpack
\cite{diffpack}, Getfem++ \cite{getfem}, and NGSolve \cite{NGSolve}.
We also note that, as part of a recent programming effort connected with
the FEniCS project, Logg and Mardal \cite{logg-mardal} have implemented
the full set of finite element spaces developed in this paper strictly
following the finite element exterior framework as laid out here and
\cite{acta}.

\section{Finite element discretizations} \label{sec:femethod}
In this section we continue the introduction to the finite element
method begun above.  We move beyond the case of one dimension, and
consider not only the formulation of the method, but also its analysis.
To motivate the theory developed later in this paper, we present further
examples that illustrate how for some problems, even rather simple ones,
deriving accurate finite element methods is not a straightforward process.

\subsection{Galerkin methods and finite elements} \label{subsec:gal}
We consider first a simple problem, which \emph{can} be discretized in a straightforward
way, namely the Dirichlet problem for Poisson's equation in a polyhedral
domain $\Omega\subset\R^n$:
\begin{equation}\label{poisson}
- \Delta u = f \text{ in } \Omega, \quad u=0 \text{ on } \partial \Omega.
\end{equation}
This is the generalization to $n$ dimensions of the problem \eqref{1d} discussed
in the introduction, and the solution may again be characterized as the minimizer of an energy
functional analogous to \eqref{min} or as the solution of a weak problem analogous
to \eqref{wf}.  This leads to discretization just as for the one-dimensional case,
by choosing a trial space $V_h\subset \0 H^1(\Omega)$ and defining the
approximate solution $u_h\in V_h$ by Galerkin's method:
\begin{equation*}
\int_{\Omega} \grad u_h(x) \cdot \grad v(x) \, dx
= \int_{\Omega} f(x) v(x) \, dx, \quad  v \in V_h.
\end{equation*}

As in one dimension, the simplest finite element method is obtained by
using the trial space consisting of all piecewise linear functions
with respect to a given simplicial triangulation of the domain $\Omega$, which
are continuous and vanish on $\partial\Omega$.  A key to
the efficacy of this finite element method is the existence of a basis for the
trial space consisting of functions which are locally supported, i.e.,
vanish on all but a small number of
the elements of the triangulation. See Figure~\ref{fg:hatfn}.
Because of this, the coefficient matrix of the linear system is
easily computed and is sparse, and so the system can be solved
efficiently.
\begin{figure}[htb]
\centerline{\raise-.05in\hbox{\includegraphics[height=1.5in]{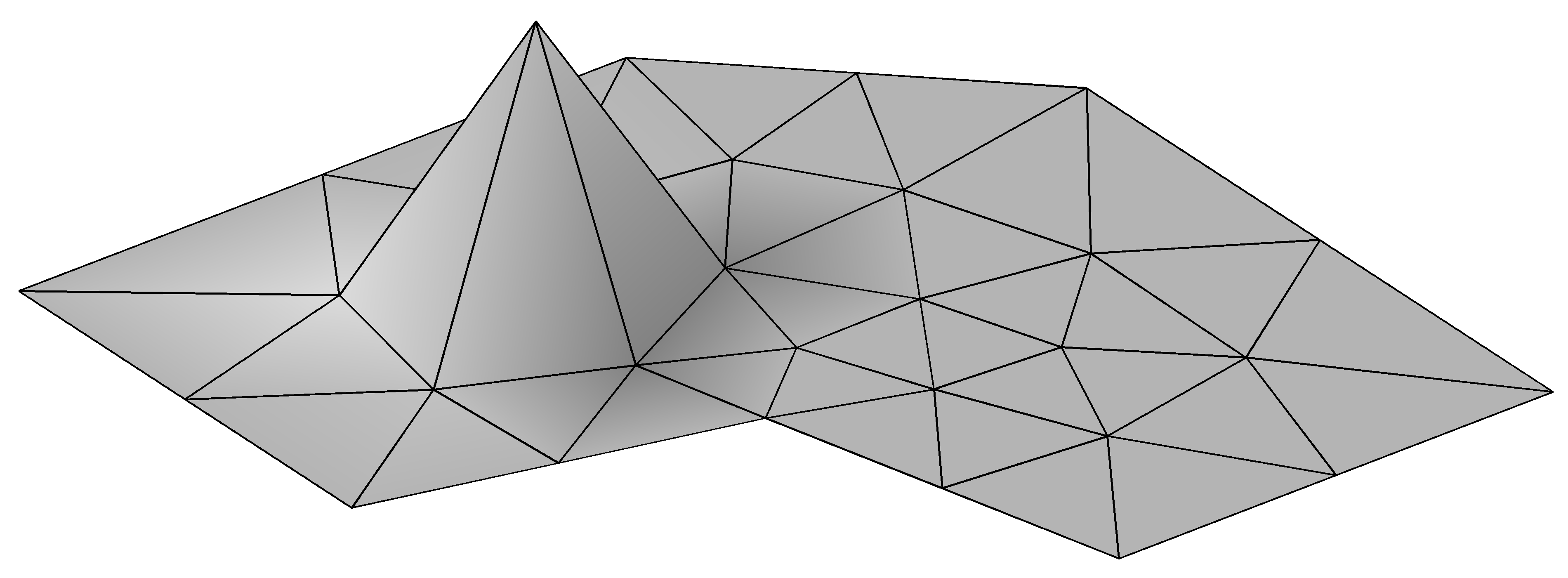}}}
\caption[]{A piecewise linear finite element basis function.}
\label{fg:hatfn}
\end{figure}

More generally, a finite element method is a Galerkin method for which the trial
space $V_h$ is a space of piecewise polynomial functions which can be obtained
by what is called the finite element assembly process.  This means that the space
can be defined by specifying  the triangulation $\T_h$ and, for each
element $T\in\T_h$, a space of polynomial functions on $T$ called the
\emph{shape functions}, and a set of \emph{degrees of freedom}.  By degrees of freedom on $T$,
we mean a set of functionals on the space of shape functions, which
can be assigned values arbitrarily to determine a unique shape function.  In other words, the
degrees of freedom form a basis for the dual space of the space of shape functions.
In the case of piecewise linear finite elements, the shape functions are of course
the linear polynomials on $T$, a space of dimension $n+1$, and the
degrees of freedom are the $n+1$ evaluation functionals $p\mapsto p(x)$, where $x$ varies over the
vertices of $T$.  For the finite element assembly process, we also
require that each degree of freedom be associated to a face of some dimension of the simplex $T$.
For example, in the case
of piecewise linear finite elements, the degree of freedom $p\mapsto p(x)$ is associated to
the vertex $x$.  Given the triangulation, shape functions, and degrees of freedom,
the finite element space $V_h$ is defined as the set
of functions on $\Omega$ (possibly multivalued on the element boundaries)
whose restriction to any $T\in\T_h$ belongs to the
given space of shape functions on $T$, and for which the degrees of freedom are
single-valued in the sense that when two elements share a common face, the
corresponding degrees of freedom take on the same value. Returning again to the
example of piecewise linear functions, $V_h$ is the set of functions which are linear
polynomials on each element, and which are single-valued at the vertices.
It is easy to see that this is precisely the space of continuous piecewise
linear functions, which is a subspace of $H^1(\Omega)$.
As another example, we could take the shape functions on $T$ to be the
polynomials of degree at most $2$, and take as degrees of freedom the functions
$p\mapsto p(x)$, $x$ a vertex of $T$, and $p\mapsto \int_e p\,ds$, $e$ an edge of $T$.
The resulting assembled finite element space is the space of all continuous piecewise
quadratics.
The finite element assembly process insures the existence of a computable locally
supported basis, which is
important for efficient implementation.

\subsection{Consistency, stability, and convergence} \label{subsec:erroranal}
We now turn to the important problem of analyzing the error in the finite element
method.
To understand when a Galerkin method will produce a good approximation to the true
solution, we introduce the standard abstract framework. Let $V$ be a Hilbert space,
$B:V \times V \to \R$ a bounded bilinear form, and $F:V \to \R$ a bounded
linear form.  We assume the problem to be solved can be stated in the form:
Find $u \in V$ such that
$$
B(u,v)=F(v), \quad v \in V.
$$
This problem is
called well-posed if for each $F \in V^*$, there exists a unique solution $u
\in V$ and the mapping $F \mapsto u$ is bounded, or, equivalently, if the operator
$L:V \to V^*$ given by $\<Lu,v\> = B(u,v)$ is an isomorphism.  For the
Dirichlet problem for Poisson's equation,
\begin{equation}\label{Bpoisson}
 V=\0H^1(\Omega),\quad B(u,v)=\int_{\Omega} \grad u(x) \cdot \grad v(x) \, dx,
\quad F(v)= \int_{\Omega} f(x) v(x) \, dx.
\end{equation}

A \emph{generalized Galerkin method} for the abstract problem begins with a finite-dimensional
normed space $V_h$ (not necessarily a subspace of $V$), a bilinear form
$B_h:V_h\x V_h\to R$, and a linear form $F_h:V_h\to\R$,
and defines $u_h \in V_h$ by
\begin{equation}
\label{genG}
B_h(u_h,v)=F_h(v), \quad v
\in V_h.
\end{equation}
A \emph{Galerkin method} is the special case of a generalized Galerkin
method for which $V_h$ is a subspace of $V$ and the forms $B_h$ and $F_h$ are
simply the restrictions of the forms $B$ and $F$ to the subspace.  The more
general setting is important since it allows the incorporation of additional approximations,
such as numerical integration to evaluate the integrals, and also allows for
situations in which $V_h$ is not a subspace of $V$.  Although
we do not treat approximations such as numerical integration in this paper,
for the fundamental discretization
method we study, namely the mixed method for the abstract Hodge Laplacian
introduced in Section~\ref{subsec:well-p-h-hc}, the trial space
$V_h$ is not a subspace of $V$, since it involves discrete harmonic forms
which will not, in general, belong to the space of harmonic forms.

The generalized Galerkin method \eqref{genG} may be written $L_h u_h=F_h$
where $L_h:V_h\to V_h^*$ is given by $\<L_hu,v\> = B_h(u,v)$, $u,v\in V_h$.
If the finite-dimensional problem is nonsingular, then we define the norm of
the discrete solution operator, $\|L_h^{-1}\|_{\Lin(V_h^*,V_h)}$, as the
\emph{stability constant} of the method.

Of course, in approximating the original problem determined by $V$, $B$, and $F$, by the
generalized Galerkin method given by $V_h$, $B_h$, and $F_h$,
we intend that the space $V_h$ in some sense approximates $V$ and
that the discrete forms $B_h$ and $F_h$ in some sense approximate $B$ and $F$.
This is the essence of \emph{consistency}.
Our goal is to prove that the discrete solution $u_h$ approximates $u$ in an appropriate
sense (\emph{convergence}).  In order to make these notions precise, we need to compare
a function in $V$ to a function in $V_h$.  To this end, we suppose that there is
a \emph{restriction operator} $\pi_h:V\to V_h$, so that $\pi_hu$ is thought to be
close to $u$.  Then the \emph{consistency error} is simply
$L_h\pi_hu-F_h$ and the error in the generalized Galerkin method which we wish
to control is $\pi_h u -u_h$.  We immediately get a relation between
the error and the consistency error
\begin{equation*}
 \pi_h u -u_h = L_h^{-1}(L_h\pi_hu-F_h),
\end{equation*}
and so the norm of the error is bounded by the product of the stability constant and
the norm of the consistency error:
\begin{equation*}
 \|\pi_h u -u_h\|_{V_h} \le  \|L_h^{-1}\|_{\Lin(V_h^*,V_h)}\|L_h\pi_hu-F_h\|_{V^*_h}.
\end{equation*}
Stated in terms of the bilinear form $B_h$, the norm of the consistency error can be
written
\begin{equation*}
\|L_h\pi_hu-F_h\|_{V^*_h}=\sup_{0\ne v\in V_h}\frac{B_h(\pi_hu,v)-F_h(v)}{\|v\|_{V_h}}.
\end{equation*}
As for stability, the finite dimensional problem is nonsingular if and only if
\begin{equation*}
\gamma_h  := \inf_{0\ne u \in V_h} \sup_{0\ne v \in V_h}
\frac{B_h(u,v)}{\|u\|_{V_h} \|v\|_{V_h}} > 0,
\end{equation*}
and the stability constant is then given by $\gamma_h^{-1}$.

Often we consider a sequence of spaces $V_h$ and forms $B_h$ and $F_h$
where we think of $h>0$ as an index accumulating at $0$.  The corresponding
generalized Galerkin method is consistent if the $V_h$ norm of the consistency
error tends to zero with $h$ and it is stable if the stability constant $\gamma_h^{-1}$
is uniformly bounded.  For a consistent, stable generalized Galerkin method,
$\|\pi_hu-u_h\|_{V_h}$ tends to zero, i.e., the method is convergent.

In the special case of a Galerkin
 method, we can bound the consistency
error
\begin{equation*}
\sup_{0\ne v\in V_h}\frac{B_h(\pi_hu,v)-F_h(v)}{\|v\|_V}
=\sup_{0\ne v\in V_h}\frac{B(\pi_hu-u,v)}{\|v\|_V}\le\|B\|\|\pi_hu-u\|_V.
\end{equation*}
In this case it is natural to choose the restriction $\pi_h$ to be the orthogonal projection onto
$V_h$, and so the consistency error is bounded by the norm of the bilinear form times the
error in the best approximation of the solution.  Thus we obtain
\begin{equation*}
 \|\pi_h u -u_h\|_V \le  \gamma_h^{-1}\|B\| \inf_{v\in V_h}\| u - v \|_V.
\end{equation*}
Combining this with the triangle inequality, we obtain
the basic error estimate for Galerkin methods
\begin{equation}
\label{basicerr}
\|u - u_h\|_V \le (1+ \gamma_h^{-1}\|B\|) \inf_{v\in V_h}\| u - v \|_V.
\end{equation}
(In fact, in this Hilbert space setting, the quantity in parentheses can be
replaced with $\gamma_h^{-1} \|B\|$, see~\cite{xu-zikatanov}.)
Note that a Galerkin method is consistent as long as the sequence of
subspaces $V_h$ is \emph{approximating} in $V$
in the sense that
\begin{equation}\label{approximating}
 \lim_{h\to 0}\inf_{v\in V_h}\|u - v\|_V =0 ,\quad u\in V.
\end{equation}
A consistent, stable Galerkin method converges, and the
approximation given by the method is quasioptimal, i.e., up to multiplication
by a constant, it is as good as the best approximation in the subspace.

In practice, it can be quite difficult to show that the finite dimensional problem is nonsingular
and to bound the stability constant, but there is one important case in which it is easy.
Namely when
the form $B$ is coercive, i.e., there is a positive constant $\alpha$ for which
\begin{equation*}
B(v,v) \ge \alpha \|v\|_V^2, \quad  v \in V,
\end{equation*}
and so $\gamma_h \ge\alpha$.  The bilinear form \eqref{Bpoisson} for Poisson's equation
is coercive, as follows from Poincar\'e's inequality.  This explains, and can be used
to prove, the good convergence behavior of the method depicted in Figure~\ref{fg:simplest}.

\subsection{Computational examples} \label{subsec:examples}
\subsubsection{Mixed formulation of the Laplacian} \label{subsubsec:mixedlap}
For an example of a problem that fits in the standard abstract framework with a noncoercive
bilinear form,  we take the mixed formulation of the Dirichlet problem for Poisson's
equation, already introduced in one dimension in Section~\ref{subsec:fem}.
Just as there, we begin by writing Poisson's equation
as the first order system
\begin{equation}\label{mixedpoisson}
\sigma = -\grad u, \quad  \div \sigma = f.
\end{equation}
The pair $(\sigma, u)$ can again be characterized variationally as the unique
critical point (a saddle point) of the functional
\begin{equation*}
I(\sigma, u) = \int_{\Omega} (\frac{1}{2} \sigma \cdot \sigma
- u \div \sigma) \, dx + \int_{\Omega} f u \, dx
\end{equation*}
over $H(\div; \Omega) \times L^2(\Omega)$ where
$H(\div; \Omega) = \{\sigma \in L^2(\Omega): \div \sigma \in L^2(\Omega)\}$.
Equivalently, it solves the weak problem: Find $\sigma \in H(\div;\Omega),
u \in L^2(\Omega)$ satisfying
\begin{gather*}
\int_{\Omega} \sigma \cdot \tau \, dx - \int_{\Omega} u \div \tau \, dx
= 0, \quad \tau \in H(\div;\Omega),
\\
\int_{\Omega} \div \sigma v \, dx =  \int_{\Omega} f v \, dx,
\quad v \in L^2(\Omega).
\end{gather*}

This mixed formulation of Poisson's equation fits in the abstract framework
if we define $V = H(\div; \Omega) \times L^2(\Omega)$,
\begin{equation*}
B(\sigma,u;\tau,v) = \int_{\Omega} \sigma \cdot \tau \, dx 
- \int_{\Omega} u \div \tau \, dx + \int_{\Omega} \div \sigma v \, dx,
\quad F(\tau,v) = \int_{\Omega} f v \, dx.
\end{equation*}
In this case the bilinear form $B$ is \emph{not} coercive, and so the choice
of subspaces and the analysis
is not so simple as for the standard finite element method for
Poisson's equation, a point we already illustrated in the one-dimensional case.

Finite element discretizations based on such saddle
point variational principles are called \emph{mixed finite element methods}.  Thus
a mixed finite element for Poisson's equation
is obtained by choosing subspaces $\Sigma_h \subset H(\div; \Omega)$ and
$V_h \subset L^2(\Omega)$ and seeking a critical point of $I$ over $\Sigma_h \times
V_h$.  The resulting Galerkin method has the form: 
Find $\sigma_h \in \Sigma_h, u_h \in V_h$ satisfying
\begin{gather*}
\int_{\Omega} \sigma_h \cdot \tau \, dx - \int_{\Omega} u_h \div \tau \, dx
= 0, \quad \tau \in \Sigma_h,
\quad
\int_{\Omega} \div \sigma_h v \, dx = \int_{\Omega} f v \, dx,
\quad v \in V_h.
\end{gather*}
This again reduces to a linear system of algebraic equations.

Since the bilinear form is not coercive, it is not automatic
that the linear system is nonsingular, i.e., that for $f=0$, the only
solution is $\sigma_h=0$, $u_h=0$.  Choosing $\tau = \sigma_h$ and $v =
u_h$ and adding the discretized variational equations, it follows immediately
that when $f=0$, $\sigma_h =0$.  However, $u_h$ need not vanish
unless the condition that
$\int_{\Omega} u_h \div \tau \, dx =0$ for all $\tau \in \Sigma_h$ implies that $u_h=0$.
In particular, this requires that $\dim(\div \Sigma_h)\ge\dim V_h$.
Thus, even nonsingularity of the approximate
problem depends on a relationship between the two finite dimensional spaces.
Even if the linear system is nonsingular, there remains the issue of stability,
i.e., a uniform bound on the inverse operator.

As mentioned earlier, the combination of continous piecewise linear elements
for $\sigma$ and piecewise constants for $u$ is not stable in two dimensions.
The simplest stable elements use the piecewise constants for $u$, and the
lowest order \emph{Raviart-Thomas elements} for $\sigma$.  These are finite elements
defined with respect to a triangular mesh by shape functions of the form $(a+
bx_1, c+ bx_2)$ and one degree of freedom for each edge $e$, namely
$\sigma\mapsto\int_e\sigma\cdot n\,ds$.  We show in Fig.~\ref{fg:f1} below two
numerical computations that demonstrate the difference between an unstable and
stable choice of elements for this problem.  The stable method accurately approximates the true
solution $u = x(1-x)y(1-y)$ on $(0,1) \times (0,1)$ with a piecewise constant,
while the unstable method is wildly oscillatory.
\begin{figure}[htb]
\centerline{\includegraphics[width=2.5in]{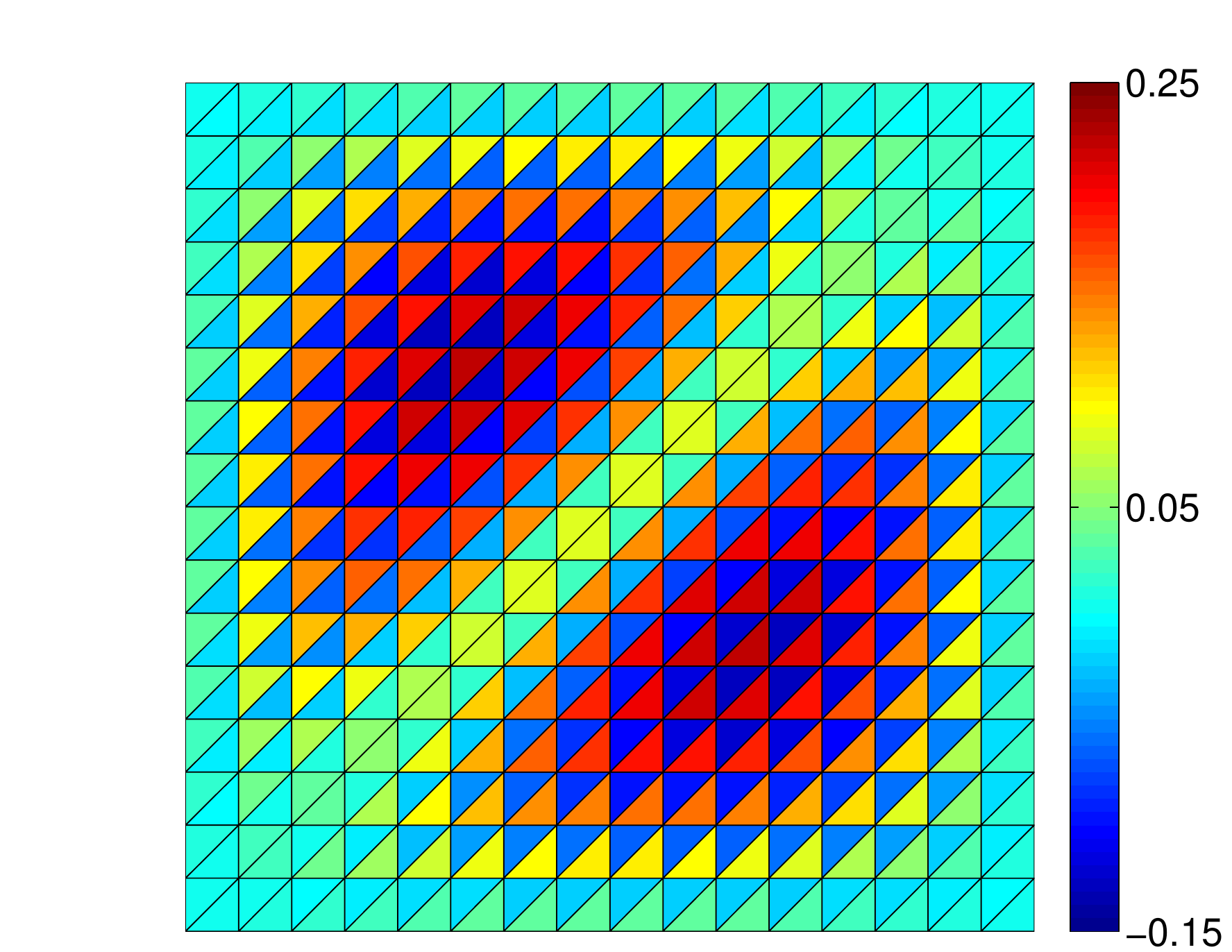}%
\includegraphics[width=2.5in]{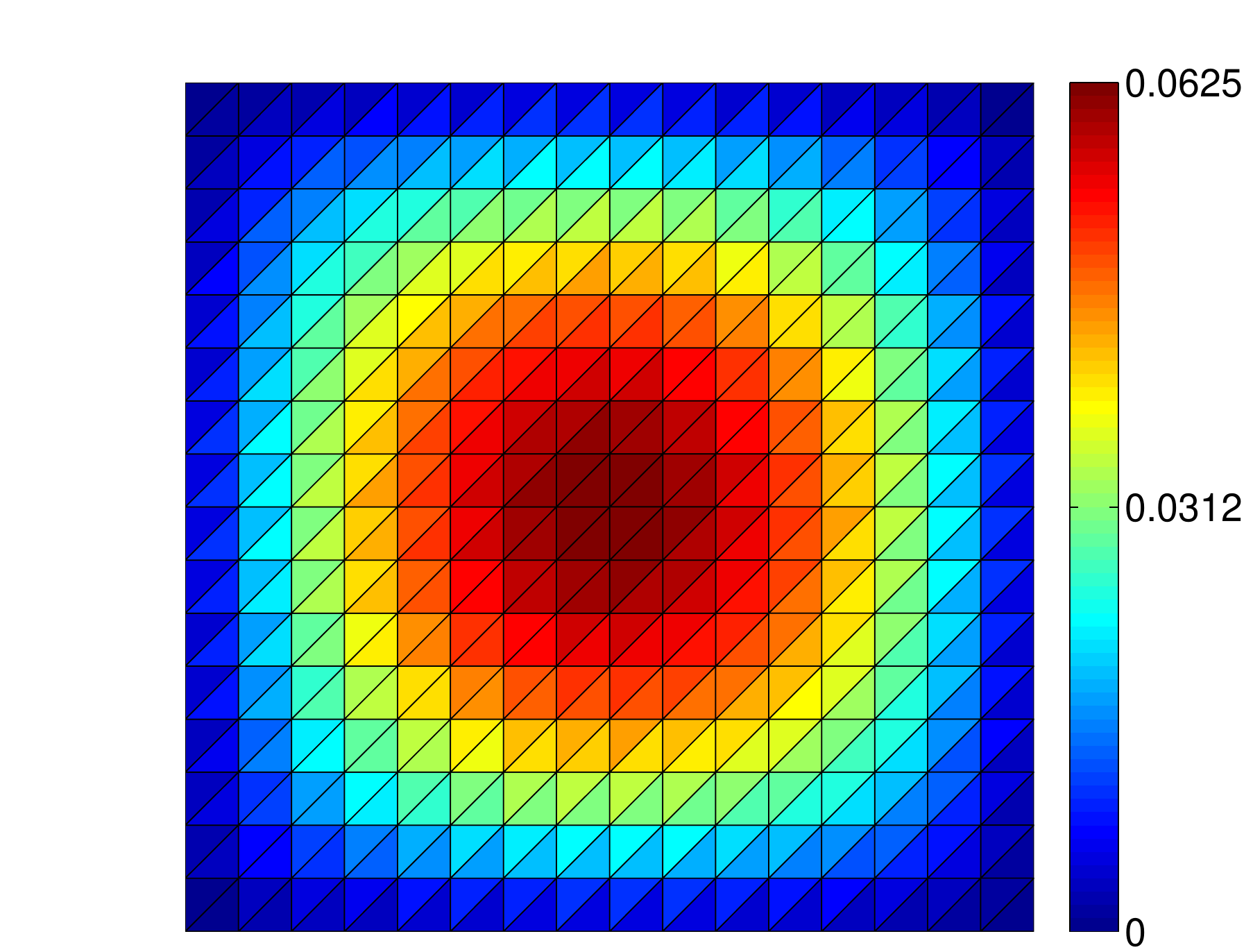}}
\caption[]{Approximation of the mixed formulation for
Poisson's equation using piecewise constants for $u$ and for $\sigma$
either continuous piecewise linears (left), or Raviart--Thomas elements (right).
The plotted quantity is $u$ in each case.}
\label{fg:f1}
\end{figure}

This problem is a special case of the Hodge Laplacian with
$k=n$ as discussed briefly in Section~\ref{subsec:deRcHc}; see especially
Section~\ref{subsubsec:hlk=3}.  The error analysis for a variety of finite element
methods for this problem, including the Raviart--Thomas elements, is thus a special
case of the general theory of this paper, yielding
the error estimates in Section~\ref{subsec:ahl}.

\subsubsection{The vector Laplacian on a nonconvex polygon} \label{subsubsec:l-shape}
Given the subtlety of finding stable pairs of finite element spaces
for the mixed variational formulation of Poisson's equation, we might
choose to avoid this formulation, in favor of the standard
formulation, which leads to a coercive bilinear form.  However, while
the standard formulation is easy to discretize for Poisson's equation,
additional issues arise already if we try to discretize the vector
Poisson equation. For a domain $\Omega$ in $\R^3$ with unit outward
normal $n$, this is the problem
\begin{equation}\label{vecpoissonbvp}
- \grad \div u + \curl \curl u = f, \text{ in } \Omega,
\quad u \cdot n =0, \quad
\curl u \x n =0, \quad \text{ on } \partial \Omega.
\end{equation}
The solution of this problem can again be characterized as the minimizer
of an appropriate energy functional,
\begin{equation}\label{Jvec}
J(u)= \frac{1}{2} \int_{\Omega} (|\div u|^2 + |\curl u|^2) \, dx
- \int_{\Omega} f \cdot u \, dx,
\end{equation}
but this time over the space $H(\curl; \Omega) \cap \0H(\div; \Omega)$, where
$H(\curl; \Omega) = \{u \in L^2(\Omega) \, | \, \curl u \in L^2(\Omega)\}$ and
$\0H(\div; \Omega) = \{u \in H(\div; \Omega) \, | \, u \cdot n =0 \
\text{on} \ \partial \Omega\}$ with $H(\div; \Omega)$ defined above.  This
problem is associated to a coercive bilinear form, but a standard finite element method based on a trial
subspace of the energy space $H(\curl; \Omega) \cap \0H(\div; \Omega)$, e.g., using continuous
piecewise linear vector functions, is very problematic.
In fact, as we shall illustrate shortly,
if the domain $\Omega$ is a nonconvex polyhedron, \emph{for almost
all $f$ the Galerkin method solution will converge to a function that is
not the true solution of the problem!}
The essence of this unfortunate situation is that any
piecewise polynomial subspace of $H(\curl; \Omega) \cap \0H(\div; \Omega)$ is
a subspace of  $H^1(\Omega) \cap \0H(\div; \Omega)$, and this space is
a closed subspace of $H(\curl; \Omega) \cap \0H(\div; \Omega)$.
For a nonconvex polyhedron, it is a proper closed subspace and generally the true solution will not
belong to it, due to a singularity at the reentrant corner.  Thus the method, while stable,
is inconsistent.  For more on
this example, see \cite{costabel}.

An accurate approximation of the vector Poisson equation can be obtained from a mixed
finite element formulation, based on the system:
\begin{equation*}
\sigma = -\div u, \quad  \grad\sigma  + \curl\curl u  = f \text{ in }
\Omega, \quad u \cdot n =0, \quad \curl u\x n =0 \text{ on }
\partial \Omega.
\end{equation*}
 Writing this system in weak form, we obtain the mixed formulation of the problem:
Find $\sigma \in H^1(\Omega)$, $u \in H(\curl;\Omega)$ satisfying
\begin{gather*}
\int_{\Omega} \sigma \tau \, dx - \int_{\Omega} u \cdot \grad \tau \, dx
=0, \quad \tau \in H^1(\Omega),
\\
\int_{\Omega} \grad \sigma \cdot v \, dx + \int_{\Omega} \curl u \cdot\curl v
\, dx = \int_{\Omega} f \cdot v \, dx, \quad  v 
\in H(\curl; \Omega).
\end{gather*}
In contrast
to a finite element method based on minimizing the energy \eqref{Jvec}, a
finite element approximation based on the mixed formulation
uses separate trial subspaces of $H^1(\Omega)$ and $H(\curl;\Omega)$, rather
than a single subspace of the intersection $H(\curl; \Omega) \cap \0H(\div; \Omega)$.

We now illustrate the nonconvergence of a Galerkin method based on energy minimization
and the convergence of one based on the mixed formulation, via computations in two space
dimensions (so now the curl of a vector $u$ is the scalar $\partial u_2/\partial x_1-
\partial u_1/\partial x_2$).  For the trial subspaces we make the simplest choices:
for the former method
we use continuous piecewise linear functions and for the mixed method we use
continuous piecewise linear functions to approximate $\sigma \in H^1(\Omega)$
and a variant of the lowest order Raviart--Thomas elements,
for which the shape functions are the infinitesimal rigid motions
$(a-bx_2,c+bx_1)$ and the degrees of
freedom are the tangential moments $u\mapsto \int_e u\cdot s\,ds$ for $e$ an edge.
The discrete solutions obtained
by the two methods for the problem when $f = (-1,0)$ are shown in Figure~\ref{fg:f2}.  As we shall show
later in this paper, the mixed formulation
gives an approximation that provably converges to the true solution,
while, as can be seen from comparing the two plots, the first approximation
scheme gives a completely different (and therefore inaccurate) result.

\begin{figure}[htb]
\centerline{\includegraphics[width=2in]{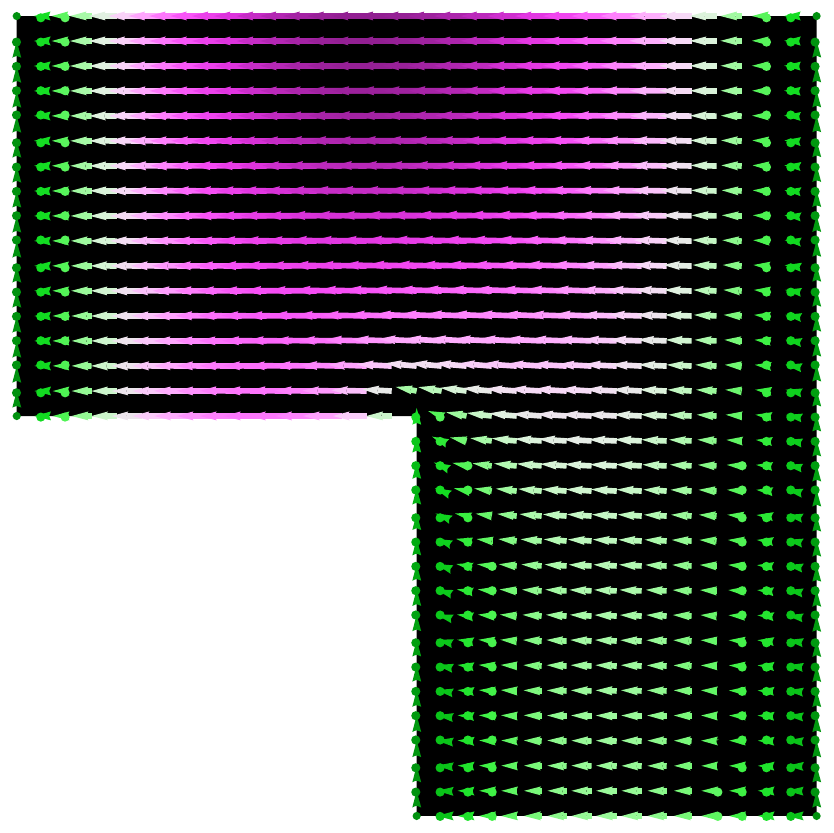} \qquad
\includegraphics[width=2in]{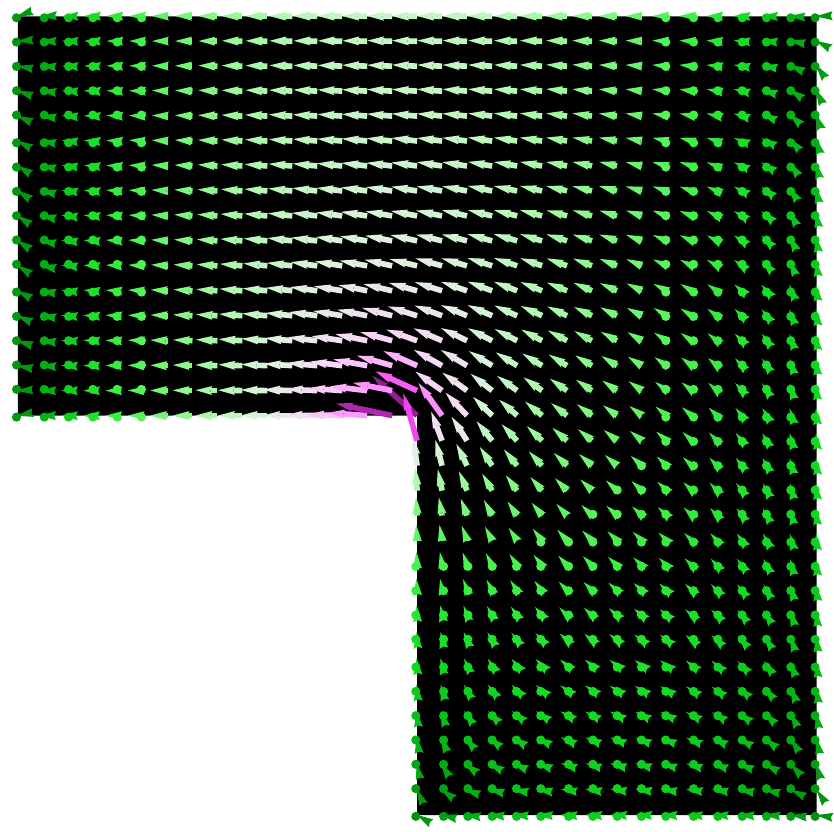}}
\caption[]{Approximation of the vector Laplacian by the standard finite
element method (left) and a mixed finite
element method (right).  The former method totally misses the singular
behavior of the solution near the reentrant corner.}
\label{fg:f2}
\end{figure}

This problem is again a special case of the Hodge Laplacian, now with $k=1$.  See
Section~\ref{subsubsec:hlk=1}.  The error analysis thus falls within the theory
of this paper, yielding estimates as in Section~\ref{subsec:ahl}.

\subsubsection{The vector Laplacian on an annulus} \label{subsubsec:annulus}
In the example just considered, the failure of a standard Galerkin method
based on energy minimization to solve the vector Poisson equation
was related to the reentrant corner of the domain and the resulting
singular behavior of the solution.  A quite different failure mode
for this method occurs if we take a domain which is smoothly bounded,
but not simply connected, e.g., an annulus.  In that case, as discussed
below in Section~\ref{subsec:abstract-hl}, the boundary value problem
\eqref{vecpoissonbvp} is not well-posed except for special values of
the forcing function $f$.  In order to obtain a well-posed problem,
the differential equation should be solved only modulo the space of
harmonic vector fields (or harmonic 1-forms)---which is a one-dimensional
space for the annulus---and the solution should be rendered unique by
enforcing orthogonality to the harmonic vector fields.  If we choose
the annulus with radii $1/2$ and $1$, and forcing function $f=(0,x)$,
the resulting solution, which can be computed accurately with a mixed
formulation falling within the theory of this paper, is displayed on the
right in Figure~\ref{fg:ann}.  However, the standard Galerkin method does
not capture the non-uniqueness and computes the discrete solution shown
on the left of the same figure, which is dominated by an approximation of
the harmonic vector field, and so is nothing like the true solution.

\begin{figure}[ht]
\centerline{\includegraphics[width=2in,viewport = 100 200 880 1000]{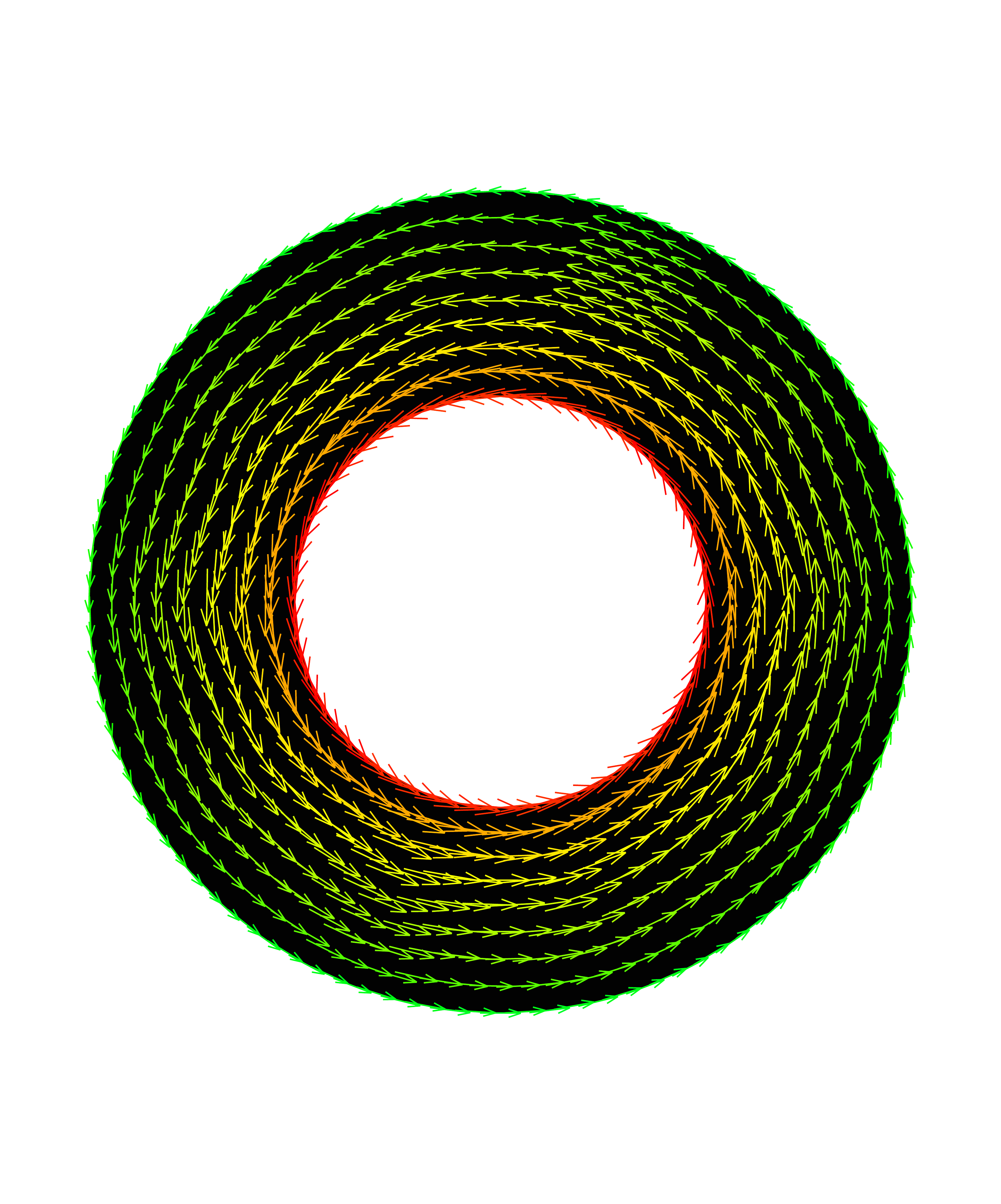}\qquad
\includegraphics[width=2in,viewport = 100 200 880 1000]{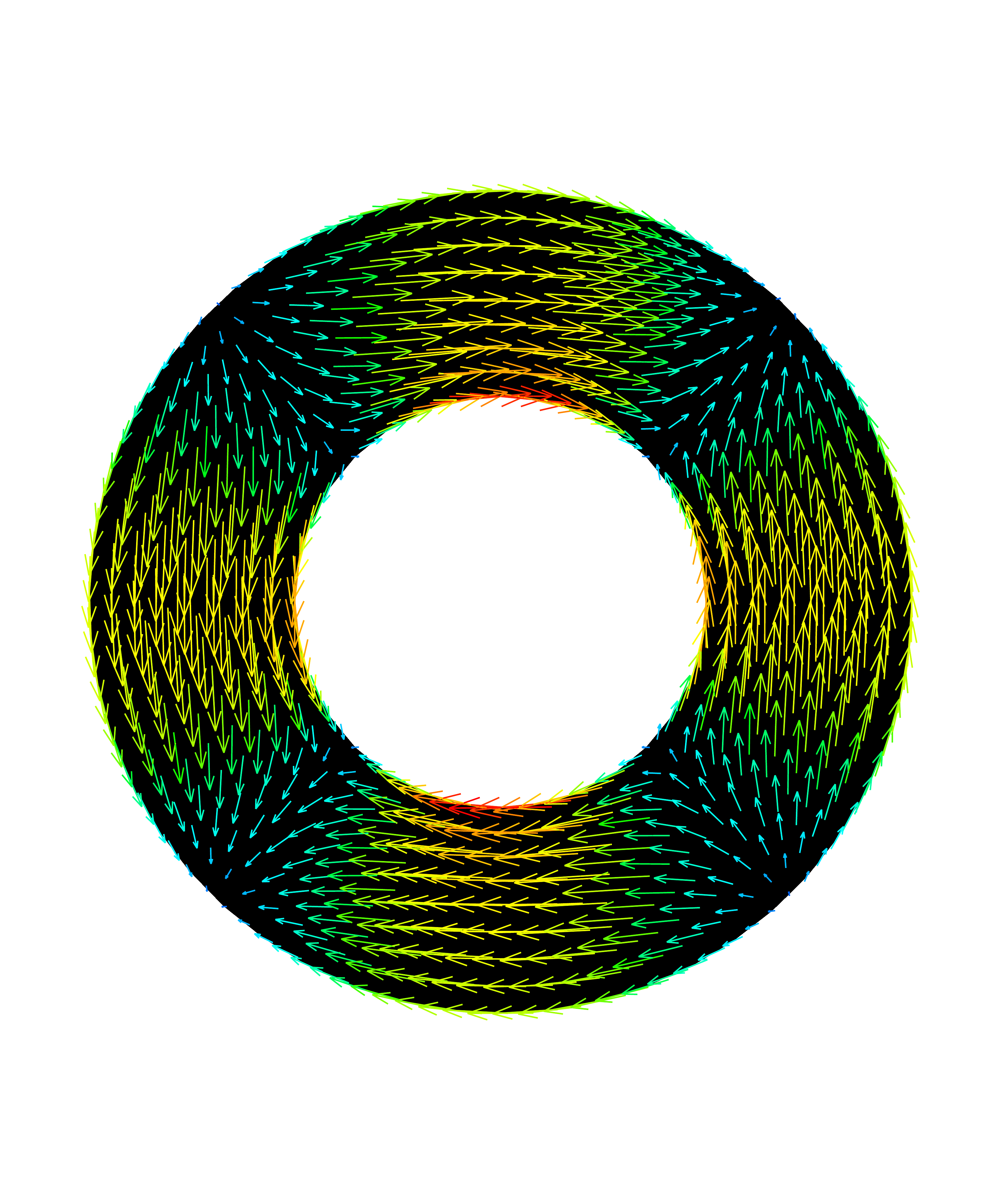}}
\caption[]{Approximation of the vector Laplacian on an annulus.  The true solution
shown here on the right in an (accurate) approximation by a mixed method.  It
is orthogonal to the harmonic fields, and satisfies the differential equation only
modulo harmonic fields.   The standard Galerkin solution using continuous piecewise
linear vector fields shown on the left, is totally different.}
\label{fg:ann}
\end{figure}

\subsubsection{The Maxwell eigenvalue problem} \label{subsubsec:maxwelleig}
Another situation where a standard finite element method gives unacceptable results,
but a mixed method succeeds, arises in the approximation of
elliptic eigenvalue problems related to the vector Laplacian or Maxwell's
equation.  This will be analyzed in detail later in this paper, and here we
only present a simple but striking computational example.
Consider the eigenvalue problem for the vector Laplacian discussed above,
which we write in mixed form as: Find nonzero $(\sigma,u) \in H^1(\Omega)\x H(\curl;\Omega)$
and $\lambda \in \R$ satisfying
\begin{equation}\label{mixedev}
 \begin{gathered}
\int_{\Omega} \sigma \cdot \tau \, dx - \int_{\Omega} \grad \tau\cdot u  \, dx
=0, \quad \tau \in H^1(\Omega).
\\
\int_{\Omega} \grad \sigma \cdot v \, dx + \int_{\Omega} \curl u \cdot\curl v
\, dx = \lambda \int_{\Omega} u \cdot v \, dx, \quad  v 
\in H(\curl; \Omega).
\end{gathered}
\end{equation}
As explained in Section~\ref{related-eigenv}, this problem can be
split into two subproblems.  In particular, if  $0\ne u \in H(\curl;\Omega)$
and $\lambda\in\R$ solves the eigenvalue problem
\begin{equation}\label{curlcurlev}
\int_{\Omega} \curl u \cdot\curl v
\, dx = \lambda \int_{\Omega} u \cdot v \, dx, \quad  v 
\in H(\curl; \Omega),
\end{equation}
and $\lambda$ is not equal to zero, then $(\sigma,u)$, $\lambda$ is an eigenpair
for \eqref{mixedev} with $\sigma=0$.

We now consider the solution of the eigenvalue problem \eqref{curlcurlev}, with two
different choices of trial subspaces in $H(\curl;\Omega)$.  Again, to make our point it is enough
to consider a two dimensional case, and we consider the solution of \eqref{curlcurlev} with
$\Omega$ a square of side length $\pi$.  For this domain, the positive
eigenvalues can be computed by separation of variables.
They are of the form $m^2+n^2$ with $m$ and $n$ integers: $1,1,2,4,4,5,5,8, \ldots$.
If we approximate \eqref{curlcurlev} using the space of continuous piecewise linear
vector fields as the trial subspace of $H(\curl;\Omega)$, the approximation fails badly.
This is shown for an unstructured mesh in Figure~\ref{fg:f3} and for a structured crisscross mesh
in Figure~\ref{fg:f4}, where the nonzero discrete eigenvalues are plotted.
Note the very different mode of failure for the two mesh types.
For more discussion of the spurious eigenvalues
arising using continuous piecewise linear vector fields on a crisscross mesh
see \cite{boffi-brezzi-gastaldi}.
By contrast, if we use the lowest order Raviart--Thomas approximation of $H(\curl;\Omega)$,
as shown on the right of Figure~\ref{fg:f3},
we obtain a provably good approximation for any mesh.  This is a very simple case of the general
eigenvalue approximation theory presented in
Section~\ref{subsec:eigenv} below.

\begin{figure}[htb]
\centerline{\includegraphics[height=1.3in]{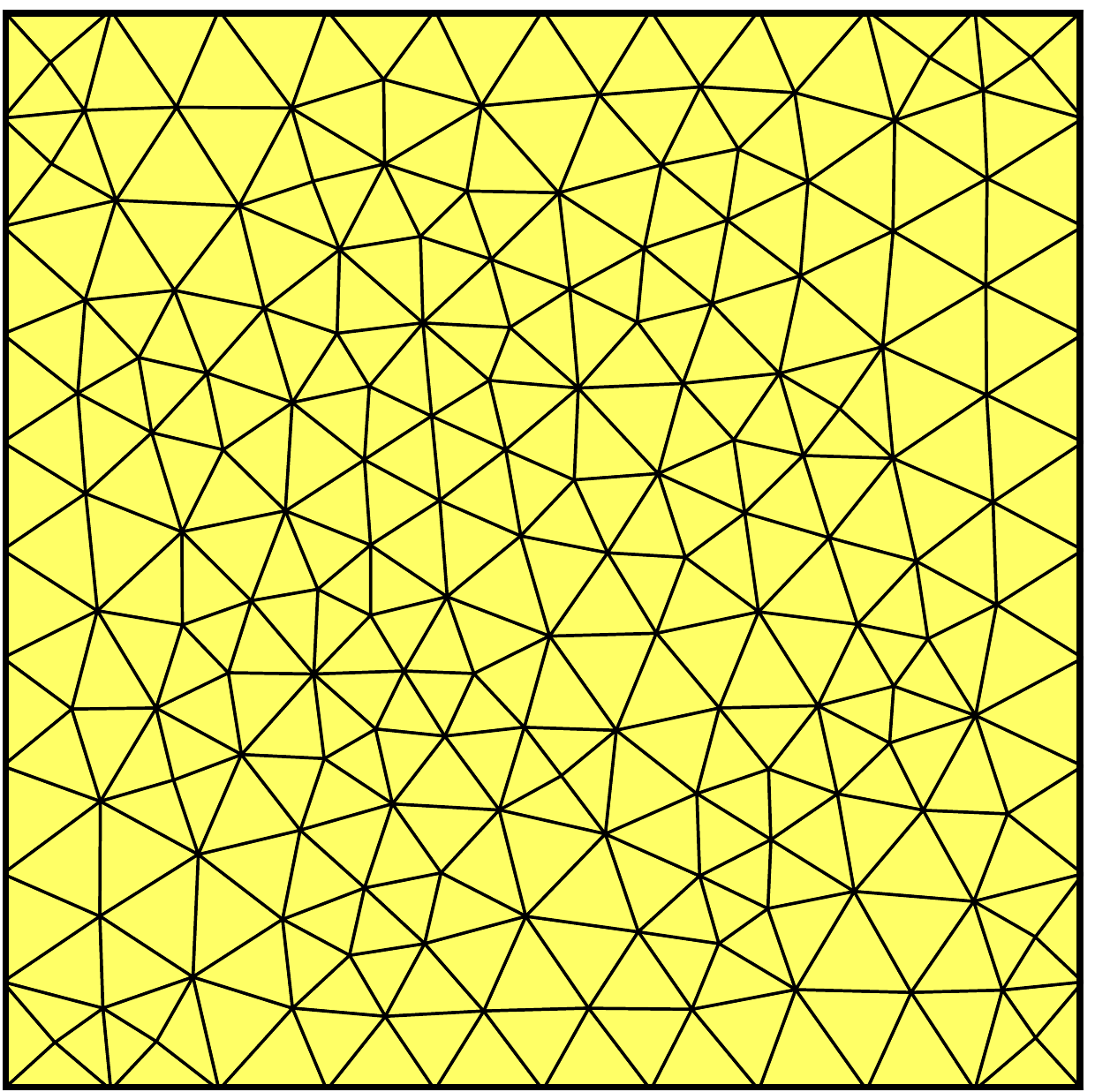}
\includegraphics[height=1.35in]{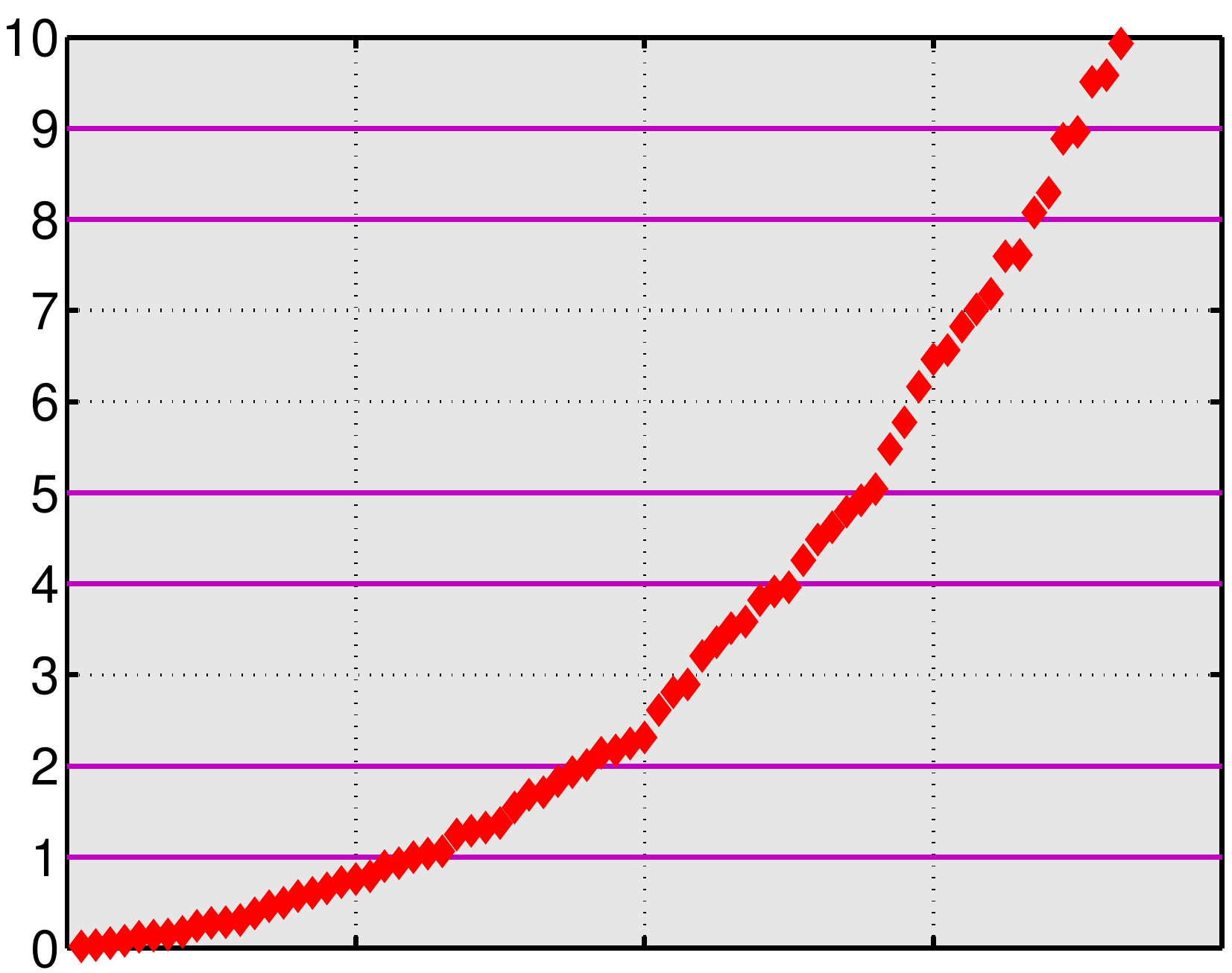}
\includegraphics[height=1.35in]{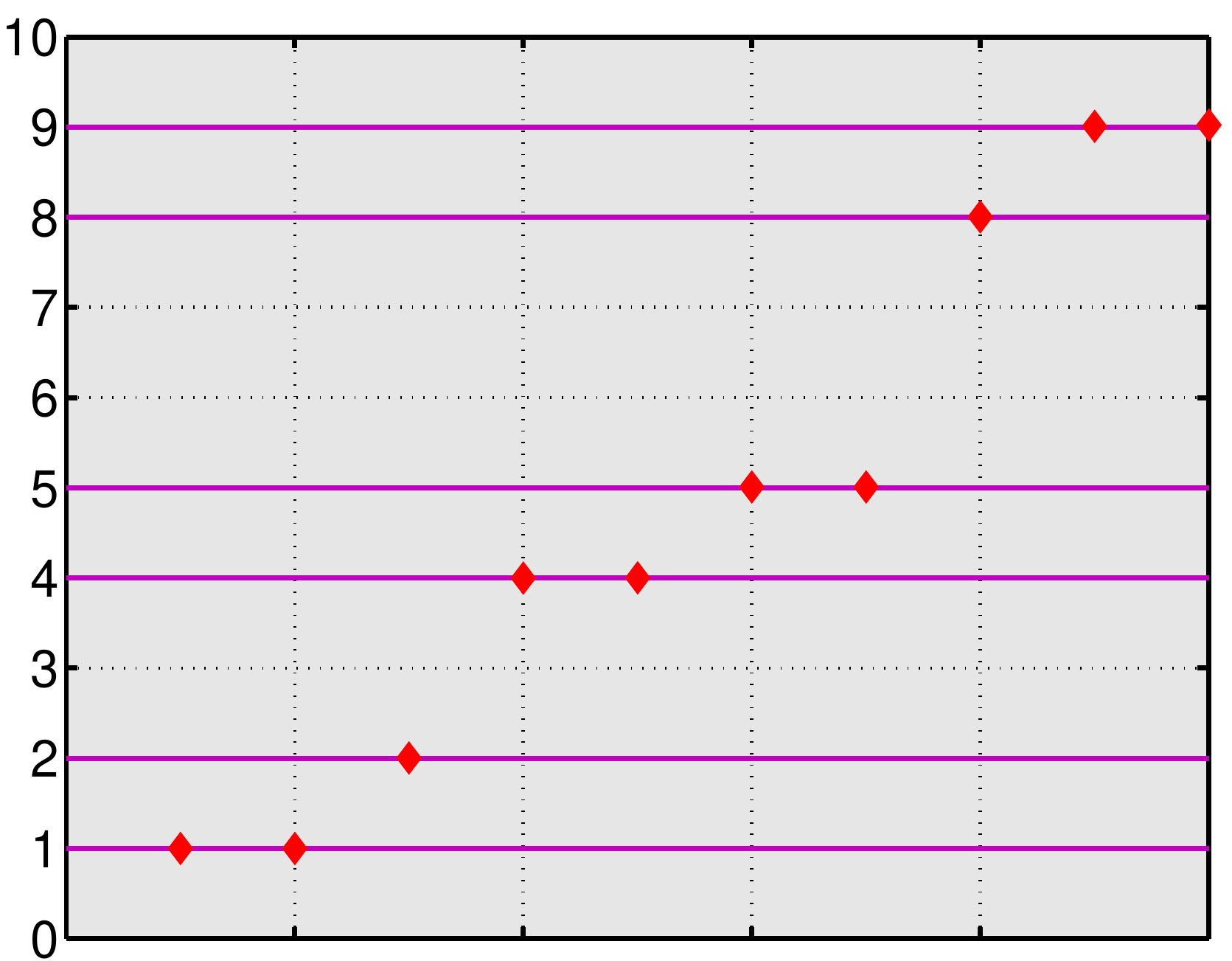}}
\caption[]{Approximation of the nonzero eigenvalues of \eqref{curlcurlev} on an unstructured mesh
of the square (left) using continuous piecewise linear finite elements (middle) and
Raviart--Thomas elements (right).  For the former, the discrete spectrum looks nothing
like the true spectrum, while for the later it is very accurate.}
\label{fg:f3}
\end{figure}

\begin{figure}[htb]
\centerline{\includegraphics[width=1.25in]{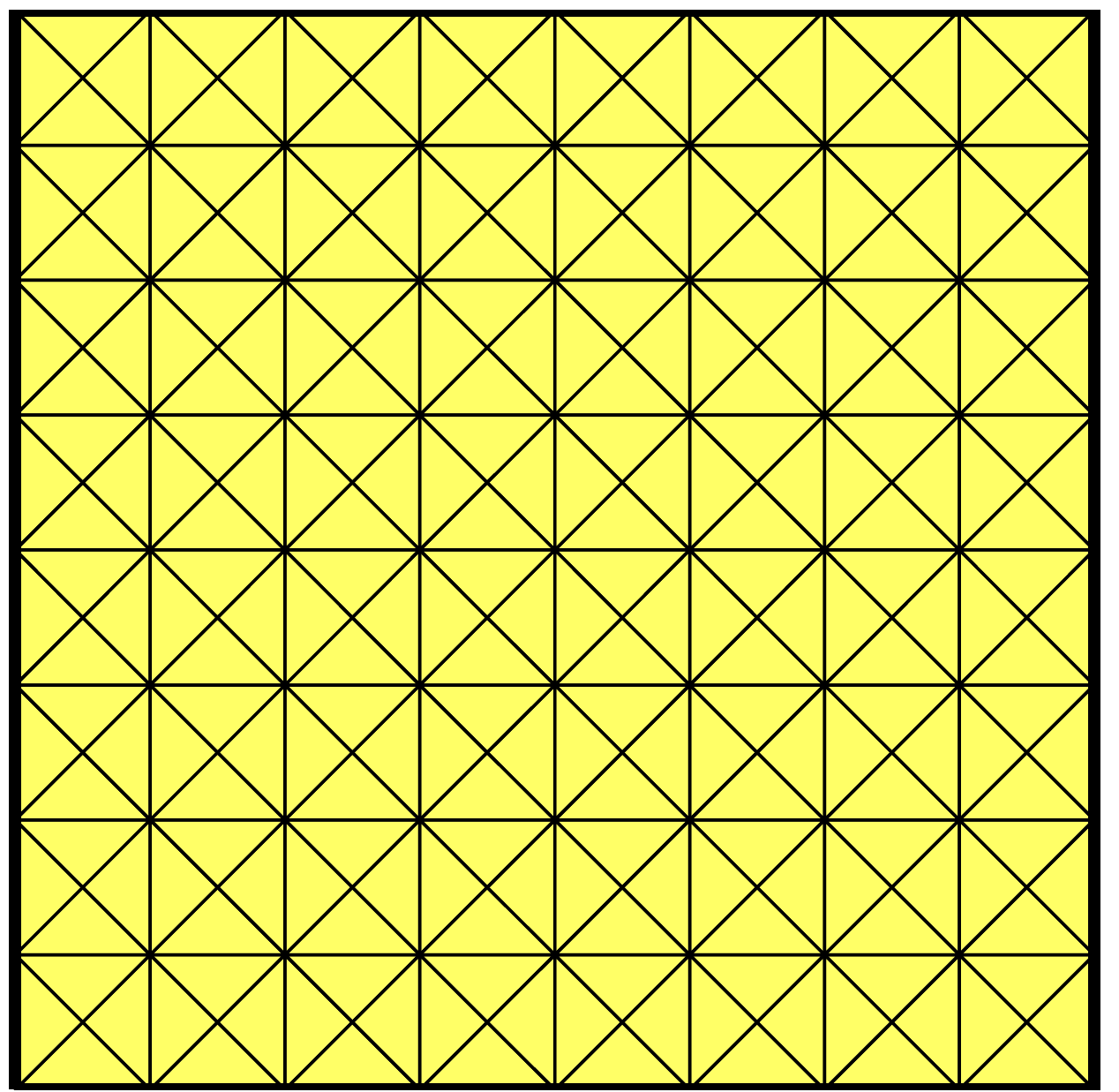} \qquad
\includegraphics[width=1.6in]{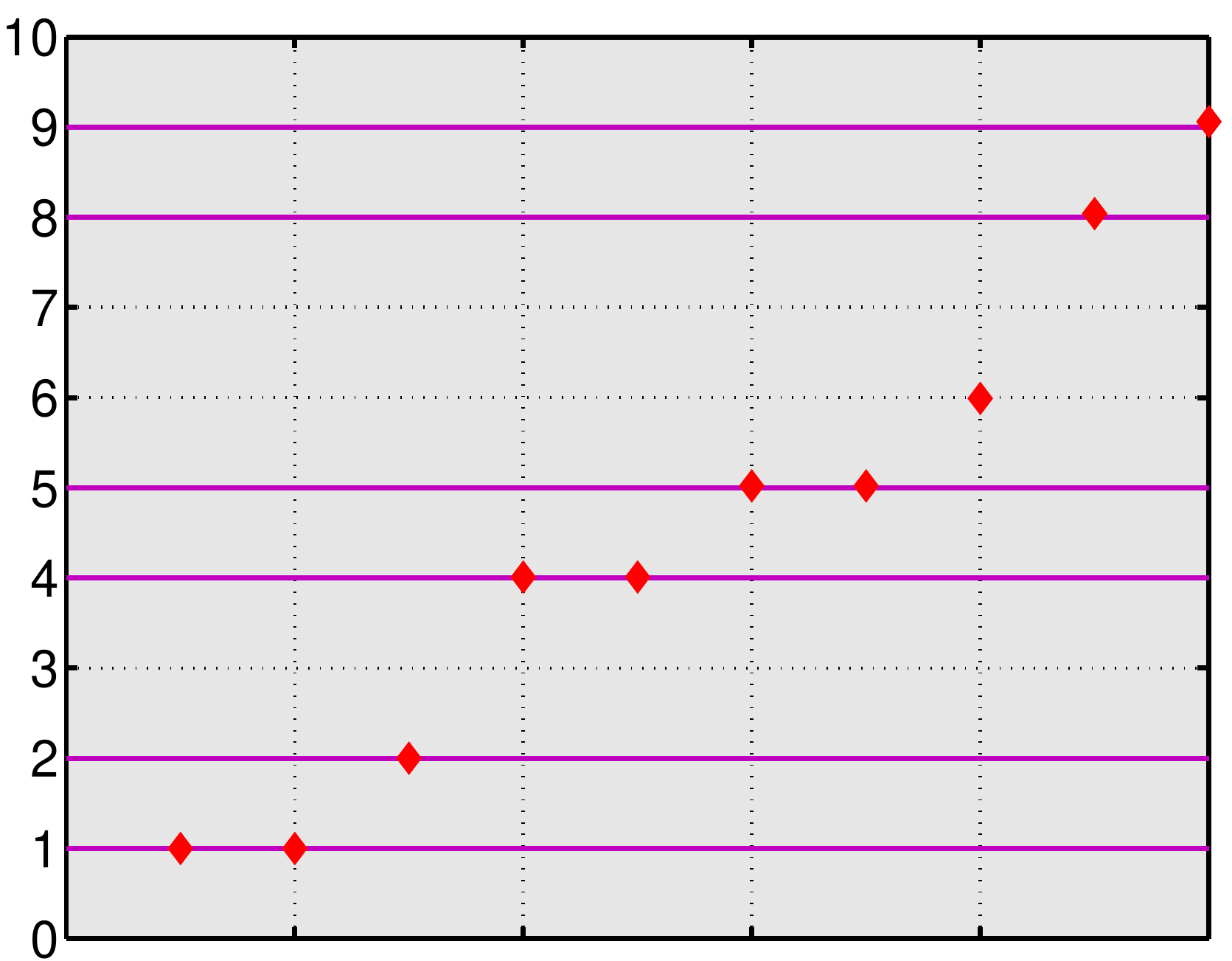}}
\caption[]{Approximation of the nonzero eigenvalues of
\eqref{curlcurlev} using continuous piecewise linear elements on the structured mesh shown.
The first seven discrete nonzero eigenvalues converge to
true eigenvalues, but the eighth converges to a spurious value.}
\label{fg:f4}
\end{figure}

\section{Hilbert complexes and their approximation}
\label{sec:HCA}
In this section, we construct a Hilbert space framework for finite element
exterior calculus.  The most basic object in this framework is a Hilbert
complex, which extracts essential features of the $L^2$ de~Rham complex.
Just as the Hodge Laplacian is naturally associated with the de~Rham
complex, there is a system of variational problems, which we call the
abstract Hodge Laplacian, associated to any Hilbert complex.  Using a
mixed formulation we prove that these abstract Hodge Laplacian problems
are well-posed.  We next consider approximation of Hilbert complexes using
finite dimensional subspaces.  Our approach emphasizes two key properties,
the subcomplex property and the existence of bounded cochain projections.
These same properties prove to be precisely what is needed both to show
that the approximate Hilbert complex accurately reproduces geometrical
quantities associated to the complex, like cohomology spaces, and also
to obtain error estimates for the approximation of the abstract Hodge
Laplacian source  and eigenvalue problems, which is our main goal in
this section.  In the following section of the paper we will derive
finite element subspaces in the concrete case of the de~Rham complex
and verify the hypotheses needed to apply the results of this section.

Although the $L^2$ de~Rham complex is the canonical example of a Hilbert
complex, there are many others.  In this paper, in Section~\ref{sec:variations}
we consider some variations of de~Rham complex that allow us to treat
more general PDEs and boundary value problems.  In the final section we
briefly discuss the equations of elasticity, for which a very different
Hilbert complex, in which one of the differentials is a second order
PDE, is needed.  Another useful feature of Hilbert complexes, is that a
subcomplex of Hilbert complex is again such, and so the properties we
establish for them apply not only at the continuous, but also at the
discrete level.

\subsection{Basic definitions}
\label{subsec:basicdefs}
We begin by recalling some basic definitions of homological algebra and functional analysis and
establishing some notation.
\subsubsection{Cochain complexes}
\label{subsubsec:complexes}
Consider a \emph{cochain complex}
$(V,d)$ of vector spaces, i.e., a sequence of vector spaces $V^k$
and linear maps $d^k$, called the differentials:
\begin{equation*}
\cdots \to V^{k-1}\xrightarrow{d^{k-1}} V^k \xrightarrow{d^k}
V^{k+1} \to \cdots
\end{equation*}
with $d^k\circ d^{k-1}=0$.  Equivalently, we may think of such a complex as
the graded vector space $V=\bigoplus V^k$, equipped with a graded
linear operator $d:V\to V$ of degree $+1$ satisfying $d\circ d=0$.
A chain complex differs from a cochain complex only in that the indices decrease.
All the complexes we consider are nonnegative and finite, meaning that
$V^k=0$ whenever $k$ is negative or sufficiently large.

Given a cochain complex $(V,d)$, the elements of the null space
$\Zfrak^k=\Zfrak^k(V,d)$ of $d^k$ are called the $k$-cocycles and the
elements of the range $\Bfrak^k=\Bfrak^k(V,d)$ of $d^{k-1}$ the
$k$-coboundaries.  The $k$th cohomology group is defined to be the quotient
space $\Zfrak^k/\Bfrak^k$.

Given two cochain complexes $(V,d)$ and $(V',d')$, a set of
linear maps $f^k:V^k\to V'^k$ satisfying $d'^k f^k = f^{k+1}d^k$
(i.e., a graded linear map $f:V\to V'$ of degree $0$
satisfying $d'f=fd$) is called a \emph{cochain map}.  When $f$ is a
cochain map, $f^k$ maps $k$-cocycles to $k$-cocycles and
$k$-coboundaries to $k$-coboundaries, and hence induces a map $\bar f$
on the cohomology spaces.  This map is functorial, i.e., it respects
composition.

Let $(V,d)$ be a cochain complex and $(V_h,d)$ a subcomplex.  In
other words, $V_h^k$ is a subspace of $V^k$ and $d^k
V_h^k\subset V_h^{k+1}$.  Then the inclusion $i_h:V_h\to
V$ is a cochain map and so induces a map of cohomology.  If there exists
a \emph{cochain projection} of $V$ onto $V_h$, i.e., a cochain map
$\pi_h$ such that $\pi_h^k:V^k\to V_h^k$ leaves the subspace $V_h^k$
invariant, then $\pi_h\circ i_h = \id_{V_h}$, so $\bar \pi_h\circ
\bar i_h= \id_{\Zfrak_h^k/\Bfrak_h^k}$ (where
$\Zfrak_h^k:=\Zfrak^k(V_h,d)$ and similarly for $\Bfrak_h^k$).  We
conclude that in this case $\bar i_h$ is injective and $\bar \pi_h$ is
surjective.  In particular, the dimension of the cohomology spaces of the
subcomplex is at most that of the supercomplex.

\subsubsection{Closed operators on Hilbert space}
\label{subsubsec:closed}
This material can be found in many places, e.g.,
\cite[Chapter~III, \S~5 and Chapter IV, \S~5.2]{kato} or \cite[Chapter II,
\S~6 and Chapter VII]{yosida}.

By an operator $T$ from a Hilbert space $X$ to a Hilbert space $Y$, we mean a
linear operator from a
subspace $V$ of $X$, called the domain of $T$, into $Y$. The operator $T$ is
not necessarily bounded and the domain $V$ is not necessarily closed. We say that
the operator $T$ is
\emph{closed} if its graph $\{\,(x,Tx)\,|\,x\in V\,\}$ is closed in $X\x Y$.
We endow the domain $V$ with the graph norm inner product,
\begin{equation*}
  \< v,w\>_{V} = \< v,w\>_X +\< Tv,Tw\>_Y.
\end{equation*}
It is easy to check that this makes $V$ a Hilbert space (i.e., complete), if
and only if $T$ is closed, and moreover, that $T$ is a bounded operator from
$V$ to $Y$.  Of course, the null space of $T$ is the set of elements of its
domain which it maps to $0$, and the range of $T$ is $T(V)$.  The null space of a
closed operator from $X$ to $Y$ is a closed subspace of $X$, but its range
need not be closed in $Y$ (even if the operator is defined on all of $X$ and
is bounded).

The operator $T$ is said to be densely defined if its domain $V$ is dense in
$X$.  In this case the adjoint operator $T^*$ from $Y$ to $X$ is defined to be
the operator whose domain consists of all $y\in Y$ for which there exists
$x\in X$ with
\begin{equation*}
 \< x,v\>_X=\< y,Tv\>_Y, \quad v\in V,
\end{equation*}
in which case $T^*y=x$ (well-defined since $V$ is dense).  If $T$ is closed
and densely defined, then $T^*$ is as well and $T^{**}=T$.  Moreover, the null
space of $T^*$ is the orthogonal complement of the range of $T$ in $Y$.
Finally, by the closed range theorem, the range of $T$ is closed
in $Y$ if and only if the range of $T^*$ is closed in $X$.

If the range of a closed linear operator is of finite codimension, i.e., 
$\dim Y/T(V)<\infty$, then the range is closed \cite[Lemma 19.1.1]{hormander}.

\subsubsection{Hilbert complexes}
\label{subsubsec:hilbert-complexes}
 A \emph{Hilbert complex} is a sequence of Hilbert spaces $W^k$ and closed, densely-defined
linear operators $d^k$ from $W^k$ to $W^{k+1}$ such that the range of $d^k$ is contained in
the domain of $d^{k+1}$ and $d^{k+1}\circ d^k=0$.  A Hilbert complex is \emph{bounded} if,
for each $k$, $d^k$ is a bounded linear operator from
$W^k$ to $W^{k+1}$.  In other words, a bounded Hilbert complex is a cochain complex in the category
of Hilbert spaces. 
A Hilbert complex is \emph{closed} if for each $k$, the range of $d^k$ is closed in $W^{k+1}$.
A \emph{Fredholm complex} is a Hilbert complex for which the range of $d^k$ is finite codimensional
in the null space of $d^{k+1}$ (and so is closed).
Hilbert and Fredholm complexes have been discussed by various authors working in
geometry and topology.
Br\"uning and Lesch \cite{bruning-lesch} have advocated for them as an abstraction of
elliptic complexes on manifolds and applied them to spectral geometry on singular spaces.
Glotko
\cite{glotko} used them to define a generalization of Sobolev spaces on Riemannian
manifolds and study their compactness properties and Gromov and Shubin \cite{gromov-shubin} to
define topological invariants of manifolds.

Associated to any Hilbert complex $(W,d)$ is a bounded Hilbert complex $(V,d)$,
called the \emph{domain complex}, for which the space $V^k$ is the domain of $d^k$, endowed with the
inner product associated to the graph norm:
$$
\< u,v\>_{V^k} = \<u,v\>_{W^k} + \<d^ku, d^kv\>_{W^{k+1}}.
$$
Then $d^k$ is a bounded linear operator from $V^k$ to $V^{k+1}$, and so $(V,d)$ is indeed a bounded
Hilbert complex.  The domain complex is closed or Fredholm if and only if the original complex $(W,d)$ is.

Of course, for a Hilbert complex $(W,d)$, we have the null spaces and ranges $\Zfrak^k$ and
$\Bfrak^k$.  Utilizing the inner product, we define the space of \emph{harmonic forms}
$\Hfrak^k=\Zfrak^k\cap\Bfrak^{k\perp}$, the orthogonal complement of $\Bfrak^k$ in $\Zfrak^k$.
It is isomorphic
to the reduced cohomology space $\Zfrak^k/\overline{\Bfrak^k}$ or, for a closed complex,
to the cohomology space $\Zfrak^k/\Bfrak^k$.  For a closed Hilbert complex, we immediately
obtain the \emph{Hodge decomposition}
\begin{equation}\label{hodgedecompw}
W^k=\Bfrak^k\oplus\Hfrak^k\oplus \Zfrak^{k\perp_W}.
\end{equation}
For the domain complex $(V,d)$, the null space, range, and harmonic forms are the same
spaces as for the original complex, and the Hodge decomposition is
\begin{equation*}
V^k = \Bfrak^k \oplus \Hfrak^k \oplus \Zfrak^{k\perp_V}.
\end{equation*}
The third summand, $\Zfrak^{k\perp_V}=\Zfrak^{k\perp_W}\cap V^k$.

Continuing to use the Hilbert space structure,
we define the \emph{dual complex} $(W,d^*)$, which is a Hilbert chain complex rather than cochain
complex.  The dual complex uses
the same spaces $W^i$, with the differential $d^*_k$ being the adjoint of
$d^{k-1}$, so $d^*_k$ is a closed, densely-defined operator from
$W^k$ to $W^{k-1}$, whose domain we denote by $V^*_k$.  The dual complex is closed or bounded
if and only if the
original complex is.  We denote by $\Zfrak^*_k=\Bfrak^{k\perp_W}$ the null space of $d^*_k$, and by
$\Bfrak^*_k$ the range of $d^*_{k+1}$.  Thus $\Hfrak^k=\Zfrak^k\cap\Zfrak^*_k$ is the space of
harmonic forms both for the original complex and the dual complex.
Since $\Zfrak^{k\perp_W}=\Bfrak^*_k$, the Hodge decomposition \eqref{hodgedecompw} can be
written
 \begin{equation}\label{hodgedecompW}
W^k = \Bfrak^k \oplus \Hfrak^k \oplus \Bfrak^*_k.
\end{equation}
We henceforth simply write $\Zfrak^{k\perp}$ for $\Zfrak^{k\perp_V}$.

Let $(W,d)$ be a closed Hilbert complex with domain complex $(V,d)$.
Then $d^k$ is a bounded bijection from
$\Zfrak^{k\perp}$ to the Hilbert space $\Bfrak^{k+1}$ and hence,
by Banach's bounded inverse theorem, there exists a constant $c_P$ such that
\begin{equation}\label{poincare}
\|v\|_V \le c_P \|d^k v\|_W,
\quad v \in \Zfrak^{k\perp},
\end{equation}
which we refer to as a \emph{Poincar\'e inequality}.  We remark that the condition that
$\Bfrak^k$ is closed is not only sufficient to obtain \eqref{poincare}, but
also necessary.

The subspace $V^k\cap V^*_k$ of $W^k$ is a Hilbert space with the norm
\begin{equation*}
 \|v\|_{V\cap V^*}^2 = \|v\|_V^2 + \|v\|_{V^*}^2 = 2\|v\|_W^2 + \|d^kv\|_W^2 
+ \|d_k^*v\|_W^2,
\end{equation*}
and is continuously included in $W^k$.  We say that the Hilbert complex $(W,d)$ has the
\emph{compactness property} if $V^k\cap V^*_k$ is dense in $W^k$ and the inclusion
is a compact operator.  Restricted to the space $\Hfrak^k$ of harmonic forms, the
$V^k\cap V^*_k$ norm is equal to the $W^k$ norm (times $\sqrt 2$).  Therefore the
compactness property implies that the inclusion of $\Hfrak^k$ into itself is compact,
so $\Hfrak^k$ is finite dimensional.  In summary, for Hilbert complexes,
$\text{compactness property} \implies \text{Fredholm} \implies \text{closed}$.

\subsection{The abstract Hodge Laplacian and the mixed formulation}
\label{subsec:abstract-hl}
Given a Hilbert complex $(W,d)$, the operator $L=dd^*+d^*d$ is an unbounded
operator $W^k\to W^k$ called, in the case of the de~Rham complex, the Hodge
Laplacian.  We refer to it as the abstract Hodge Laplacian in the general
situation.  Its domain is
$$
D_L=\{\,u\in V^k\cap V^*_k\,|\, du\in V^*_{k+1},d^*u\in V^{k-1}\,\}.
$$
If $u$ solves $Lu=f$, then
\begin{equation}\label{primal}
\< du,dv\>+\< d^*u,d^*v\> =\< f,v\>,\quad v\in V^k\cap V^*_k.
\end{equation}
Note that, in this equation, and henceforth, we use $\<\,\cdot\,,\,\cdot\,\>$
and $\|\,\cdot\,\|$ without subscripts, meaning the inner product and norm in the
appropriate $W^k$ space.

The harmonic functions measure the extent to which the Hodge Laplace problem
\eqref{primal} is well-posed.  The solutions to the homogeneous problem ($f=0$)
are precisely the functions in $\Hfrak^k$.  Moreover, a necessary condition
for a solution to exist for non-zero $f\in W^k$ is that $f\perp\Hfrak^k$.

For computational purposes, a formulation of the Hodge Laplacian based on
\eqref{primal} may be problematic, even when there are no harmonic forms,
because it may not be possible to construct an efficient finite element
approximation for the space $V^k\cap V^*_k$.  We have already seen an example
of this in the discussion of the approximation of a boundary value problem for
the vector Laplacian in Section~\ref{subsubsec:l-shape}. Instead we introduce
another formulation, which is a generalization of the mixed formulation discussed in
Section~\ref{sec:femethod} and which, simultaneously, accounts for
the nonuniqueness associated with harmonic forms.  With $(W,d)$ a Hilbert complex, $(V,d)$
the associated domain complex, and $f \in W^k$ given, we define the
\emph{mixed formulation of the abstract Hodge Laplacian} as the problem of
finding $(\sigma,u,p)\in V^{k-1}\x V^k \x \Hfrak^k$ satisfying
\begin{equation}\label{wfhc}
\begin{aligned}
\<\sigma,\tau\> - \<d \tau,u\> &=0,   &&\tau\in V^{k-1},
\\*
\< d \sigma,v\> + \<d  u,d  v\> + \< v, p \> 
&= \< f, v \>,
&&v\in V^k,
\\*
\< u, q \> &= 0, &&q\in \Hfrak^k.
\end{aligned}
\end{equation}

\begin{remark}
The equations \eqref{wfhc} are the Euler--Lagrange equations associated to
a variational problem.  Namely, if we define the quadratic functional
$I: V^{k-1} \x V^k \x \Hfrak^k \to \R$ by
\begin{equation*}
I(\tau,v,q)=\frac12\<\tau,\tau\> -\< d \tau,v\> -\frac12\< d  v,d  v\>
- \< v, q \> + \< f, v \>,
\end{equation*}
then a point $(\sigma,u,p)\in V^{k-1}\x V^k \x \Hfrak^k$ is a critical point of
$I$ if and only if \eqref{wfhc} holds, and in this case
$$
I(\sigma,v,q)\le I(\sigma,u,p) \le J(\tau,u,p), 
\quad (\tau,v,q)\in V^{k-1} \x V^k \x \Hfrak^k.
$$
Thus the critical point is a saddle point.
\end{remark}

An important result is that if the Hilbert complex is closed, then the mixed formulation
is well-posed.  The requirement that the complex is closed is crucial, since
we rely on the Poincar\'e inequality.
\begin{thm}\label{wp}
Let $(W,d)$ be a closed Hilbert complex with domain complex $(V,d)$.
The mixed formulation of the abstract Hodge Laplacian is well-posed.  That is,
 for any $f\in W^k$, there exists a unique $(\sigma,u,p)\in V^{k-1}\x V^k \x
 \Hfrak^k$ satisfying \eqref{wfhc}.  Moreover
$$
\|\sigma\|_V+\|u\|_V + \|p\| \le c \|f\|,
$$ where $c$ is a constant depending only on the Poincar\'e constant $c_P$ in
\eqref{poincare}.
\end{thm}

We shall prove Theorem~\ref{wp} in Section~\ref{subsubsec:well-p}.  First, we interpret the
mixed formulation.

\subsubsection{Interpretation of the mixed formulation}
\label{subsubsec:interp}
The first equation states that $u$ belongs to the domain of
$d^*$ and $d^*u=\sigma\in V^{k-1}$.  The second equation similarly states that
$du$ belongs to the domain of $d^*$ and $d^*du=f-p-d\sigma$.  Thus $u$ belongs
to the domain $D_L$ of $L$ and solves the abstract Hodge Laplacian equation
$$
Lu=f-p.
$$ 
The harmonic form $p$ is simply the orthogonal projection $P_\Hfrak f$ of $f$ onto $\Hfrak^k$, required
for existence of a solution.  Finally the third equation fixes a particular
solution, through the condition $u\perp\Hfrak^k$.  Thus Theorem~\ref{wp}
establishes that for any $f\in W^k$ there is a unique $u\in D_L$ such
that $Lu=f-P_\Hfrak f$ and $u\perp\Hfrak^k$.  We
define $Kf=u$, so the solution operator
$K:W^k\to W^k$ is a bounded linear operator mapping into $D_L$.
The solution to the mixed formulation is
\begin{equation*}
 \sigma=d^*Kf, \quad u=Kf, \quad p=P_{\Hfrak}f.
\end{equation*}

The mixed formulation \eqref{wfhc} is also intimately connected to the
Hodge decomposition \eqref{hodgedecompW}.  Since $d\sigma\in\B^k$,
$p\in\Hfrak^k$, and $d^*du\in\Bfrak^*_k$, the expression $f=d\sigma+p+d^*du$
is precisely the Hodge decomposition of $f$.
In other words
\begin{equation*}
 P_\Bfrak = dd^*K, \quad P_{\Bfrak^*} = d^*dK,
\end{equation*}
where $P_\Bfrak$ and $P_{\Bfrak^*}$ are the $W^k$-orthogonal projections onto
$\Bfrak^k$ and $\Bfrak^*_k$, respectively.
We also note that $K$ commutes with $d$ and $d^*$ in the sense that
\begin{equation*}
dKf=Kdf, \ f\in V^k, \quad d^*Kg = Kd^*g, \ g\in V^*_k.
\end{equation*}
Indeed, if $f\in V^k$ and $u=Kf$, then $u\in D_L$, which implies that
$du\in V^{k+1}\cap V^*_{k+1}$.  Also $d^*u\in V^{k-1}$, so
$d^*du=f-P_\Hfrak f - dd^*u\in V^k$.  This shows that $du\in D_L$.
Clearly
\begin{equation*}
 Ldu=(dd^*+d^*d)du=dd^*du=d(dd^*+d^*d)u=dLu=df,
\end{equation*}
and both $du$ and $df$ are orthogonal to harmonic forms.  This establishes
that $du=Kdf$, i.e., $dKf=Kdf$.  The second equation is established similarly.

If we restrict the data $f$ in the abstract Hodge Laplacian to an element of
$\Bfrak^*_k$ or of $\Bfrak^k$, we get two other problems which are also of
great use in applications.

\paragraph{\it The $\Bfrak^*$ problem.}  If $f\in\Bfrak^*_k$, then $u=Kf\in \Bfrak^*_k$ satisfies $$
d^*d u = f, \quad u\perp \Zfrak^k,
$$
while $\sigma=d^*u=0$, $p=P_\Hfrak f=0$.   The solution $u$ can be characterized as
the unique element of $\Zfrak^{k\perp}$ such that
\begin{equation}\label{bfrak*}
 \< du,dv\>=\<f,v\>, \quad v\in\Zfrak^{k\perp}.
\end{equation}
and any solution to this problem is a solution of $Lu=f$, and so is uniquely determined.

\paragraph{\it The $\Bfrak$ problem.}  If $f\in\Bfrak^k$, then $u=Kf\in \Bfrak^k$
satisfies $dd^*u=f$, while $p=P_\Hfrak f=0$.  With $\sigma=d^*u$, the pair
$(\sigma,u)\in V^{k-1}\x \Bfrak^k$ is the unique solution of
\begin{equation}\label{bfrak}
 \<\sigma,\tau\>-\<d\tau,u\>=0, \ \tau\in V^{k-1}, \quad \<
d\sigma,v\>=\<f,v\>, \ v\in\Bfrak^k,
\end{equation}
and any solution to this problem is a solution of $Lu=f$, $\sigma=d^*u$, and so is
uniquely determined.

\subsubsection{Well-posedness of the mixed formulation}
\label{subsubsec:well-p}
We now turn to the proof of Theorem~\ref{wp}.  Let $B:X\x X\to\R$ be a
symmetric bounded bilinear form on a Hilbert space $X$ which satisfies the
inf-sup condition
$$
\gamma:=\inf_{0\ne y\in X}\sup_{0\ne x\in X}\frac{B(x,y)}{\|x\|_X\|y\|_X}>0.
$$ Then the problem of finding $x\in X$ such that $B(x,y)=F(y)$ for all $y\in
X$ is well-posed: it has a unique solution $x$ for each $F\in X^*$, and
$\|x\|_X\le \gamma^{-1}\|F\|_{X^*}$ \cite{babuska-71}.  The abstract Hodge
Laplacian problem \eqref{wfhc} is of this form, where $B:[V^{k-1} \x V^k \x
\Hfrak^k] \x [V^{k-1} \x V^k \x \Hfrak^k] \to\R$ denotes the bounded bilinear
form
\begin{equation*}
B(\sigma,u,p;\tau,v,q)= \< \sigma,\tau\> - \< d \tau,u\> + \< d \sigma,v\>
+ \< d u, d v\> + \< v,p\> - \< u,q \>,
\end{equation*}
and $F(\tau,v,p)=\< f,v\>$.

 The following theorem establishes the inf-sup
condition, and so implies Theorem~\ref{wp}.

\begin{thm}\label{inf-sup-c}
Let $(W,d)$ be a closed Hilbert complex with domain complex $(V,d)$.
There exists a constant $\gamma>0$, depending only on the constant $c_P$ in
the Poincar\'e inequality \eqref{poincare}, such that for any $(\sigma,u,p)\in
V^{k-1} \x V^k \x \Hfrak^k$, there exists $(\tau,v,q)\in V^{k-1} \x V^k \x
\Hfrak^k$ with
\begin{equation*}
B(\sigma,u,p;\tau,v,q)\ge \gamma(\|\sigma\|_{V}+\|u\|_{V}
+ \|p\|)(\|\tau\|_{V}+\|v\|_{V} + \|q\|).
\end{equation*}
\end{thm}

\begin{proof}
By the Hodge decomposition, $u= u_\Bfrak + u_{\Hfrak} + u_\perp$, where
$u_\Bfrak = P_\Bfrak u$, $u_{\Hfrak} = P_\Hfrak u$, and 
$u_\perp = P_{\Bfrak^{*}}u$.
Since $u_\Bfrak \in \Bfrak^k$, $u_\Bfrak =d \rho$, for some $\rho \in \Zfrak^{k-1
\perp}$.  Since $d u_\perp =d u$, we get using \eqref{poincare} that
\begin{equation}
\label{poincare-stab}
\|\rho\|_V \le c_P  \|u_\Bfrak\|, \qquad \|u_\perp\|_V \le c_P \|d u\|,
\end{equation}
where $c_P\ge 1$ is the constant in Poincare\'s inequality. 
Let
\begin{equation}\label{tfdef}
\tau=\sigma-\frac{1}{c_P^2}\rho\in V^{k-1},\quad
v=u+d\sigma+p\in V^k,\quad q = p - u_{\Hfrak} \in \Hfrak^k.
\end{equation}
From \eqref{poincare-stab} and the orthogonality of the Hodge
decomposition, we have
\begin{equation}\label{ub}
\|\tau\|_{V}+\|v\|_{V} + \|q\| \le C (\|\sigma\|_{V}+\|u\|_{V} + \|p\|).
\end{equation}
We also get, from a simple
computation using \eqref{poincare-stab} and \eqref{tfdef}, that
\begin{align*}
&B(\sigma,u,p;\tau,v,q)
\\*
&=\|\sigma\|^2+\|d \sigma\|^2
+\|d u\|^2+ \|p\|^2 + \|u_{\Hfrak}\|^2 + \frac1{c_P^2} \|u_\Bfrak\|^2
-\frac1{c_P^2} \<\sigma,\rho\>
\\
&\ge \frac{1}{2}\|\sigma\|^2+\|d \sigma\|^2
+\|d u\|^2 + \|p\|^2 + \|u_{\Hfrak}\|^2
 + \frac1{c_P^2} \|u_\Bfrak\|^2 - \frac1{2c_P^4}\|\rho\|^2
\\
&\ge \frac{1}{2}\|\sigma\|^2+\|d \sigma\|^2
+\|d u\|^2 + \|p\|^2 + \|u_{\Hfrak}\|^2 + 
\frac1{2c_P^2}\|u_\Bfrak\|^2
\\
&\ge \frac{1}{2}\|\sigma\|^2+\|d \sigma\|^2
+\frac{1}{2}\|d u\|^2 + \|p\|^2 + \|u_{\Hfrak}\|^2 
+ \frac{1}{2 c_P^2} \|u_\Bfrak\|^2
+ \frac{1}{2 c_P^2} \|u_\perp\|^2
\\*
&\ge \frac{1}{2c_P^2}(\|\sigma\|_V^2+ \|u\|_V^2
+ \|p\|^2).
\end{align*}
The theorem easily follows from this bound and \eqref{ub}.
\end{proof}

We close this section with two remarks.  First we note that in fact
Theorem~\ref{inf-sup-c} establishes more than the well-posedness of the
problem \eqref{wfhc} stated in Theorem~\ref{wp}.  It establishes that, for any
$G\in (V^{k-1})^*$, $F\in (V^k)^*$, and $R\in (\Hfrak^k)^*$ (these are the
dual spaces furnished with the dual norms), there exists a unique
$(\sigma,u,p)\in V^{k-1}\x V^k \x \Hfrak^k$ satisfying
\begin{equation*}
\begin{aligned}
\<\sigma,\tau\> - \<
d \tau,u\> &=G(\tau),   &&\tau\in V^{k-1},
\\*
\< d \sigma,v\>
+ \<
d  u,d  v\> + \< v, p \> 
&= F(v),
&&v\in V^k,
\\*
\< u, q \> &= R(q), 
&&q\in \Hfrak^k,
\end{aligned}
\end{equation*}
and moreover the correspondence $(\sigma,u,p)\leftrightarrow(F,G,R)$ is an
isomorphism of $V^{k-1}\x V^k \x \Hfrak^k$ onto its dual space.

Second, we note that the above result bears some relation to a fundamental
result in the theory of mixed finite element methods, due to Brezzi
\cite{Brezzi}, which we state here.
\begin{thm}\label{bt}
Let $X$ and $Y$ be Hilbert spaces and $a:X\x X\to \R$, $b:X\x Y\to \R$ bounded
bilinear forms.  Let $Z=\{\,x\in X\,|\, b(x,y)=0\ \forall y\in Y\,\}$, and
suppose that there exists positive constants $\alpha$ and $\gamma$ such that
\begin{enumerate}
\item (coercivity in the kernel) $a(z,z)\ge \alpha \|z\|_X^2,\quad z\in Z$,
\item (inf-sup condition) $\displaystyle \inf_{0\ne y\in Y}\sup_{0\ne x\in X}
\frac{b(x,y)}{\|x\|_X\|y\|_Y} \ge \gamma$.
\end{enumerate}
Then, for all $G\in X^*$, $F\in Y^*$, there exists a unique $u\in X$, $v\in Y$
such that
\begin{equation}\label{mixed}
\begin{aligned}
a(u,x) + b(x,v) &= G(x), \quad x\in X,
\\
b(u,y) &= F(y), \quad y\in Y.
\end{aligned}
\end{equation}
Moreover $\|u\|_X+\|v\|_Y\le c(\|G\|_{X^*} + \|F\|_{Y^*})$ with the constant
$c$ depending only on $\alpha$, $\gamma$, and the norms of the bilinear forms
$a$ and $b$.
\end{thm}
If we make the additional assumption (usually satisfied in applications of
this theorem), that the bilinear form $a$ is symmetric and satisfies
$a(x,x)>0$ for all $0\ne x\in X$, then this theorem can be viewed as a special
case of Theorem~\ref{inf-sup-c}.  In fact, we define $W^0$ as the completion
of $X$ in the inner product given by $a$, let $W^1=Y$, and define $d$ as the
closed linear operator from $W^0$ to $W^1$ with domain $X$ given by
\begin{equation*}
\<dx,y\>_Y = b(x,y), \quad x\in X, y\in Y.
\end{equation*}
In this way we obtain a Hilbert complex (with just two spaces $W^0$ and $W^1$).
The inf-sup condition of Theorem~\ref{bt} implies that $d$ has closed
range, so it is a closed Hilbert complex, and so Poincar\'e's inequality holds.
The associated abstract Hodge Laplacian is just the system \eqref{mixed},
and from Theorem~\ref{inf-sup-c} and the first remark above, we have
\begin{equation*}
\sqrt{a(u,u)}+\|du\|_Y+\|v\|_Y \le c (\|G\|_{X^*} + \|F\|_{Y^*}).
\end{equation*}
But, using the coercivity in the kernel, the decomposition $X=Z+Z^\perp$, and
Poincar\'e's inequality, we get $\|u\|_X\le C[\sqrt{a(u,u)}+\|du\|_Y]$, which gives the
estimate from Brezzi's theorem.  Finally, we mention that we could dispense
with the extra assumption about the symmetry and positivity of the bilinear
form $a$, but this would require a slight generalization of
Theorem~\ref{inf-sup-c} which we do not consider here.

\subsection{Approximation of Hilbert complexes}
\label{subsec:approxhc}
The remainder of this section will be devoted to approximation of
quantities associated to a Hilbert complex, such as the cohomology spaces, harmonic
forms, and solutions to the Hodge Laplacian, by quantities associated to
a subcomplex.

Let $(W,d)$ be a Hilbert complex with domain complex $(V,d)$, and suppose
we choose a finite-dimensional subspace $V^k_h$ of $V^k$ for each $k$.
We assume that $dV^k_h\subset V^{k+1}_h$ so that $(V_h,d)$ is a subcomplex
of $(V,d)$.  We also take $W^k_h$ to be the same subspace $V^k_h$ but endowed
with the norm of $W^k$.  In this way we obtain a closed (even bounded) Hilbert complex
$(W_h,d)$ with domain complex $(V_h,d)$, and all the results of
Sections~\ref{subsubsec:hilbert-complexes} and \ref{subsec:abstract-hl} apply to this
subcomplex.  Although the differential for the subcomplex is just the restriction
of $d$, and so does not need a new notation, its adjoint $d^*_h:V^{k+1}_h\to V^k_h$,
defined by
\begin{equation*}
 \< d^*_h u,v\>=\<u,dv\>,\quad u\in V^{k+1}_h,v\in V^k_h,
\end{equation*}
is not the restriction of $d^*$.  We use the notations
$\Bfrak_h$, $\Zfrak_h$, $\Bfrak^*_h=\Zfrak^\perp_h$, $\Hfrak_h$, $K_h$,
with the obvious meanings.  We use the term discrete when we wish to emphasize quantities
associated to the subcomplex.  For example, we refer to
$\Hfrak^k_h$ as the space of discrete harmonic $k$-forms, and the discrete Hodge decomposition
is
$$
V^k_h=\Bfrak^k_h\oplus\Hfrak^k_h\oplus\Zfrak^{k\perp}_h.
$$
The $W^k$-projections $P_{\Bfrak_h}:W^k\to\Bfrak^k_h$, $P_{\Bfrak^*_h}:W^k\to\Zfrak^{k\perp}_h$
satisfy $P_{\Bfrak_h}=dd^*_hK_h$
and $P_{\Bfrak^*_h}=d^*_hdK_h$, respectively, when restricted to $V^k_h$.
We also have that $K_h$ commutes with both $d$ and $d^*_h$.
Note that $\Bfrak^k_h\subset\Bfrak^k$ and $\Zfrak^k_h\subset\Zfrak^k$,
but in general $\Hfrak^k_h$ is not contained in $\Hfrak^k$,
nor is $\Zfrak^{k\perp}_h$ contained in $\Zfrak^{k\perp}$.

In order that the subspaces $V^k_h$ can be used effectively to approximate quantities associated
to the original complex, we require not only that they form a subcomplex, but of course
need to know something about the \emph{approximation} of $V^k$ by $V^k_h$, i.e., an assumption that
$\inf_{v\in V^k_h}\|u-v\|_V$ is sufficiently small for some or all $u\in
V^k$.  A third assumption, which plays
an essential role in our analysis, is that there exists a \emph{bounded
cochain projection} $\pi_h$ from the complex $(V,d)$ to the subcomplex $(V_h,d)$.
Explicitly, for each $k$, $\pi^k_h$ maps $V^k$ to  $V^k_h$, leaves the subspace invariant,
satisfies $d^k\pi^{k}_h=\pi^{k+1}d^k$, and there exists a constant
$c$ such that $\|\pi^k_hv\|_V\le c\|v\|_V$ for all $v\in V^k$.
In other words, we have the following commuting diagram relating the
complex $(V,d)$ to the subcomplex $(V_h,d)$:
\[
\begin{CD}
0\to\,@. V^0 @>\D>> V^1 @>\D>>
\cdots @>\D>> V^n @.\,\to0\\
 @. @VV\pi_h V @VV\pi_hV @.  @VV\pi_hV\\
0\to\,@. V^0_h @>\D>> 
V^1_h  @>\D>>
\cdots @>\D>> V^n_h @.\,\to 0.
\end{CD}
\] 
Note that a bounded projection gives quasioptimal approximation:
\begin{equation*}
 \|u-\pi_h u\|_V = \inf_{v\in V^k_h}\|(I-\pi_h)(u-v)\|_V \le c \inf_{v\in V^k_h}\|u-v\|_V.
\end{equation*}

We now present two results indicating
that, under these assumptions,
the space $\Hfrak^k_h$ of discrete harmonic forms provides a faithful
approximation of $\Hfrak^k$.
In the first result we show that
a bounded cochain projection into a subcomplex of a bounded closed Hilbert complex which
satisfies a rather weak approximability assumption (namely \eqref{pi-bnd}
below), induces, not only a surjection, but an \emph{isomorphism on
cohomology}.

\begin{thm}
\label{isom-cohom}
Let $(V,d)$ be a bounded closed Hilbert complex, $(V_h,d)$ a Hilbert subcomplex, and $\pi_h$
a bounded cochain projection.  Suppose that for all $k$
\begin{equation}\label{pi-bnd}
\|q-\pi_h^k q\|_V < \|q\|_V, \quad 0\ne q\in\Hfrak^k.
\end{equation}
Then the induced map on cohomology is an isomorphism.
\end{thm}
\begin{proof}
We already know that the induced map is a surjection, so it is sufficient to
prove that it is an injection.  Thus, given $z\in\Zfrak^k$ with $\pi_h
z\in\Bfrak_h^k$, we must prove that $z\in\Bfrak^k$.  By the Hodge
decomposition, $z=q+b$ with $q\in\Hfrak^k$ and $b\in\Bfrak^k$.  We have that
$\pi_hz\in\Bfrak_h^k$ by assumption and $\pi_hb\in \Bfrak_h^k$ since
$b\in\Bfrak^k$ and $\pi_h$ is a cochain map.  Thus
$\pi_hq=\pi_hz-\pi_hb\in \Bfrak_h^k\subset\Bfrak^k$, and so
$\pi_hq\perp q$.  In view of \eqref{pi-bnd}, this implies that $q=0$, and so
$z\in\Bfrak^k$ as desired.
\end{proof}

\begin{remark}  In applications, the space of harmonic forms, 
$\Hfrak^k$, is a finite-dimensional space of smooth functions, and $\pi_h$
is a projection operator associated to a triangulation with mesh size $h$.
The estimate \eqref{pi-bnd} will then be satisfied for $h$ sufficiently small.
However in the most important application, in which $(V,d)$ is the de~Rham complex
and $(V_h,d)$ is a finite element discretization, $\pi_h$ induces an isomorphism on
cohomology not only for $h$ sufficiently small, but in fact for all $h$.  See Section~\ref{subsec:ahl}.
\end{remark}

The second result relating $\Hfrak^k$ and $\Hfrak^k_h$ is quantitative in nature,
bounding the distance, or \emph{gap}, between these two spaces.
Recall that the gap between two subspaces $E$ and $F$ of
a Hilbert space $V$ is defined \cite[Chapter IV, \S~2.1]{kato}
\begin{equation}
\label{gapdef}
\gap (E,F) = \max\biggl(\sup_{\substack{u \in E\\ \| u \| = 1}}
   \inf_{v \in F} \| u - v \|_V, 
\sup_{\substack{v \in F\\  \| v\| = 1}} \inf_{u \in E} \| u - v \|_V \biggr).
\end{equation}
\begin{thm}
\label{gap}
Let $(V,d)$ be a bounded closed Hilbert complex, $(V_h,d)$ a Hilbert subcomplex, and $\pi_h$
a bounded cochain projection.  Then
\begin{gather}
\label{gap1}
 \|(I-P_{\Hfrak_h})q\|_V\le \|(I-\pi_h^k)q\|_V, \quad q\in\Hfrak^k,
\\
\label{gap2}
 \|(I-P_{\Hfrak})q\|_V\le \|(I-\pi_h^k)P_{\Hfrak}q\|_V, \quad q\in\Hfrak^k_h,
\\
\label{gap0}
\gap\bigl(\Hfrak^k,\Hfrak_h^k\bigr) \le\sup_{\substack{q \in\Hfrak^k \\ \| q\|
= 1}}\|(I-\pi_h^k)q\|_V.
\end{gather}
\end{thm}
\begin{proof}
Given $q\in\Hfrak^k$, $P_{\Hfrak_h}q=P_{\Zfrak_h}q$, since $\Zfrak_h^k= \Hfrak_h^k\oplus\Bfrak_h^k$ and
$q\perp\Bfrak^k\supset\Bfrak_h^k$.  Also $\pi_h^k q\in\Zfrak_h^k$, since
$\pi_h$ is a cochain map.  This implies \eqref{gap1}.

If $q\in\Hfrak_h^k\subset\Zfrak_h^k\subset\Zfrak^k$, the Hodge
decomposition gives us $q-P_{\Hfrak}q\in \Bfrak^k$, so $\pi^k_h(q-P_{\Hfrak}q)\in
\Bfrak^k_h$, and so is orthogonal to both the discrete harmonic
form $q$ and the harmonic form $P_{\Hfrak}q$.  Therefore
$$
\|q-P_{\Hfrak}q\|_V\le \|q-P_{\Hfrak}q - \pi_h(q-P_{\Hfrak}q)\|_V
=\|(I-\pi^k_h)P_{\Hfrak}q\|_V.
$$
Finally \eqref{gap0} is an immediate consequence of \eqref{gap1} and \eqref{gap2}.
\end{proof}

Next we deduce another important property of a Hilbert
subcomplex with a bounded cochain projection.
Since the complex $(V,d)$ is closed and bounded, the Poincar\'e inequality
\eqref{poincare} holds (with the $W$ and $V$ norms coinciding).
Now we obtain the Poincar\'e
inequality for the subcomplex with a constant that depends only on the Poincar\'e constant for
the supercomplex and the norm of the cochain projection.  In the
applications, we will have a sequence of such subcomplexes related to a
decreasing mesh size parameter and this theorem will imply that the discrete
Poincar\'e inequality is uniform with respect to the mesh parameter, an
essential step in proving stability for numerical methods.
\begin{thm}\label{Poincare_discrete}
Let $(V,d)$ be a bounded closed Hilbert complex, $(V_h,d)$ a Hilbert
subcomplex, and $\pi_h^k$ a bounded cochain projection.  Then
\[
\|v\|_V \le c_P\|\pi_h^k\|\|dv\|_V, \quad v\in\Zfrak_h^{k\perp},
\]
where $c_P$ is the constant appearing in the Poincar\'e inequality
\eqref{poincare} and $\|\pi_h^k\|$ denotes the $V^k$ operator norm of
$\pi_h^k$.
\end{thm}

\begin{proof}
Given $v \in \Zfrak^{k\perp}_{h}$, define $z \in \Zfrak^{k\perp} \subset V^k$
by $dz=dv$.  By \eqref{poincare}, $\|z\|\le c_P\|dv\|$, so it is enough to
show that $\|v\|_V \le \|\pi_h\|\|z\|_V$.  Now, $v-\pi_h z\in V^k_h$ and
$d(v-\pi_h z)=0$, so $v-\pi_h z\in\Zfrak^k_h$.  Therefore
\begin{equation*}
\|v\|_V^2=\< v,\pi_h z\>_V +\< v,v-\pi_h z\>_V
=\< v,\pi_h  z\>_V \le \|v\|_V\|\pi_h z\|_V,
\end{equation*}
and the result follows.
\end{proof}

We have established several important properties possessed by a subcomplex
of a bounded closed Hilbert complex with bounded cochain projection. 
We also remark that from \eqref{gap2} and the triangle inequality, we have
\begin{equation}\label{phinv}
 \|q\|_V \le c\| P_{\Hfrak}q\|_V, \quad q\in\Hfrak_h^k.
\end{equation}
We close this section
by presenting a converse result.  Namely we show that if the discrete Poincar\'e inequality 
and the bound \eqref{phinv} hold, then a bounded cochain projection exists.
\begin{thm}\label{converse}
 Let $(V,d)$ be a bounded closed Hilbert complex and $(V_h,d)$ a
subcomplex.  Assume that
\begin{equation*}
\|v\|_V \le c_1 \|dv\|_V, \quad v\in\Zfrak_h^{k\perp}, \quad
\text{and}\quad \|q\|_V \le c_2 \| P_{\Hfrak}q\|_V, \quad
q\in\Hfrak_h^k,
\end{equation*}
for some constants $c_1$ and $c_2$. Then there
exists a bounded cochain projection $\pi_h$ from $(V,d)$
to $(V_h,d)$, and the $V$ operator norm $\|\pi_h\|$
can be bounded in terms of $c_1$ and $c_2$.
\end{thm}
\begin{proof}
As a first step of the proof we define an operator $Q_h : V^k \to
\Zfrak_h^{k\perp}$ by $dQ_hv = P_{\Bfrak_h}dv$. By the first assumption, this
operator is $V$-bounded since
\[
\| Q_hv \|_V \le c_1 \|P_{\Bfrak_h} dv \| \le c_1 \| v \|_V,
\]
and if $v \in V_h$ then $Q_hv = P_{\Bfrak^*_h}v$.
By the second assumption, the operator
$P_\Hfrak|_{\Hfrak_h}$
has a bounded inverse $R_h$ mapping $\Hfrak^k$ to $\Hfrak^k_h$.
Note that for any $v \in \Zfrak_h^k$, $R_h P_\Hfrak v = P_{\Hfrak_h} v$.
We now define $\pi_h : V^k \to V^k_h$ by
\[
\pi_h = P_{\Bfrak_h} + R_hP_{\Hfrak} (I - Q_h) + Q_h.
\]
This operator is bounded in $V^k$ and invariant on $V_h^k$, since the
three terms correspond exactly to the discrete Hodge decomposition in
this case.
Furthermore, $\pi_h d v= P_{\Bfrak_h}dv = d Q_h v = d \pi_h v$, so
$\pi_h$ 
is indeed a bounded cochain projection.
\end{proof}

\subsection{Stability and convergence of the mixed method}
\label{subsec:well-p-h-hc}
Next we consider a closed Hilbert complex $(W,d)$ and
the Galerkin discretization of its Hodge Laplacian using
finite dimensional subspaces $V^k_h$ of the domain spaces $V^k$.  Our main assumptions are
those of Section~\ref{subsec:approxhc}: first, that $dV^k_h\subset
V^{k+1}_h$, so that we obtain a \emph{subcomplex}
\begin{equation*}
0 \to V^0_h \xrightarrow{d} V^1_h \xrightarrow{d} \cdots
\xrightarrow{d} V^n_h \to 0,
\end{equation*}
and, second, that there exists a \emph{bounded cochain projection}
$\pi_h$ from $(V,d)$ to $(V_h,d)$.

Let $f\in W^k$.  In view of the mixed formulation \eqref{wfhc},
we take as an approximation scheme: Find $\sigma_h \in V^{k-1}_h$,
$u_h\in V^k_h$, $p_h \in \Hfrak^k_h$, such that
\begin{equation}\label{wfhhc}
\begin{aligned}
\<\sigma_h,\tau\> - \< d \tau,u_h\> &= 0,  & \tau&\in V^{k-1}_h,
\\*
\<  d \sigma_h,v\> + \< d  u_h, d  v\> + \< v, p_h\>
&= \< f,v\>, &
v&\in V^k_h,
\\* 
\< u_h, q \>  &=0,
&q&\in \Hfrak^k_h.
\end{aligned}
\end{equation}
(Recall that we use $\<\,\cdot\,,\,\cdot\,\>$ and $\|\,\cdot\,\|$
without subscripts for the $W$ inner product and norm.) 
The discretization \eqref{wfhhc} is a generalized Galerkin method
as discussed in Section~\ref{subsec:erroranal}.  In the case that there are
no harmonic forms (and therefore no discrete harmonic forms), it is a Galerkin
method, but in general not, since $\Hfrak^k_h$ is not in general a subspace of
$\Hfrak^k$.  We may write the solution of \eqref{wfhhc}, which always exists and is unique in view of
the results of Section~\ref{subsubsec:well-p}, as
\begin{equation*}
 u_h=K_hP_h f, \quad \sigma_h=d^*_h u_h, \quad p_h=P_{\Hfrak_h}f,
\end{equation*}
where $P_h:W^k\to V^k_h$ is the $W^k$-orthogonal projection.

 As in Section~\ref{subsec:erroranal}, we will bound the error
in terms of the stability of the discretization and the consistency error.
We start by establishing a lower bound on the inf-sup constant, i.e., an
upper bound on the stability constant.
\begin{thm}
\label{inf-sup-d}
Let $(V_h,d)$ be a family of subcomplexes of the domain complex
$(V,d)$ of a closed Hilbert complex, parametrized by $h$ and admitting
uniformly $V$-bounded cochain projections.  Then there exists a constant $\gamma_h>0$,
depending only on $c_P$ and the norm of
the projection operators $\pi_h$, such that for any $(\sigma,u,p)\in
V^{k-1}_h\x V^k_h \x\Hfrak^k_h$, there exists $(\tau,v,q)\in V^{k-1}_h\x V^k_h
\x\Hfrak^k_h$ with
\begin{align*}
B(\sigma,u,p;\tau,v,q)&\ge
\gamma_h(\|\sigma\|_{V}+\|u\|_{V}
+ \|p\|)(\|\tau\|_{V}+\|v\|_{V} +\|q\|).
\end{align*}
\end{thm}

\begin{proof}
This is just Theorem~\ref{inf-sup-c} applied to the Hilbert complex $(V_h,d)$,
combined with the fact that the constant in the Poincar\'e inequality for
$V^k_h$ is $c_P\|\pi_h\|$ by Theorem~\ref{Poincare_discrete}.
\end{proof}

From this stability result, we obtain the following error estimate.
\begin{thm}\label{qo1cor}
Let $(V_h,d)$ be a family of subcomplexes of the domain complex
$(V,d)$ of a closed Hilbert complex, parametrized by $h$ and admitting
uniformly $V$-bounded cochain projections, and let $(\sigma,u,p)\in V^{k-1} \x V^k \x \Hfrak^k$
be the solution of problem
\eqref{wfhc} and $(\sigma_h, u_h, p_h) \in V^{k-1}_h \x V^k_h \x \Hfrak^k_h$
the solution of problem \eqref{wfhhc}.  Then
\begin{multline*}
\|\sigma-\sigma_h\|_{V} + \|u-u_h\|_{V}
+ \|p - p_h\|
\\
\le C(\inf_{\tau\in V_h^{k-1}}\|\sigma-\tau\|_{V}
 +\inf_{v\in V_h^k} \|u-v\|_{V}
 +\inf_{q\in V_h^k} \|p - q\|_{V}
 + \mu\inf_{v\in V_h^k}\|P_\Bfrak u - v\|_V),
\end{multline*}
where
$\displaystyle
\mu =\mu^k_h =
\sup_{\substack{r \in \Hfrak^k \\ \| r \|=1}} \|(I - \pi_h^k)r\|$.
\end{thm}

\begin{proof}
First observe that $(\sigma,u,p)$ satisfies
\begin{equation*}
B(\sigma,u,p;\tau_h,v_h,q_h) = \< f, v_h \> 
- \< u, q_h \>, \quad 
(\tau_h,v_h,q_h) \in V^{k-1}_h \x V^{k}_h \x \Hfrak^k_h.
\end{equation*}
Let $\tau$, $v$, and $q$ be the $V$-orthogonal projections of
$\sigma$, $u$, and $p$ into $V^{k-1}_h$, $V^{k}_h$, and $\Hfrak^k_h$, respectively.
Then, for any
$(\tau_h,v_h,q_h) \in V^{k-1}_h \x V^{k}_h \x \Hfrak^k_h$, we have
\begin{align*}
B(\sigma_h&- \tau,u_h -v,p_h -q;\tau_h,v_h,q_h) 
\\*
&= B(\sigma- \tau,u -v,p -q;\tau_h,v_h,q_h) + \< u, q_h \>
\\
&= B(\sigma- \tau,u -v,p -q;\tau_h,v_h,q_h) 
+ \< P_{\Hfrak_h}u, q_h \>
\\
&\le C (\|\sigma- \tau\|_{V}
 + \|u -v\|_{V} +\|p -q\| +\|P_{\Hfrak_h}u\|)
  (\|\tau_h\|_{V}+\|v_h\|_{V}+\|q_h\|).
\end{align*}
Theorem~\ref{inf-sup-d} then gives
\begin{multline}
\label{e1}
\|\sigma_h - \tau\|_{V} + \|u_h-v\|_{V}
+ \|p_h-q\|
\\*
\le C (\|\sigma - \tau\|_{V} + \|u-v\|_{V}
+ \|p-q\| + \|P_{\Hfrak_h}u\|).
\end{multline}
Using
\eqref{gap1} and the boundedness of the projection $\pi_h$ we have
\begin{equation}
 \label{e2}
\|p - q\| \le \|(I-\pi_h) p\|\le C \inf_{q\in V_h^k} \|p - q\|_{V}.
\end{equation}
Next we show that
$$
\|P_{\Hfrak_h}u\|\le \mu \|(I- \pi_h) u_\Bfrak\|_V.
$$
Now $u\perp\Hfrak^k$, so $u=u_\Bfrak + u_\perp$,
with $u_\Bfrak \in\Bfrak^k$ and $u_\perp\in\Zfrak^{k\perp}$.
Since $\Hfrak^k_h\subset\Zfrak^k$, $P_{\Hfrak_h}u_\perp=0$, and
since $\pi_h u_\Bfrak\in \Bfrak^k_h$, $P_{\Hfrak_h}\pi_h u_\Bfrak =0$.
Let $q = P_{\Hfrak_h}u/\|P_{\Hfrak_h}u\| \in \Hfrak^k_h$.
By Theorem~\ref{gap}, there exists $r\in\Hfrak^k$ (and so $r\perp\Bfrak^k$)
with $\| r \| \le 1$ and
\begin{equation*}
\| q - r \| \le \|(I - \pi_h)r\| \le 
\sup_{\substack{r \in \Hfrak^k \\ \| r \|=1}} \|(I - \pi_h)r\|.
\end{equation*}
Therefore
\begin{multline}\label{e3}
\|P_{\Hfrak_h}u \| = 
\< u_\Bfrak - \pi_h u_\Bfrak,q - r \> 
\\
\le \|(I -  \pi_h)u_\Bfrak\|\sup_{\substack{r \in \Hfrak^k \\ 
\| r \|=1}} \|(I - \pi_h)r\|\le c\,\mu \inf_{v\in V_h^k}\|P_\Bfrak u - v\|_V),
\end{multline}
since $\pi_h$ is a bounded projection.
The theorem follows from \eqref{e1}--\eqref{e3} and the triangle inequality.
\end{proof}

To implement the discrete problem, we need to be able to compute the discrete
harmonic forms.  The following lemma shows one way to do this; namely it shows
that the discrete harmonic forms can be computed as the elements of the null
space of a matrix.  For finite element approximations of the de~Rham sequence,
which are the most canonical example of this theory and which will be
discussed below, it is often possible to compute the discrete harmonic forms
more directly.  See, for example, \cite{amrouche}.
\begin{lem}
Consider the homogeneous linear system: Find $(\sigma_h,u_h)\in V^{k-1}_h\x
V^k_h$ such that
\begin{gather*}
\<\sigma_h,\tau\> = \<
 d \tau,u_h\>,  \quad \tau\in V^{k-1}_h,
\\
\<  d \sigma_h,v\>
+ \< d  u_h, d  v\> 
= 0, \quad
v\in V^k_h.
\end{gather*}
Then $(\sigma_h,u_h)$ is a solution if and only if $\sigma_h=0$
and $u_h\in\Hfrak^k_h$.
\end{lem}

\begin{proof}
Clearly $(0,u_h)$ is a solution if $u_h\in\Hfrak^k_h$.  On the other hand, if
$(\sigma_h,u_h)$ is a solution, by taking $\tau=\sigma_h$, $v=u_h$, and
combining the two equations, we find that $\|\sigma_h\|^2+\| d u_h\|^2=0$, so
that $\sigma_h=0$ and $ d u_h=0$.  Then the first equation implies that $\< d
\tau,u_h\>=0$ for all $\tau\in V^{k-1}_h$, so indeed $u_h\in\Hfrak^k_h$.
\end{proof}

\subsection{Improved error estimates}
\label{subsec:improved}
Suppose we have a family of subcomplexes $(V_h,d)$ of the domain complex
$(V,d)$ of a closed Hilbert complex, parametrized by $h$ with
uniformly bounded cochain projections $\pi_h$.  Assuming also that the subspaces
$V^k_h$ are approximating in $V^k$ in the sense of \eqref{approximating}, we can conclude from
Theorem~\ref{qo1cor} that $\sigma_h\to \sigma$, $u_h\to u$, and $p_h\to p$ as
$h\to 0$ (in the norms of $V^{k-1}$ and $V^k$).  In other words, the Galerkin method
for the Hodge Laplacian is convergent.

The \emph{rate of convergence} will depend on the approximation properties of
the subspaces $V^k_h$, the particular component considered
($\sigma_h$, $u_h$, or $p_h$), the norm in which we measure the error (e.g.,
$W$ or $V$), as well as properties of the data $f$ and the corresponding
solution.  For example, in Section~\ref{sec:FEDF} we will consider approximation of the
de~Rham complex using various subcomplexes for which the spaces $V^k_h$ consist of
piecewise polynomial differential forms with respect to a triangulation
$\T_h$ of the domain with mesh size $h$.  The space $W^k$ is the space of
$L^2$ differential $k$-forms in this case.  One possibility we consider for the
solution of the Hodge Laplacian for $k$-forms using the mixed formulation is to take
subspaces $\P_{r+1}\Lambda^{k-1}(\T_h)$ and $\P_r\Lambda^k(\T_h)$.  The space $\P_r\Lambda^k(\T_h)$,
which is defined in Section~\ref{subsec:dof}, consists piecewise of all $k$-forms
of polynomial degree at most $r$.  Assuming that the solution $u$
to the Hodge Laplacian is sufficiently smooth, application of Theorem~\ref{qo1cor}
will give, in this case,
\begin{equation*}
 \|\sigma-\sigma_h\|_{V} + \|u-u_h\|_{V} + \|p - p_h\| =O(h^r).
\end{equation*}
Approximation theory tells us that this rate is the best possible for $\|u-u_h\|_{V}$,
but we might hope for a faster rate for $\|u-u_h\|$ and for
$\|\sigma-\sigma_h\|_{V}$ and $\|\sigma-\sigma_h\|$.

In order to obtain improved error estimates,
we make two addition assumptions, first that the complex $(W,d)$ satisfies the
compactness property introduced at the end of Section~\ref{subsec:basicdefs},
and second, that the cochain projection is bounded not only in $V$ but in
$W$:
\begin{itemize}
 \item The intersection $V^k\cap V^*_k$ is a dense subset of $W^k$ with 
compact inclusion.
\item
The cochain projections $\pi^k_h$ are bounded in $\Lin(W^k,W^k)$ 
uniformly with respect to $h$.
\end{itemize}
The second property implies that $\pi_h$ extends to a bounded
linear operator $W^k\to V^k_h$.  Since the subspaces $V^k_h$ are approximating in $W^k$
as well as in $V^k$ (by density), it follows that
$\pi_h$ converges pointwise to the identity in $W^k$.  Finally, note that if we
are given a $W$-bounded cochain projection mapping $W^k\to V^k_h$, the restrictions to $V^k$
define a $V$-bounded cochain projection.

Next, we note that on the Hilbert space $V^k\cap V_k^*$, the inner product given by 
\[
\< u,v \>_{V\cap V^*}
:=\< d^*u, d^* v\>
+ \< d u,  d v\> + \< P_\Hfrak u, P_\Hfrak v \>
\]
is equivalent to the usual intersection inner product which is the sum of the
inner products for $V^k$ and $V^*_k$.  This can be seen by Hodge decomposing
$u$ as $P_\Bfrak u+P_\Hfrak u+P_{\Bfrak^*}u$, and using the Poincar\'e bound
$\|P_{\Bfrak^*}u\|\le c\|dP_{\Bfrak^*}u\|=c\|du\|$ and the analogous
bound $\|P_\Bfrak u\|\le
c\|d^*P_\Bfrak u\|=c\|d^*u\|$.  Now
$K$ maps $W^k$ boundedly into $V^k\cap V^*_k$ and satisfies
\[
\< K f,v \>_{V\cap V^*} =  \< f, v-P_{\Hfrak}v \>, \quad f\in W^k, 
\ v\in V^k\cap V^*_k.
\]
In other words, $K \in \Lin(W^k, V^k\cap V^*_k)$ is the adjoint of the
operator $(I-P_{\Hfrak})\mathcal I\in \Lin(V^k\cap V^*_k,W^k)$ where $\mathcal
I\in\Lin(V^k\cap V^*_k,W^k)$ is the compact inclusion operator.  Hence $K$ is
a compact operator $W^k\to V^k\cap V^*_k$ and, a fortiori, compact as an
operator from $W^k$ to itself.  As an operator on $W^k$, $K$ is also
self-adjoint, since
$$
\<f,Kg\>=\<f-P_{\Hfrak}f,Kg\>=\<dd^*Kf+d^*dKf,Kg\>
=\<d^*Kf,d^*Kg\>+\<dKf,dKg\>
$$
for all $f,g\in W^k$.  Furthermore, if we follow $K$ by one of the bounded
operators $d:V^k\to W^{k+1}$ or $d^*:V^*_k\to W^{k-1}$, the compositions $dK$
and $d^*K$ are also compact operators from $W^k$ to itself.  Since we have
assumed the compactness property, $\dim \Hfrak^k<\infty$, and so
$P_{\Hfrak^k}$ is also a compact operator on $W^k$.
Define
\begin{gather*}
 \delta = \delta_h^k = \|(I -  \pi_h) K \|_{\Lin(W^k, W^k)},\quad
\mu = \mu_h^k
= \|(I -  \pi_h) P_\Hfrak \|_{\Lin(W^k, W^k)},
\\
\eta=\eta_h^k=\max_{j=0,1}[\|(I -  \pi_h) d K \|_{\Lin(W^{k-j}, W^{k-j+1})},
\|(I -  \pi_h) d^* K \|_{\Lin(W^{k+j}, W^{k+j-1})}].
\end{gather*}
(Note that $\mu$ already appeared in Theorem~\ref{qo1cor}.)
Recalling that composition on the right with a compact operator
converts pointwise convergence to norm convergence, we see that
$\eta,\delta,\mu \rightarrow 0$ as $h \rightarrow 0$.
In the applications in Sections~\ref{sec:FEDF} and \ref{sec:variations}, the
spaces $\Lambda^k_h$ will consist of piecewise polynomials.  We will
then have $\eta=O(h)$, $\delta=O(h^{\min(2,r+1)})$, and $\mu=O(h^{r+1})$
where $r$ denotes
the largest degree of complete polynomials in the space $\Lambda^k_h$.

In Theorem~\ref{qo1cor}, the error estimates were given in terms of the best approximation
error afforded by the subspaces in the $V$ norm.  The improved error estimates will be
in terms of the best approximation error in the $W$ norm, for which we introduce the notation
$$
 E(w)=E^k_h(w) = \inf_{v\in V^k_h}\|w-v\|, \quad w\in W^k.
$$
The following theorem gives improved the improved error estimates.  Its proof
incorporates a variety
of techniques developed in the numerical analysis literature in the last decades, e.g.,
\cite{falk-osborn,douglas-roberts,acta}.
\begin{thm}
\label{impest}
Let $(V,d)$ be the domain complex of a closed Hilbert complex $(W,d)$ satisfying the compactness property,
and $(V_h,d)$ a family of subcomplexes parametrized by $h$ and admitting
uniformly $W$-bounded cochain projections. Let $(\sigma,u,p)\in V^{k-1} \x V^k \x \Hfrak^k$
be the solution of problem
\eqref{wfhc} and $(\sigma_h, u_h, p_h) \in V^{k-1}_h \x V^k_h \x \Hfrak^k_h$
the solution of problem \eqref{wfhhc}.  Then for some constant
$C$ independent of $h$ and $(\sigma,u,p)$, we have
\begin{gather}
\label{estds}
\|d(\sigma-\sigma_h)\| \le C E(d \sigma),
\\
\label{ests}
\|\sigma - \sigma_h\| \le C[E(\sigma) + \eta E(d \sigma)],
\\
\label{estp}
\|p-p_h\| \le C [E(p) + \mu E(d \sigma)].
\\
\label{estdu}
\|d(u-u_h)\| \le C(E(du) + \eta[E(d \sigma) + E(p)]),
\\
\label{estu}
\|u-u_h\| \le C(E(u) + \eta [E(du) +E(\sigma)]
+ (\eta^2 + \delta) [E(d \sigma) + E(p)] + \mu E(P_{\Bfrak}u)]).
\end{gather}
\end{thm}
We now develop the proof of Theorem~\ref{impest} in a series of lemmas.
\begin{lem}
\label{prelim}
Let $v_h \in \Zfrak^{k\perp}_h$ and $v = P_{\Bfrak^*}v_h$. Then
\begin{equation*}
\|v- v_h\| \le  \|(I-\pi_h^k)v\|\le \eta \|dv_h\|.
\end{equation*}
\end{lem}
\begin{proof}
Since  $\pi_h v - v_h \in \Zfrak_h^k \subset
\Zfrak^k$, $v_h - v\perp\pi_h v - v_h$, so, by the Pythagorean theorem,
$\| v_h - v \| \le \|(I - \pi_h)v\|$.  The second inequality
holds since $v=d^*Kdv_h$.
\end{proof}
\begin{lem}
\label{sig-ests}
The estimates \eqref{estds} and \eqref{ests} hold.  Moreover
\begin{equation}
\label{estpbu}
\|P_{\Bfrak_h} (u -u_h)\| 
\le C[\eta E(\sigma) + (\eta^2 + \delta) E(d \sigma)].
\end{equation}
\end{lem}
\begin{proof}
Since $d\sigma_h=P_{\Bfrak_h} f =P_{\Bfrak_h}P_{\Bfrak} f =P_{\Bfrak_h} d \sigma$, we have
\begin{equation*}
\|d(\sigma- \sigma_h)\| =\|(I- P_{\Bfrak_h})d \sigma\|
\le \|(I - \pi_h)d \sigma\| \le C E(d \sigma),
\end{equation*}
giving \eqref{estds}.  To prove \eqref{ests} we write
\begin{equation*}
\sigma=d^* K d\sigma= d^* K (I - P_{\Bfrak_h}) d \sigma
+ d^* K P_{\Bfrak_h} d \sigma=:\sigma^1+\sigma^2.
\end{equation*}
Taking $\tau=\pi_h \sigma^2- \sigma_h$ in \eqref{wfhc} and \eqref{wfhhc}, we obtain
\begin{equation*}
\langle\sigma -\sigma_h,\pi_h \sigma_2- \sigma_h \rangle
= \langle d (\pi_h \sigma_2- \sigma_h),u-u_h\rangle =0.
\end{equation*}
Hence,
\begin{equation*}
\|\sigma -\sigma_h\| \le \|\sigma- \pi_h \sigma^2 \|
\le \|(I - \pi_h)\sigma \| + \|\pi_h \sigma^1 \|
\le C E(\sigma)| + C\|\sigma^1 \|.
\end{equation*}
Since $\sigma^1\in\Bfrak^*_k$,
\begin{multline*}
\|\sigma^1\|^2 = \langle Kd\sigma^1, d \sigma^1 \rangle
= \langle Kd\sigma^1, (I- P_{\Bfrak_h})d \sigma \rangle
= \langle (I- P_{\Bfrak_h})Kd\sigma^1 , (I- P_{\Bfrak_h})d \sigma \rangle
\\
\le \|(I- \pi_h) K d \sigma^1\| \|(I- \pi_h) d \sigma\| 
\le C \eta \|\sigma^1\| E(d \sigma).
\end{multline*}
Then \eqref{ests} follows from the last two estimates.

Let $e=P_{\Bfrak_h}( u- u_h)$.  To estimate $e$, we set
$w=Ke$, $\phi=d^*w$, $w_h=K_he$, $\phi_h=d^*_h w_h$.  Then $d\phi=d\pi_h\phi=d\phi_h=e$, so
$\pi_h\phi-\phi_h\in \Zfrak^k_h$, and so is orthogonal to $\phi-\phi_h$.  Thus
$\|\phi-\phi_h\| \le \|(I- \pi_h) \phi\| = \|(I- \pi_h)d^*Ke\|$. 
Then
\begin{align*}
\|e\|^2 
&= \langle d \phi_h, e \rangle
= \langle d \phi_h, u -u_h\rangle = 
\langle \sigma - \sigma_h, \phi_h \rangle 
\\
&= \langle \sigma - \sigma_h, \phi_h - \phi\rangle 
+ \langle d (\sigma - \sigma_h) ,w \rangle
\\
&= \langle \sigma- \sigma_h, \phi_h - \phi\rangle 
+ \langle (I- P_{\Bfrak_h}) d \sigma, (I- P_{\Bfrak_h}) w \rangle
\\
&\le \|\sigma - \sigma_h\| \|\phi_h - \phi\| 
+ \|(I- \pi_h)d \sigma\| \| (I- \pi_h) w\|
\\
&\le \|\sigma - \sigma_h\| \|(I- \pi_h)d^*Ke\|
+ C E(d \sigma) \| (I- \pi_h) Ke\|
\\
&\le [\eta \|\sigma-\sigma_h\| 
+  C E(d \sigma) \delta] \|e\|.
\end{align*}
Combining with \eqref{ests} we get \eqref{estpbu}.
\end{proof}

\begin{lem}
\label{pest}
Estimate \eqref{estp} holds and, moreover, 
$\|P_{\Hfrak_h} u\| \le C \mu E(P_{\Bfrak} u)$.
\end{lem}
\begin{proof}
The second estimate is just \eqref{e3}.  Using the Hodge decomposition of
$f$ and the fact that $P_{\Hfrak_h}\Bfrak^*_k=0$, we have
$P_{\Hfrak_h}p-p_h=P_{\Hfrak_h}P_{\Hfrak}f - P_{\Hfrak_h} f=P_{\Hfrak_h}f_{\Bfrak}$,
where $f_{\Bfrak} = P_{\Bfrak} f$. Therefore,
\begin{equation*}
\|p-p_h\| = \|p-P_{\Hfrak_h}p\|+\|P_{\Hfrak_h}f_{\Bfrak}\|
\le C E(p) + \|P_{\Hfrak_h}f_{\Bfrak}\|,
\end{equation*}
by \eqref{gap1}.
Applying \eqref{gap2}  we get
\begin{multline*}
\|P_{\Hfrak_h}f_{\Bfrak}\|^2 = \langle f_{\Bfrak} - \pi_h f_{\Bfrak},
P_{\Hfrak_h}f_{\Bfrak} \rangle
= \langle f_{\Bfrak} - \pi_h f_{\Bfrak}, P_{\Hfrak_h}f_{\Bfrak}  - P_{\Hfrak} P_{\Hfrak_h}f_{\Bfrak}
\rangle
\\
\le C \|(I-\pi_h)f_{\Bfrak}\| \|(I- \pi_h) P_{\Hfrak} P_{\Hfrak_h}f_{\Bfrak}\|
\le C E(d \sigma) \mu \|P_{\Hfrak_h}f_{\Bfrak}\|.
\end{multline*}
Combining these results completes the proof of the lemma.
\end{proof}

\begin{lem}\label{duest}
The estimate \eqref{estdu} holds.
\end{lem}
\begin{proof}
From \eqref{wfhc} and \eqref{wfhhc},
\begin{equation}
\label{orthog2}
\langle d (\sigma -\sigma_h),v_h\rangle
+ \langle
d  (u-u_h),d  v_h\rangle + \langle v_h, p-p_h \rangle 
= 0 \qquad v_h\in V_h^k.
\end{equation}
Choose $v_h = P_{\Bfrak^*_h}(\pi_hu - u_h)$
in \eqref{orthog2} and set $v = d^* K d v_h=P_{\Bfrak^*}v_h$.   Using Lemmas~\ref{prelim}
and \eqref{gap1}, we get
\begin{multline*}
\langle d(u-u_h), d(\pi_h u - u_h) \rangle = 
\langle d(u-u_h), d v_h \rangle
= - \langle d(\sigma - \sigma_h) + (p-p_h), v_h \rangle
\\
= - \langle d(\sigma - \sigma_h) + (p-P_{\Hfrak_h}p),
v_h -v \rangle
\le [\|d(\sigma - \sigma_h)\| + \|p-P_{\Hfrak_h}p\|] \|v_h -v\|
\\
\le C [E(d \sigma) + E(p)]\eta \|d v_h\|
\le C \eta [E(d \sigma) + E(p)] \|d(\pi_h u - u_h)\|.
\end{multline*}
Hence,
\begin{align*}
\|d(\pi_h u-u_h)\|^2 &= \langle d(\pi_h u -u), d(\pi_h u-u_h) \rangle
+ \langle d(u -u_h), d(\pi_h u-u_h) \rangle
\\
&\le \{\|d(\pi_h u -u)\| + C \eta [E(d \sigma) + E(p)]\}\|d(\pi_h u-u_h)\|
\\
&\le C\{ E(du) + \eta [E(d \sigma) + E(p)]\} \|d(\pi_h u-u_h)\|.
\end{align*}
The result follows by the triangle inequality.
\end{proof}

\begin{lem}\label{hests}
$\|P_{\Bfrak^*_h} (u - u_{h})\|  \le 
C\{E(u) + \eta E(du) + (\eta^2 + \delta) [E(d \sigma) + E(p)])\}$.
\end{lem}
\begin{proof}
Again letting $v_h = P_{\Bfrak^*_h}(\pi_hu - u_h)$ and $v=P_{\Bfrak^*}v_h$, we
observe that
\begin{equation*}
\|P_{\Bfrak^*_h} (u - u_{h})\|
\le \|P_{\Bfrak^*_h} (u - \pi_h u)\| + \|v_h\|
\le \|u - \pi_h u\| + \|v_h\| \le CE(u) + \|v_h\|.
\end{equation*}
Next
\begin{equation*}
\|v_h\|^2 = \langle v_h - v, v_h \rangle + \langle v, \pi_h u- u_h \rangle
= \langle v_h - v, v_h \rangle + \langle v, \pi_h u -u \rangle
+  \langle v, u -u_h \rangle.
\end{equation*}
Then using Lemmas~\ref{prelim} and \ref{duest}, we get
\begin{multline*}
\langle v_h - v, v_h \rangle + \langle v, \pi_h u -u \rangle
\le \|v_h - v\| \|v_h\| + \|v\| \|(I-\pi_h)u\|
\\
\le [\eta \|d v_h\| + \|(I-\pi_h)u\|] \|v_h\|
\le [\eta \|d(\pi_hu-u_h)\| + C E(u)] \|v_h\|.
\end{multline*}
We next estimate the term $\langle v, u -u_h \rangle$.
Since $dv = dv_h$ and $Kv \in \Zfrak^{k \perp}$, we get
\begin{align*}
\langle v, u -u_h \rangle &= 
\langle K dv_h, d(u-u_h) \rangle = \langle K dv , d(u-u_h) \rangle
= \langle d K v , d(u-u_h) \rangle 
\\
&= \langle (I-\pi_h)d K v , d(u-u_h) \rangle 
+ \langle d \pi_h K v , d(u-u_h) \rangle
\\
&= \langle (I-\pi_h)d K v , d(u-u_h) \rangle 
+ \langle d P_{\Bfrak^*_h} \pi_h K v , d(u-u_h) \rangle.
\end{align*}
Now
\begin{multline*}
\langle (I-\pi_h)d K v , d(u-u_h) \rangle 
\le \eta \|v\| \|d(u-u_h)\| 
\le C \eta \|d(u-u_h)\|  \|v_h\|,
\end{multline*}
while
\begin{align*}
\langle d P_{\Bfrak^*_h} \pi_h K v , d(u-u_h) \rangle
&= - \langle d(\sigma - \sigma_h) + (p-p_h), 
P_{\Bfrak^*_h} \pi_h Kv -Kv \rangle
\\
&= - \langle d(\sigma - \sigma_h) + (p-P_{\Hfrak_h}p), 
P_{\Bfrak^*_h} \pi_h K v -Kv \rangle
\\
&\le [\|d(\sigma-\sigma_h)\|
+ \|p-P_{\Hfrak_h}p\|] \|P_{\Bfrak^*_h} \pi_h K v -Kv\|.
\end{align*}
But
\begin{multline*}
\|P_{\Bfrak^*_h} \pi_h K v -Kv\|^2
= \langle P_{\Bfrak^*_h} \pi_h K v -Kv, \pi_h Kv - Kv \rangle
\\
\le \|P_{\Bfrak^*_h} \pi_h K v -Kv\| \|(I- \pi_h) Kv\|
\le \|P_{\Bfrak^*_h} \pi_h K v -Kv\| \delta \|v\|.
\end{multline*}
Hence,
\begin{equation*}
\|P_{\Bfrak^*_h} \pi_h K v -Kv\|
\le C \delta \|v_h\|.
\end{equation*}

Combining these results, and using the previous lemmas, we obtain
\begin{multline*}
\|v_h\| \le C( \eta \|d(\pi_hu-u_h)\| + \eta \|d(u-u_h)\| + E(u) 
+ \delta [\|d(\sigma - \sigma_h)\| + \|p-P_{\Hfrak_h}p\|])
\\
\le C\{E(u) + \eta E(du) + (\eta^2 + \delta) [E(d \sigma) + E(p)]\}.
\end{multline*}
The final result of the lemma follows immediately.
\end{proof}

It is now an easy matter to prove \eqref{estu} and so complete
the proof Theorem~\ref{impest}. We write
\begin{equation*}
u -u_h = (u-P_hu)+ P_{\Bfrak_h} (u - u_h) + P_{\Hfrak_h^k} u
+ P_{\Bfrak^*_h} (u - u_h),
\end{equation*}
and so \eqref{estu} follows from Lemmas~\ref{sig-ests},
\ref{pest}, and \ref{hests}.

To get a feeling for these results, we return to the example mentioned
earlier, where $V^{k-1}_h=\P_{r+1}\Lambda^{k-1}(\T_h)$ and
$V^k_h=\P_r\Lambda^k(\T_h)$ are used to approximate the $k$-form Hodge
Laplacian with some $r\ge 1$.  If the domain is convex, we may apply
elliptic regularity to see that $Kf$ belongs to the Sobolev space
$H^2\Lambda^k$ for $f\in L^2\Lambda^k$, so $dKf\in H^1\Lambda^{k+1}$
and $d^*Kf\in H^1\Lambda^{k-1}$, and then standard approximation
theory shows that $\eta=O(h)$, $\delta=O(h^2)$, and $\mu=O(h^2)$.
From Theorem~\ref{impest}, we then obtain that
\begin{equation*}
 \|\sigma-\sigma_h\|+h\|d(\sigma-\sigma_h)\|
+h \|u-u_h\| + h \|p-p_h\| +h^2 \|d(u-u_h)\|=O(h^{r+2}),
\end{equation*}
assuming the solution $u$ is sufficiently smooth.  That is, all
components converge with the optimal order possible given the degree
of polynomial approximation.

Finally we note a corollary of Theorem~\ref{impest}, which will be useful in the
analysis of the eigenvalue problem.
\begin{cor}
  \label{oper-conv}
Under the assumptions of Theorem~\ref{impest}, there exists a constant $C$ such that
\begin{gather*}
\|K-K_hP_h\|_{\Lin(W^k,W^k)}\le C(\eta^2 + \delta + \mu), \\
\|dK-dK_hP_h\|_{\Lin(W^k,W^{k+1})}+
\|d^*K-d^*_hK_hP_h\|_{\Lin(W^k,W^{k-1})}\le C\eta.
\end{gather*}
Therefore all three operator norms converge to zero with $h$.
\end{cor}
\begin{proof}
 Let $f\in W^k$ and set $u=Kf$, $\sigma=d^*Kf$, $p=P_\Hfrak f$, and $u_h=K_hP_hf$,
$\sigma_h=d^*_hK_hP_hf$.
The desired bounds on $(K-K_hP_h)f=u-u_h$, $d(K-K_hP_h)f=d(u-u_h)$, and
$(d^*K-d^*_hK_hP_h)f=\sigma-\sigma_h$ follow from Theorem~\ref{impest}, since
$$
E(u)\le\delta\|f\|,\quad
E(du)+E(\sigma)\le \eta\|f\|,\quad 
E(d\sigma)+E(p)+E(P_\Bfrak u)\le \|f\|.
$$
\end{proof}

We end this section with an estimate of the difference between the true Hodge decomposition
and the discrete Hodge decomposition of an element of $V^k_h$.  While this result
is not needed in our approach, such an estimate was central to the estimation of the
eigenvalue error for the Hodge Laplacian using Whitney forms made in
\cite{dodziuk-patodi}.
\begin{lem}
\label{dodziuk}
Let $P_{\Bfrak}u + P_{\Hfrak}u + P_{\Bfrak^*}u$ and 
$P_{\Bfrak_h}u + P_{\Hfrak_h}u + P_{\Bfrak_h^*}u$ denote the Hodge and discrete
Hodge decompositions of $u \in V_h^k$. Then
\begin{gather*}
\|P_{\Bfrak^*}u - P_{\Bfrak_h^*}u\| +
\|P_{\Hfrak}u - P_{\Hfrak_h}u\| +
\|P_{\Bfrak}u - P_{\Bfrak_h}u\| \le C (\eta +\mu)\| u\|_V.
\end{gather*}
\end{lem}
\begin{proof}
Let $v_h = P_{\Bfrak_h^*}u$ and $v = P_{\Bfrak^*} P_{\Bfrak_h}u = P_{\Bfrak^*} u$.
The first estimate follows immediately from Lemma~\ref{prelim}. To obtain
the second estimate, we write
\begin{equation*}
P_{\Hfrak}u - P_{\Hfrak_h} u = P_{\Hfrak}(P_{\Hfrak_h} u + P_{\Bfrak_h^*}u)
- P_{\Hfrak_h} u = (P_{\Hfrak} -I) P_{\Hfrak_h} u
+ P_{\Hfrak}(P_{\Bfrak_h^*}u - P_{\Bfrak^*}u).
\end{equation*}
The second estimate now follows directly from \eqref{gap2} and the first
estimate and the final estimate follows from the first two by the triangle
inequality.
\end{proof}
Applied to the case of Whitney forms, both $\eta$ and $\mu$ are $O(h)$, and so this
result improves the $O(h |\log h|)$ estimate of \cite[Theorem~2.10]{dodziuk-patodi}.

\subsection{The eigenvalue problem}\label{subsec:eigenv}
The purpose of this section is to study the eigenvalue problem associated to
the abstract Hodge Laplacian \eqref{wfhc}. As in the previous section, we will
assume $(W,d)$ is a Hilbert complex satisfying the compactness property and that the cochain
projections $\pi_h^k$ are bounded in $\Lin(W^k,W^k)$, uniformly in $h$.  A
pair $(\lambda,u) \in \R \x V^k$, where $u \neq 0$, is referred to as an
eigenvalue/eigenvector of the problem \eqref{wfhc} if there exists $(\sigma,p)
\in V^{k-1}\x \Hfrak^k$ such that
\begin{equation}\label{eigenv-pr}
\begin{aligned}
\<\sigma,\tau\> - \<
d \tau,u\> &=0,   &&\tau\in V^{k-1},
\\*
\< d \sigma,v\>
+ \<
d  u,d  v\> + \< v, p \> 
&= \lambda \< u, v \>
&&v\in V^k,
\\*
\< u, q \> &= 0, 
&&q\in \Hfrak^k.
\end{aligned}
\end{equation}
In operator terms, $u = \lambda Ku$, $\sigma = d^*u$, $p = 0$.
Note that it follows from this system that
\[
\lambda \| u \|^2 = \| d u \|^2 + \| d^* u \|^2 > 0,
\]
so $Ku = \lambda^{-1}u$. Since the
operator $K \in \Lin(W^k,W^k)$ is compact and self-adjoint, we can conclude that the problem
\eqref{eigenv-pr} has at most a countable set of eigenvalues, each of finite
multiplicity.  We denote these by
\[
0 < \lambda_1 \le \lambda_2 \le \ldots 
\]
where each eigenvalue is repeated according to its multiplicity.  Furthermore,
when $W^k$ is infinite dimensional, we have $\lim_{j \rightarrow \infty}
\lambda_j = \infty$.  We denote by $\{v_i\}$ a corresponding orthonormal basis
of eigenvectors for $W^k$.

The corresponding discrete eigenvalue problem takes the form
\begin{equation}\label{eigenv-pr-h}
\begin{aligned}
\<\sigma_h,\tau\> - \<
 d \tau,u_h\> &= 0,  & \tau&\in V^{k-1}_h,
\\*
\<  d \sigma_h,v\>
+ \< d  u_h, d  v\> + \< v, p_h\>
&= \lambda_h\< u_h,v\>, &
v&\in V^k_h,
\\* 
\< u_h, q \>  &=0,
&q&\in \Hfrak_h^k,
\end{aligned}
\end{equation}
where $\lambda_h \in \R$, and $(\sigma_h,u_h,p_h) \in V^{k-1}_h \x V^{k}_h \x
\Hfrak_h^k$, with $u_h \neq 0$.  As above, we can conclude that $p_h = 0$,
$\lambda_h >0$, and that $\lambda_h^{-1}$ is an eigenvalue for the operator
$K_h$, i.e. $K_h u_h = \lambda_h^{-1}u_h$.  We denote by
\[
0 < \lambda_{1,h} \le \lambda_{2,h} \le \ldots \le \lambda_{N_h,h},
\]
$N_h=\dim V^k_h$, the eigenvalues of problem \eqref{eigenv-pr-h}, repeated
according to multiplicity, and by $\{v_{i,h}\}$ a corresponding
$W^k$-orthonormal eigenbasis for $V^k_h$.

Next, we will study how the discrete eigenvalue problem \eqref{eigenv-pr-h}
converges to the eigenvalue problem \eqref{eigenv-pr}, i.e., how the
eigenvalues and eigenvectors of the operator $K_h$ converge to those of $K$.
We let $E_i$ and $E_{i,h}$ denote the 1-dimensional spaces spanned by $v_i$
and $v_{i,h}$.  For every positive integer $j$, let $m(j)$ denote the number
of eigenvalues less than or equal to the $j$th \emph{distinct} eigenvalue of
the Hodge--Laplace problem \eqref{eigenv-pr}.  Thus $\sum_{i=1}^{m(j)} E_i$ is
the space spanned by the eigenspaces associated to the first $j$ distinct
eigenvalues and does not depend on the choice of the eigenbasis.

The discrete eigenvalue problem \eqref{eigenv-pr-h} is said to {\em converge}
to the exact eigenvalue problem \eqref{eigenv-pr} if, for any $\epsilon >0$
and integer $j>0$, there exists a mesh parameter $h_0 > 0$ such that, for all
$h \le h_0$, we have
\begin{equation}\label{evalconv}
\max_{1 \le i \le m(j)} |\lambda_i - \lambda_{i,h}| \le \epsilon,
\quad \text{and} \quad 
\gap\Biggl(\sum_{i=1}^{m(j)}E_i, \sum_{i=1}^{m(j)}E_{i,h}\Biggr)
\le \epsilon,
\end{equation}
where the gap between two subspaces $E$ and $F$ of a Hilbert space is
defined by \eqref{gapdef}.  The motivation for this rather strict concept of convergence
is that it not only implies that each eigenvalue is approximated by the
appropriate number of discrete eigenvalues, counting multiplicities, and that
the eigenspace is well-approximated by the corresponding discrete eigenspaces,
but it also rules out spurious discrete eigenvalues and eigenvectors.  In
particular, this rules out the behavior exemplified in Figure~\ref{fg:f4} in
Section~\ref{sec:femethod}.

A key result in the perturbation theory of linear operators is that, for
eigenvalue problems of the form we consider, corresponding to the bounded
compact self-adjoint Hilbert space operators $K$ and $K_h$, convergence of the
eigenvalue approximation holds if the operators $K_hP_h$ converge to $K$ in
$\Lin(W^k,W^k)$.  This result, widely used in the theory of mixed finite
element eigenvalue approximation, essentially follows from the contour
integral representation of the spectral projection, and can be extracted from
\cite[Chapters III, IV]{kato} or \cite[Section~7]{babuska-osborn}.  For a
clear statement, see \cite{boffi-ima}.  In fact, as observed in
\cite{boffi-brezzi-gastaldi}, this operator norm convergence is sufficient as
well as necessary for obtaining convergence of the eigenvalue approximations
in the sense above.  As a consequence of Corollary~\ref{oper-conv}, we therefore
obtain the following theorem.
\begin{thm}
\label{ev-conv}
Let $(V,d)$ be the domain complex of a closed Hilbert complex $(W,d)$ satisfying the compactness property,
and $(V_h,d)$ a family of subcomplexes parametrized by $h$ and admitting
uniformly $W$-bounded cochain projections. Then
the discrete eigenvalue problems \eqref{eigenv-pr-h} 
converge to the problem \eqref{eigenv-pr}, i.e., \eqref{evalconv} holds.
\end{thm}

It is also possible to use the theory developed above to obtain
rates of convergence for the approximation of eigenvalues, based on the
following result
(Theorem 7.3 of \cite{babuska-osborn}).
Here, we consider only the case of a simple eigenvalue, but  with a small modification
the results extend to eigenvalues of positive multiplicity.
\begin{lem}
\label{b-o-lem}
If $\lambda$ is a simple eigenvalue and $u$ a normalized eigenvector,
then
\begin{equation*}
|\lambda^{-1} - \lambda_h^{-1}| \le C\{|\langle (K-K_hP_h)u,u \rangle| 
+ \|(K-K_hP_h)u\|_V^2\}.
\end{equation*}
\end{lem}
\begin{thm}
\label{ev-basicest}
Let $(V,d)$ be the domain complex of a closed Hilbert complex $(W,d)$ satisfying the compactness property,
and $(V_h,d)$ a family of subcomplexes parametrized by $h$ and admitting
uniformly $W$-bounded cochain projections.
Let $\lambda$ be a simple eigenvalue, $u$ the corresponding eigenvector,
and $\lambda_h$ be the corresponding discrete eigenvalue.  Let $w = Ku$,
$\phi= d^*w$ denote the solution of the mixed formulation of the Hodge Laplace
problem with source term $u$, and let $w_h = K_hP_hu$, $\phi_h = d_h^* w_h$
denote the corresponding discrete solution. Then
\begin{equation*}
|\lambda^{-1} - \lambda_h^{-1}| 
\le C( \|w-w_h\|_V^2 + \|\phi-\phi_h\|^2 + \left| \langle d(\phi-\phi_h),
w-w_h \rangle \right|).
\end{equation*}
\end{thm}
\begin{proof}
To estimate the first term in Lemma~\ref{b-o-lem}, we write
\begin{align*}
\langle (K-K_hP_h)u&,u\rangle = \langle Ku,u \rangle 
- \langle K_h P_h u, P_h u \rangle
\\
&=\langle w, (d^* d + dd^*) w \rangle - 
\langle w_h, (d_h^*d  + d d_h^*)w_h \rangle 
\\
&=\|dw\|^2 + \|\phi\|^2 - \|d w_h\|^2 - \|\phi_h\|^2
\\
&= \|d(w-w_h)\|^2 + 2 \langle d(w-w_h), d w_h \rangle
- \|\phi-\phi_h\|^2 + 2  \langle \phi- \phi_h, \phi \rangle
\\
&= \|d(w-w_h)\|^2 - \|\phi-\phi_h\|^2 - 2 \langle d(\phi- \phi_h), w_h \rangle
+ 2 \langle d(\phi- \phi_h), w \rangle
\\
&= \|d(w-w_h)\|^2 - \|\phi-\phi_h\|^2 
+ 2 \langle d(\phi- \phi_h), w -w_h \rangle,
\end{align*}
where we have used the orthogonality condition
\begin{equation*}
\langle d (\phi - \phi_{h}),v\rangle
+ \langle d  (w - w_{h}), d  v\rangle = 0, \quad v \in V^k_h
\end{equation*}
in the second last line above.  The theorem then follows from this estimate
and the fact that $\|(K-K_hP_h)u\|_V^2 = \|w-w_h\|_V^2 $.
\end{proof}

Order of convergence estimates now follow directly from
Theorem~\ref{impest}, with $f$ replaced by $u$, $u$ replaced by $w$,
and $\sigma$ replaced by $\phi$.  In the example following
Theorem~\ref{impest}, i.e., $V^{k-1}_h=\P_{r+1}\Lambda^{k-1}(\T_h)$
and $V^k_h=\P_r\Lambda^k(\T_h)$,
we saw that
\begin{equation*}
\|\phi-\phi_h\| + h \|d(\phi-\phi_h)\| + h \|w-w_h\|
+ h^2 \|d(w-w_h)\|=O(h^{r+2}),
\end{equation*}
as long as the domain is convex and the solution $w=Ku$ sufficiently smooth.  Inserting these
results into Theorem~\ref{ev-basicest}, we find that
the eigenvalue error
$|\lambda-\lambda_h|=O(h^{2r})$, double the rate achieved for the source
problem.  As another example, one can check that for the Whitney forms,
$V^{k-1}_h=\P_1^-\Lambda^{k-1}(\T_h)$,
$V^k_h=\P_1^-\Lambda^k(\T_h)$ we get $|\lambda-\lambda_h|=O(h^2)$, improving on
the $O(h|\log h|)$ estimate of \cite{dodziuk-patodi}.

\begin{remark}
It has long been observed that for mixed finite element approximation of
eigenvalue problems, stability and approximability alone, while sufficient for
convergence of approximations of the source problem, is not sufficient for
convergence of the eigenvalue problem.  An extensive literature developed in
order to obtain eigenvalue convergence, and a wide variety of additional
properties of the finite element spaces have been defined and hypothesized.
In particular, the first convergence results for the the important case of
electromagnetic eigenvalue problems, were obtained by Kikuchi based on the
\emph{discrete compactness property} \cite{kikuchi,boffi-dcp}, and the more
recent approach by Boffi and collaborators emphasize the Fortid property
\cite{boffi-ima}.
In our context, the discrete compactness property says that whenever $v_h\in\Zfrak^{k\perp}_h$
is uniformly bounded in $V^k$, then there exists a sequence $h_i$ converging to
zero such that $v_{h_i}$ converges in $W^k$, and the Fortid property says that
$\lim_{h\to0}\|I-\pi^k_h\|_{\Lin(V^k\cap V^*_k,W^k)}=0$.
These additional properties play no role in the theory as
presented here.  We prove eigenvalue convergence under the same sort of
assumptions we use to establish stability and convergence for the source
problem, namely subcomplexes, bounded cochain projections, and
approximability, but we require boundedness in $W$ for eigenvalue convergence, while
for stability we only require boundedness in $V$.
We believe that subcomplexes with bounded cochain projections provide
an appropriate framework for the
analysis of the eigenvalue problem, since, as far as we know, these properties
hold in all examples where eigenvalue convergence has been obtained by other
methods.  Moreover, it is easy to show that the discrete compactness property and
Fortid property hold whenever there exist $W$-bounded cochain projections.
\end{remark}

\subsubsection{Related eigenvalue problems}\label{related-eigenv}
Recall that the source problem for the Hodge Laplacian, given by \eqref{wfhc},
can be decomposed into the problems \eqref{bfrak*} and \eqref{bfrak}, referred
to as the $\Bfrak^*$ and the $\Bfrak$ problem, respectively. More precisely,
these problems arise if the right hand side $f$ of \eqref{wfhc} is in
$\Bfrak_k^*$ or $\Bfrak^k$, respectively. In a similar way the eigenvalue
problem \eqref{eigenv-pr} can be decomposed into a $\Bfrak^*$ and the $\Bfrak$
problem. To see this, note if $(\lambda,u) \in \R \x V^k$ is an
eigenvalue/eigenvector for \eqref{eigenv-pr}, then $u \in \Hfrak^{k\perp}$,
and therefore $u = u_\Bfrak + u_\perp$, where $u_\Bfrak \in \Bfrak^k$ and $u_\perp \in
\Bfrak_k^* \cap V^k = \Zfrak^{k\perp}$. It is straightforward to check that
$u_\perp$ satisfies the $\Bfrak_k^*$ eigenvalue problem given by
\begin{equation}\label{b*eigenv}
\< du_\perp,dv \> = \lambda \<u_\perp, v\>, \quad v \in \Zfrak^{k\perp}.
\end{equation}
Here $u_\perp$ can be equal to zero even if $u \neq 0$ is an eigenvector.
On the other hand, if the pair $(\lambda,u_\perp) \in \R \x \Zfrak^{k\perp}$ is
an eigenvalue/eigenvector for \eqref{b*eigenv}, then $(\lambda,u_\perp)$ is
also an eigenvalue/eigenvector for \eqref{eigenv-pr}, where $\sigma$
and $p$ are both zero. Furthermore, for
the extended eigenvalue problem
\begin{equation}\label{b*eigenv-ext}
\< du_\perp,dv \> = \lambda \<u_\perp, v\>, \quad v \in V^k,
\end{equation}
where $u_\perp$ is sought in $V^k$, the nonzero eigenvalues correspond precisely
to the eigenvalues of \eqref{b*eigenv}, while the eigenvalue $\lambda = 0$ has
the eigenspace $\Zfrak^k$.

The pair $(\lambda,u_\Bfrak)$ satisfies the corresponding $\Bfrak^k$ eigenvalue
problem given by
\begin{equation}\label{b-eigenv}
\begin{aligned}
\<\sigma,\tau\> - \<
d \tau,u_\Bfrak\> &=0,   &&\tau\in V^{k-1},
\\*
\< d \sigma,v\>
&= \lambda \< u_\Bfrak, v \>
&&v\in \Bfrak^k,
\end{aligned}
\end{equation}
where $\sigma = d^*u_\Bfrak \in \Bfrak_{k-1}^* \cap V^{k-1} =
\Zfrak^{(k-1)\perp}$. Furthermore, any solution $(\lambda,u_\Bfrak)$ of
\eqref{b-eigenv}, where $u_\Bfrak \neq 0$, corresponds to an eigenvalue/eigenvector
of the full Hodge Laplacian \eqref{eigenv-pr}. In fact, any eigenvalue of
\eqref{eigenv-pr} is an eigenvalue of either the $\Bfrak^*$ or $\Bfrak$
problem, or both, and all eigenvalues of the $\Bfrak^*$ and $\Bfrak$ problems
correspond to an eigenvalue of \eqref{eigenv-pr}.  In short, \emph{the
eigenvalue problem \eqref{eigenv-pr} can be decomposed into the two problems
\eqref{b*eigenv} and \eqref{b-eigenv}.}

Also, note that if $(\lambda,u_\Bfrak)$ is an eigenvalue/eigenvector for the
$\Bfrak^k$ problem \eqref{b-eigenv}, and $\sigma = d^*u_\Bfrak \in
\Zfrak^{(k-1)\perp}$, then $\sigma \neq 0$, and by taking $v = d\tau$ in
\eqref{b-eigenv}, we obtain
\begin{equation}\label{b*eigenv_k-1}
\< d\sigma,d\tau \> = \lambda \<\sigma, \tau\>, \quad \tau \in
\Zfrak^{(k-1)\perp},
\end{equation}
which is a $\Bfrak_{k-1}^*$ problem. Hence, we conclude that any eigenvalue of
the $\Bfrak^k$ problem is an eigenvalue of the $\Bfrak_{k-1}^*$ problem. The
converse is also true.  To see this, let $(\lambda,\sigma) \in \R \x
\Zfrak^{(k-1)\perp}$, $\sigma \neq 0$, be a solution of \eqref{b*eigenv_k-1},
and define $u_\Bfrak \in \Bfrak^k$ such that $d^*u_\Bfrak = \sigma$. Since $d^* :
\Bfrak^k \to \Zfrak^{(k-1)\perp}$ is a bijection, this determines $u_\Bfrak$
uniquely, and $(\lambda,u_\Bfrak)$ is an eigenvalue/eigenvector for the $\Bfrak^k$
problem \eqref{b-eigenv}.  Therefore, \emph{the two problems \eqref{b-eigenv}
and \eqref{b*eigenv_k-1} are equivalent,} with the same eigenvalues, and with
eigenvectors related by $d^*u_\Bfrak = \sigma$.  In particular, \emph{the
eigenvalues of the full Hodge Laplace problem \eqref{eigenv-pr} correspond
precisely to the eigenvalues of the $\Bfrak_{k-1}^*$ and $\Bfrak_k^*$
problems.}

In the special case of the de Rham complex, the $\Bfrak_1^*$ problem is
closely related to Maxwell's equations.  Discretizations of this problem have
therefore been intensively studied in the literature (e.g., see
\cite{boffi-dcp} and the references therein). However, it is usually not
straightforward to compute the eigenvalues and eigenvectors of the discrete
$\Bfrak_k^*$ problem from the formulation \eqref{b*eigenv} since we do not, a
priori, have a basis for the corresponding space $\Zfrak_h^{k\perp} \subset
V_h^k$ available. Usually, we will have only have a basis for the space
$V_h^k$ at our disposal, and a direct computation of a basis for
$\Zfrak_h^{k\perp}$ from this basis is costly. A better alternative is
therefore to solve the discrete version of the extended eigenvalue problem
\eqref{b*eigenv-ext}, and observe that the positive eigenvalues, and the
corresponding eigenvectors, are precisely the solutions of the corresponding
problem \eqref{b*eigenv}. Alternatively, we can solve the full Hodge Laplace
problem \eqref{eigenv-pr-h}, and then throw away all eigenvalues corresponding
to $\sigma_h \neq 0$.

\section{Exterior calculus and the de Rham complex}
\label{sec:deRham}
We next turn to the most important example of the preceding theory, in which
the Hilbert complex is the de~Rham complex associated to a bounded domain
$\Omega$ in $\R^n$.  We begin by a quick review of basic notions from exterior
calculus.  Details can be found in many references, e.g.,
\cite{acta,Arnold,Bott-Tu,Federer,Janich,Lang,Taylor}.

\subsection{Basic notions from exterior calculus}
For a vector space $V$ and a non-negative integer $k$, we denote by $\Alt^kV$
the space of real-valued $k$-linear forms on $V$.  If $\dim V=n$, then
$\dim\Alt^kV=\binom{n}{k}$.  The wedge product
$\omega\wedge\eta\in\Alt^{j+k}V$ of $\omega\in\Alt^jV$ and $\eta\in\Alt^kV$ is
given by
\begin{equation*}
(\omega\wedge\eta)(v_1,\ldots,v_{j+k})
=\sum_\sigma
(\sign\sigma)\omega(v_{\sigma(1)},\ldots,v_{\sigma(j)})
\eta(v_{\sigma(j+1)},\ldots,v_{\sigma(j+k)}),
\end{equation*}
where the sum is over all permutations $\sigma$ of $\{1,\ldots,j+k\}$, for
which $\sigma(1)<\sigma(2)<\cdots<\sigma(j)$ and
$\sigma(j+1)<\sigma(j+2)<\cdots<\sigma(j+k)$.  An inner product on $V$ induces
an inner product on $\Alt^kV$:
\begin{equation*}
\< \omega,\eta\> = \sum_\sigma \omega(e_{\sigma(1)},\ldots,
e_{\sigma(k)})\eta(e_{\sigma(1)},\ldots,e_{\sigma(k)}),
\quad \omega,\eta\in\Alt^kV,
\end{equation*}
where the sum is over increasing sequences
$\sigma:\{1,\ldots,k\}\to\{1,\ldots,n\}$ and $e_1,\ldots$, $e_n$ is any
orthonormal basis (the right-hand side being independent of the choice of
orthonormal basis).  If we orient $V$ by assigning positive orientation to
some particular ordered basis (thereby assigning a positive or negative orientation
to all ordered bases, according to the determinant of the change of
basis transformation), then we may define a unique \emph{volume form} $\vol$
in $\Alt^n V$, $n=\dim V$, characterized by $\vol(e_1,\ldots,e_n)=1$ for any
positively oriented ordered orthonormal basis $e_1,\ldots,e_n$.  The
\emph{Hodge star} operator is an isometry of $\Alt^k V$ onto $\Alt^{n-k} V$
given by
\begin{equation*}
 \omega\wedge\mu=\< \star\omega,\mu\>\vol,
\quad \omega\in\Alt^k V, \ \mu\in\Alt^{n-k} V.
\end{equation*}

Given a smooth manifold $\Omega$, possibly with boundary, a \emph{differential
$k$-form} $\omega$ is a section of the $k$-alternating bundle, i.e., a map
which assigns to each $x\in\Omega$ an element $\omega_x\in\Alt^kT_x\Omega$
where $T_x\Omega$ denotes the tangent space to $\Omega$ at $x$.  We write
$C^m\Lambda^k(\Omega)$ for the space of $m$ times continuously differentiable
$k$-forms, i.e., forms for which
\begin{equation*}
 x \mapsto \omega_x(v_1(x),\ldots,v_n(x))
\end{equation*}
belongs to $C^m(\Omega)$ for any smooth vector fields $v_i$. Similarly, we may
define $C^\infty\Lambda^k$, $C^\infty_c\Lambda^k$ (smooth forms with compact
support contained in the interior of $\Omega$), etc.  If $\Omega$ is a
Riemannian manifold, and so has a measure defined on it, we may similarly
define Lebesgue spaces $L_p\Lambda^k$ and Sobolev spaces $W^m_p\Lambda^k$ and
$H^m\Lambda^k=W^m_2\Lambda^k$.  The spaces $H^m\Lambda^k$ are Hilbert spaces.
In particular, the inner product in $L_2\Lambda^k=H^0\Lambda^k$ is given by
\begin{equation*}
 \langle\omega,\eta\rangle = \langle\omega,\eta\rangle_{L^2\Lambda^k} 
= \int_\Omega \< \omega_x,\eta_x\>\, dx.
\end{equation*}
We also write $\Lambda^k$ or $\Lambda^k(\Omega)$ for the space of all smooth
differential $k$-forms, or at least sufficiently smooth, as demanded by the
context.

For any smooth manifold $\Omega$ and $\omega\in \Lambda^k(\Omega)$, the
\emph{exterior derivative} $d\omega$ is a $(k+1)$-form, which itself has
vanishing exterior derivative: $d(d\omega)=0$.  On an oriented $n$-dimensional
piecewise smooth manifold, a differential $n$-form (with, e.g., compact
support) can be integrated, without recourse to a measure or metric.

A smooth map $\phi:\Omega'\to\Omega$ between manifolds, induces a pullback of
differential forms from $\Omega$ to $\Omega'$.  Namely, if
$\omega\in\Lambda^k(\Omega)$, the \emph{pullback}
$\phi^*\omega\in\Lambda^k(\Omega')$ is defined by
\begin{equation*}
(\phi^*\omega)_x(v_1,\ldots,v_k)=
\omega_{\phi(x)}\bigl(D\phi_x(v_1),\ldots,D\phi_x(v_k)\bigr),
\quad x\in\Omega',\ v_1,\ldots,v_k\in T_x\Omega'.
\end{equation*}
The pullback respects exterior products and exterior derivatives:
\begin{equation*}
 \phi^*(\omega\wedge\eta)=\phi^*\omega\wedge\phi^*\eta, \quad
 \phi^*(d\omega)=d\phi^*\omega.
\end{equation*}
If $\phi$ is an orientation-preserving diffeomorphism of oriented
$n$-dimensional manifolds, and $\omega$ is an $n$-form on $\Omega$, then
\begin{equation*}
 \int_{\Omega'} \phi^*\omega = \int_{\Omega} \omega.
\end{equation*}
If $\Omega'$ is a submanifold of $\Omega$, then the pullback of the inclusion
map is the \emph{trace} map $\tr_{\Omega,\Omega'}$, written simply
$\tr_{\Omega'}$ or $\tr$ when the manifolds can be inferred from context.  We
recall the trace theorem which states that if $\Omega'$ is a submanifold of
codimension $1$, then the trace map extends to a bounded operator from
$H^1\Lambda^k(\Omega)$ to $L^2\Lambda^k(\Omega')$, or, more precisely, to a
bounded surjection of $H^1\Lambda^k(\Omega)$ onto
$H^{1/2}\Lambda^k(\Omega')$. A particularly important situation is when
$\Omega'=\partial\Omega$, in which case Stokes's theorem relates the
integrals of the exterior derivative and trace of an $(n-1)$-form $\omega$ on
an oriented $n$-dimensional manifold-with-boundary $\Omega$:
\begin{equation*}
\int_\Omega d\omega = \int_{\partial\Omega}\tr \omega.
\end{equation*}
(A common abuse of notation is to write $\int_{\partial\Omega}\omega$ for the
$\int_{\partial\Omega}\tr\omega$.)  Applying Stokes's theorem to the
differential form $\omega\wedge\eta$ with $\omega\in\Lambda^{k-1}$,
$\eta\in\Lambda^{n-k}$, and using the Leibniz rule
$d(\omega\wedge\eta)=d\omega\wedge\eta+(-1)^{k-1}\omega\wedge d\eta$, we
obtain the integration-by-parts formula for differential forms:
\begin{equation}\label{ibp1}
\int_\Omega d\omega\wedge\eta = (-1)^k\int_\Omega \omega\wedge d\eta +
\int_{\partial\Omega}\tr\omega\wedge\tr\eta, \quad 
\omega\in\Lambda^{k-1},\quad\eta\in\Lambda^{n-k}.
\end{equation}

On an oriented $n$-dimensional Riemannian manifold, there is a volume form
$\vol\in\Lambda^n(\Omega)$ which at each point $x$ of the manifold is equal to
the volume form on $T_x\Omega$.  Consequently, the Hodge star operation takes
$\omega\in\Lambda^k(\Omega)$ to $\star\omega\in\Lambda^{n-k}(\Omega)$
satisfying
\begin{equation*}
 \int_\Omega \omega\wedge\mu=\< \star\omega,\mu\>_{L^2\Lambda^{n-k}},
\end{equation*}
for all $\mu\in\Lambda^{n-k}(\Omega)$.  Introducing the coderivative operator
$\delta:\Lambda^k\to\Lambda^{k-1}$ defined by
\begin{equation}\label{defdelta}
 \star\delta\omega=(-1)^kd\star\omega,
\end{equation}
and setting $\eta = \star \mu$, the integration-by-parts formula becomes
\begin{equation}\label{ibp2}
 \langle d\omega,\mu\rangle = \langle \omega,\delta\mu\rangle
+\int_{\partial\Omega}
\tr\omega\wedge\tr\star\mu,\quad \omega\in\Lambda^{k-1},\ \mu\in\Lambda^k.
\end{equation}
Since the Hodge star operator is smooth and an isometry at every
point, every property of $k$-forms yields a corresponding property for
$n-k$ forms.  For example, the Hodge star operator maps the spaces
$C^m\Lambda^k$, $W^m_p\Lambda^k$, etc., isometrically onto
$C^m\Lambda^{n-k}$, $W^m_p\Lambda^{n-k}$, etc.  By definition,
$\delta:\Lambda^{n-k}\to\Lambda^{n-k-1}$ corresponds to
$d:\Lambda^k\to\Lambda^{k+1}$ via the Hodge star isomorphism.  Thus
for every property of $d$ there is a corresponding property of
$\delta$.

In case $\Omega$ is a domain in $\R^n$, we may (but usually will not) use the
standard coordinates of $\R^n$ to write a $k$-form as
\begin{equation*}
\omega = \sum_{1\le\sigma_1<\cdots<\sigma_k\le n} a_\sigma \,
d x_{\sigma_1}\wedge\cdots\wedge d x_{\sigma_k},
\end{equation*}
where the $a_\sigma\in L^2(\Omega)$ and $d x_j:\R^n\to\R$ is the linear form
which associates to a vector its $j$th coordinate.  In this case, the exterior
derivative is given by the simple formula
\begin{equation*}
 d(a_{\sigma}\,
d x_{\sigma_1}\wedge\cdots\wedge d x_{\sigma_k}) = \sum_{j=1}^n
\frac{\partial a_\sigma}{\partial x_j}\,dx_j\wedge dx_{\sigma_1}\wedge
\cdots\wedge d x_{\sigma_k},
\end{equation*}
extended by linearity to a sum of such terms.  If $\Omega$ is a domain in
$\R^n$, then $\int_\Omega a\, d x_1\wedge\cdots\wedge d x_{\sigma_n})$, has
the value suggested by the notation.  In the case of a domain in $\R^n$, the
volume form is simply $dx_1\wedge\cdots\wedge dx_n$.  We remark that the
exterior derivative and integral of differential forms can be computed on
arbitrary manifolds from the formulas on subdomains on $\R^n$ and pullbacks
through charts.

\subsection{The de~Rham complex as a Hilbert complex}
\label{subsec:deRcHc}
Henceforth we restrict attention to the case that $\Omega$ is a bounded domain in $\R^n$
with a piecewise smooth, Lipschitz boundary.
In this section we show that the de~Rham complex is a Hilbert
complex which satisfies the compactness property, and so the abstract theory of
Section~\ref{sec:HCA} applies.  We then interpret the results in the case of the de~Rham complex.

We begin by indicating how the exterior derivative $d$ can be viewed as a closed
densely-defined operator from $W^k=L^2\Lambda^k(\Omega)$ to
$W^{k+1}=L^2\Lambda^{k+1}(\Omega)$. 
Let $\omega\in L^2\Lambda^k(\Omega)$.  In view of \eqref{ibp2}, we say that
$\eta\in L^2\Lambda^{k+1}(\Omega)$ is the weak exterior derivative of $\omega$
if
\begin{equation*}
 \langle \omega,\delta\mu\rangle=\langle \eta,\mu\rangle,\quad 
\mu\in C^\infty_c\Lambda^{k+1}.
\end{equation*}
The weak exterior derivative of $\omega$, if one exists, is unique and we
denote it by $d\omega$. In analogy with the definition of Sobolev spaces (cf.,
e.g., \cite[Section~5.2.2]{evans}), we define $H\Lambda^k$ to be the space of
forms in $L^2\Lambda^k$ with a weak derivative in $L^2\Lambda^{k+1}$.  With the inner product
\begin{equation*}
 \langle\omega,\eta\rangle_{H\Lambda^k}
:= \langle\omega,\eta\rangle_{L^2\Lambda^k} 
+ \langle d\omega,d\eta\rangle_{L^2\Lambda^{k+1}},
\end{equation*}
this is easily seen to be a Hilbert space
and clearly $d$ is a bounded map from $H\Lambda^k$ to $L^2\Lambda^{k+1}$.  A
standard smoothing argument, as in \cite[Section~5.3]{evans}, implies that
$C^\infty\Lambda^k(\bar\Omega)$ is dense in $H\Lambda^k$.  We take
$H\Lambda^k(\Omega)$ as the domain $V^k$ of the exterior derivative, which is
thus densely defined in $L^2\Lambda^k(\Omega)$.  Since $H\Lambda^k$ is
complete, $d$ is a closed operator.

Thus the spaces $L^2\Lambda^k(\Omega)$ and the exterior derivative operators
$d$ form a Hilbert complex with the associated domain complex
\begin{equation}\label{drl2}
0\to H\Lambda^0(\Omega) \xrightarrow{d} H\Lambda^1(\Omega) 
\xrightarrow{d} \cdots \xrightarrow{d}
H\Lambda^n(\Omega)\to 0.
\end{equation}
This is the $L^2$ de~Rham complex. We shall see below that it satisfies the
compactness property.

To proceed, we need to identify the adjoint operator $d^*$ and its
domain $V^*_k$.  Using the surjectivity of the trace operator from
$H^1\Lambda^l(\Omega)$ onto $H^{1/2}\Lambda^l(\partial\Omega)$ and the
integration-by-parts formula \eqref{ibp1}, we can show that (cf.,
\cite[page 19]{acta}) the trace operator extends boundedly from
$H\Lambda^k(\Omega)$ to $H^{-1/2}\Lambda^k(\partial\Omega)$, and that
\eqref{ibp1} holds for $\omega\in H\Lambda^{k-1}$, $\eta\in
H^1\Lambda^{n-k}$ where the integral on the boundary is interpreted
via the pairing of $H^{-1/2}(\partial\Omega)$ and
$H^{1/2}(\partial\Omega)$.  Equivalently, we have an extended version
of \eqref{ibp2}:
\begin{equation}\label{ibp2a}
 \langle d\omega,\mu\rangle = \langle \omega,\delta\mu\rangle
+\int_{\partial\Omega}
\tr\omega\wedge\tr\star\mu,\quad \omega\in H\Lambda^{k-1},
\ \mu\in H^1\Lambda^k.
\end{equation}
Of course, there is a corresponding result obtained by the Hodge star
isomorphism which interchanges $d$ and $\delta$.  After reindexing, this is
nothing but the fact that \eqref{ibp2a} holds also for $\omega\in
H^1\Lambda^{k-1},\ \mu\in H^*\Lambda^k$ where
\begin{equation}\label{Hstar}
 H^*\Lambda^k:= \star (H\Lambda^{n-k}).
\end{equation}
Note that $H^*\Lambda^k$ consists of those differential form in $L^2\Lambda^k$
for which a weak coderivative exists in $L^2\Lambda^{k-1}$, where the weak
exterior coderivative is defined in exact analogy to the weak exterior
derivative.  Its inner product is
\begin{equation*}
 \langle\omega,\eta\rangle_{H^*\Lambda^k}
:= \langle\omega,\eta\rangle_{L^2\Lambda^k} 
+ \langle \delta\omega,\delta\eta\rangle_{L^2\Lambda^{k-1}}.
\end{equation*}
The space $H^*\Lambda^k$ is isometric to $H\Lambda^{n-k}$ via the Hodge star,
but is quite different from $H\Lambda^k$.

We also make use of the trace defined on $H\Lambda^k$ to define the subspace
with vanishing trace:
\begin{equation*}
 \0H\Lambda^k(\Omega)= \{\,\omega\in H\Lambda^k(\Omega) \,|\, 
\tr_{\partial\Omega}\omega=0\,\}.
\end{equation*}
Correspondingly, for $\omega\in H^*\Lambda^k$, the quantity $\tr\star\omega$ is
well defined, and we have
\begin{equation}\label{H0star}
 \0H^*\Lambda^k(\Omega) := \star \0H\Lambda^{n-k} 
= \{\,\omega\in H^*\Lambda^k(\Omega) \,|\, 
\tr_{\partial\Omega}\star\omega=0\,\}.
\end{equation}

From \eqref{ibp2a}, we have
\begin{equation}\label{ibp3}
 \langle d\omega,\mu\rangle = \langle \omega,\delta\mu\rangle,
\quad \omega\in H\Lambda^{k-1},\ \mu\in\0H^*\Lambda^k.
\end{equation}
(We certainly have \eqref{ibp3} with the stronger condition $\omega\in
H^1\Lambda^{k-1}$, but then we can extend to all $\omega\in H\Lambda^{k-1}$ by
continuity and density.)  Of course, the corresponding result, where
$\omega\in \0H\Lambda^{k-1},\ \mu\in H^*\Lambda^k$, holds as well.
\begin{thm}  Let $d$ be the exterior derivative viewed as an unbounded operator
$L^2\Lambda^{k-1}\to L^2\Lambda^k$ with domain $H\Lambda^k$.  Then the adjoint
$d^*$, as an unbounded operator $L^2\Lambda^k\to L^2\Lambda^{k-1}$, has as its
domain $\0H^*\Lambda^k$, and coincides with the operator $\delta$ defined in
\eqref{defdelta}.
\end{thm}
\begin{proof}
 We must show that for $\mu\in L^2\Lambda^k$, there exists $\omega\in
 L^2\Lambda^{k-1}$ such that
\begin{equation}\label{adj}
 \langle\mu,d\nu\rangle = \langle\omega,\nu\rangle,\quad \nu\in H\Lambda^{k-1},
\end{equation}
if and only if $\mu\in\0H^*\Lambda^k$ and $\omega=\delta\mu$.  The if
direction is immediate from \eqref{ibp3}.  Conversely, if \eqref{adj} holds,
then $\mu$ has a weak exterior coderivative in $L^2$, namely
$\delta\mu=\omega$.  Thus $\mu\in H^*\Lambda^k$.  Integrating by parts we have
\begin{equation*}
 \int_{\partial\Omega}\tr\nu\wedge\tr\star\mu=\langle\mu,d\nu\rangle 
- \langle\omega,\nu\rangle=0,\quad \nu\in H^1\Lambda^{k-1},
\end{equation*}
which implies that $\tr\star\mu=0$, i.e., $\mu\in\0H^*\Lambda^k$.
\end{proof}
As a corollary, we obtain a concrete characterization of the
harmonic forms:
\begin{equation}\label{hfrak}
 \Hfrak^k=\{\,\omega\in H\Lambda^k\cap \0 H^*\Lambda^k\,|\,d\omega=0,
\delta\omega=0\,\}.
\end{equation}
In other words, a $k$-form is harmonic if it satisfies the differential
equations $d\omega=0$ and $\delta\omega=0$ together with the boundary
conditions $\tr\star\omega=0$.

If the boundary of $\Omega$ is smooth, then $H\Lambda^k\cap \0 H^*\Lambda^k$
is contained in $H^1\Lambda^k$ \cite{gaffney}, and hence, by the Rellich theorem,
we obtain the
compactness property of Section~\ref{subsubsec:hilbert-complexes}.  For a general Lipschitz
boundary, e.g., for a polygonal domain, the inclusion of $H\Lambda^k\cap \0
H^*\Lambda^k$ in $H^1\Lambda^k$ need not hold, but the compactness
property remains valid \cite{picard}.  Thus all the results of Section~\ref{sec:HCA}
apply to the de~Rham complex.
In particular, we have the Hodge
decomposition of $L^2\Lambda^k$ and of $H\Lambda^k$, the Poincar\'e
inequality, well-posedness of the mixed formulation of the Hodge
Laplacian, and all the approximation results established in
Sections~\ref{subsec:approxhc}--\ref{subsec:eigenv}. We now interpret these results
a bit more concretely in the present setting.

First of all, the cohomology groups associated to the
complex \eqref{drl2} are the de~Rham cohomology groups, whose dimensions are
the Betti numbers of the domain.
Turning next to the Hodge Laplacian problem given in the abstract case by
\eqref{wfhc}, we get that $(\sigma,u,p)\in H\Lambda^{k-1}\x
H\Lambda^k\x\Hfrak^k$ is a solution if and only if
\begin{gather}\label{pdes}
 \sigma=\delta u,\ d\sigma +\delta d u = f-p \quad\text{ in $\Omega$},
\\\label{bcs}
\tr\star u =0, \ \tr\star du =0 \quad\text{on $\partial\Omega$},
\\\label{sc}
u\perp \Hfrak^k.
\end{gather}
The first differential equation and the first boundary condition are implied
by the first equation in \eqref{wfhc}, and the second differential equation
and boundary condition by the second, while the third equation in \eqref{wfhc}
is simply the side condition $u\perp\Hfrak^k$.  Note that both boundary
conditions are \emph{natural} in this variational formulation: they are
implied but not imposed in the spaces where the solution is sought.  Essential
boundary conditions could be imposed instead.  We discuss this in
Section~\ref{subsec:ebc}.

To make things more concrete, we now restrict to a domain $\Omega\subset\R^3$,
and consider the Hodge Laplacian for $k$-forms, $k=0,1,2$, and $3$.  We also
discuss the $\Bfrak^*$ and $\Bfrak$ problems given by \eqref{bfrak*} and
\eqref{bfrak} for each $k$.  We shall encounter many of the most important
partial differential equations of mathematical physics: the Laplacian,
the vector Laplacian, div-curl problems, and curl-curl problems.  These PDEs arise
in manifold applications in electromagnetism, solid mechanics, fluid mechanics,
and many other fields.

We begin by noting that, on any oriented Riemannian
manifold of dimension $n$, we have a natural way to view $0$-forms and
$n$-forms as real-valued functions, and $1$-forms and $(n-1)$-forms as vector
fields.  In fact, $0$-forms are real-valued functions and $1$-forms are
covector fields, which can be identified with vector fields via the inner
product.  The Hodge star operation then carries these identifications to
$n$-forms and $(n-1)$-forms.  In the case of a 3-dimensional domain in $\R^3$,
via these identifications all $k$-forms can be viewed as either scalar or
vector fields (sometimes called \emph{proxy fields}).  With these
identifications, the Hodge star operation becomes trivial in the sense if a
certain vector field is the proxy for, e.g., a $1$-form $\omega$, the exact
same vector field is the proxy for the $2$-form $\star\omega$.  Via proxy
fields, the exterior derivatives coincide with standard differential operators
of calculus:
\begin{equation*}
 d^0=\grad, \ d^1=\curl,\ d^2 = \div, 
\end{equation*}
and the de~Rham complex \eqref{drl2} is realized as
\begin{equation*}
0\to H^1(\Omega) \xrightarrow{\grad} H(\curl;\Omega)
\xrightarrow{\curl} H(\div;\Omega) \xrightarrow{\div} L^2(\Omega) \to 0,
\end{equation*}
where
\begin{align*}
 H(\curl;\Omega)&= \{\, u:\Omega\to\R^3\,|\, u\in L^2, \curl u\in L^2\,\},\\
 H(\div;\Omega)&= \{\, u:\Omega\to\R^3\,|\, u\in L^2, \div u\in L^2\,\}.
\end{align*}
The exterior coderivatives $\delta$ become, of course, $-\div$, $\curl$, and
$-\grad$, when acting on $1$-forms, $2$-forms, and $3$-forms, respectively.
The trace operation on $0$-forms is just the restriction to the boundary,
and the trace operator on $3$-forms vanishes (since there are no nonzero
$3$-forms on $\partial\Omega$).  The trace operator from $1$-forms on $\Omega$
to $1$-forms on the boundary takes a vector field $u$ on $\Omega$ to a
tangential vector field on the boundary, namely at each boundary point $x$,
$(\tr u)_x$ is the tangential projection of $u_x$.  For a $2$-form $u$, the
trace corresponds to the scalar $u\cdot n$ (with $n$ the unit normal) at each boundary point.

\subsubsection{The Hodge Laplacian for $k=0$}  For $k=0$, the boundary
value problem \eqref{pdes}--\eqref{sc} is the Neumann problem for the ordinary
scalar Laplacian. The space $H\Lambda^{-1}$ is understood to be $0$, so
$\sigma$ vanishes.  The harmonic form space $\Hfrak^0$ consists of the
constant functions (we assume $\Omega$ is connected---otherwise $\Hfrak^0$
would consist of functions which are constant on each connected component),
and $p$ is just the average of $f$.  The first differential equation of
\eqref{pdes} vanishes, and the second gives Poisson's equation
\begin{equation*}
 -\div \grad u = f-p \quad\text{in $\Omega$}.
\end{equation*}
Similarly, the first boundary condition in \eqref{bcs} vanishes, while the
second is the Neumann condition
\begin{equation*}
 \grad u\cdot n = 0 \quad\text{on $\partial\Omega$}.
\end{equation*}
The side condition \eqref{sc} specifies a unique solution by requiring its
average value to be zero.

Nothing additional is obtained by considering the split into the $\Bfrak^*$
and $\Bfrak$ subproblems, since the latter is trivial.  Furthermore, the
eigenvalue problem \eqref{eigenv-pr} is precisely the corresponding eigenvalue
problem for the scalar Laplacian, with the $0$ eigenspace $\Hfrak^0$ filtered
out.

\subsubsection{The Hodge Laplacian for $k=1$} \label{subsubsec:hlk=1}
In this case the differential
equations and boundary conditions are
\begin{equation}\label{k=1}
\begin{gathered}
 \sigma=-\div u,\ \grad\sigma +\curl\curl u = f-p \quad\text{ in $\Omega$},
\\
u\cdot n =0, \ \curl u\x n =0 \quad\text{on $\partial\Omega$},
\end{gathered}
\end{equation}
which is a formulation of a boundary value problem for the vector Laplacian
$\curl\curl-\grad\div$.  (Here we have used the fact that the vanishing of the
tangential component of a vector is equivalent to the vanishing of its cross
product with the normal.)  The solution is determined uniquely by the
additional condition that $u$ be orthogonal to $\Hfrak^1$, which in this case
consists of those vector fields satisfying
\begin{equation*}
\curl p = 0,\ \div p=0 \ \text{ in $\Omega$}, \quad
p\cdot n =0, \ \text{on $\partial\Omega$}.
\end{equation*}
The dimension of $\Hfrak^1$ is equal to the first Betti number, i.e., the
number of handles, of the domain (so $\Hfrak^1=0$ if the domain is
simply-connected).

The $\Bfrak^*_1$ problem \eqref{bfrak*} is defined for $L^2$ vector fields $f$
which are orthogonal to both gradients and the vector fields in $\Hfrak^1$.
In that case, the solution to \eqref{k=1} has $\sigma=0$ and $p=0$, while $u$
satisfies
\begin{equation*}
\curl\curl u = f,\ \div u=0 \ \text{ in $\Omega$}, \quad
u\cdot n =0, \ \curl u\x n =0 \ \text{on $\partial\Omega$}.
\end{equation*}
The orthogonality condition $u\perp\Hfrak^1$ again determines the solution
uniquely.

Next we turn to the $\Bfrak^1$ problem.  For source functions of the form
$f=\grad F$ for some $F\in H^1$, which we normalize so that $\int_\Omega F=0$,
\eqref{k=1} reduces to the problem of finding $\sigma\in H^1$ and
$u\in\Bfrak^1=\grad H^1$ such that:
\begin{gather*}
\sigma=-\div u,\ \grad\sigma = f \ \text{ in $\Omega$}, \quad
u\cdot n =0,  \ \text{on $\partial\Omega$}.
\end{gather*}
The differential equations may be simplified to $-\grad\div u = f$, and the
condition that $u\in\Bfrak^1$ can be replaced by the differential equation
$\curl u=0$, together with orthogonality to $\Hfrak^1$.  Now
$\grad(\sigma-F)=0$ and $\int_\Omega\sigma=-\int_{\partial\Omega}u\cdot
n=0=\int_\Omega F$, so $\sigma=F$, and we may rewrite the system as
\begin{equation*}
-\div u=F, \ \curl u=0 \ \text{ in $\Omega$}, \quad
u\cdot n =0, \ \text{on $\partial\Omega$},
\end{equation*}
which, again, has a unique solution subject to orthogonality to $\Hfrak^1$.

The eigenvalue problem \eqref{eigenv-pr} is the corresponding eigenvalue
problem for the vector Laplacian with the boundary conditions of \eqref{k=1},
and with the eigenspace $\Hfrak^1$ of the eigenvalue $\lambda = 0$ filtered
out.  As mentioned above, in this case the $\Bfrak_1^*$ eigenvalue problem,
given by \eqref{b*eigenv}, is important for models based on Maxwell's
equations. This problem takes the form
\begin{equation}\label{b*eigenv-1}
\curl\curl u = \lambda u,\ \div u=0 \ \text{ in $\Omega$}, \quad
u\cdot n =0, \ \curl u\x n =0 \ \text{on $\partial\Omega$}, \ u \perp
\Hfrak^1.
\end{equation}

\subsubsection{The Hodge Laplacian for $k=2$}  The differential equations
and boundary conditions are
\begin{equation}\label{k=2}
\begin{gathered}
 \sigma=\curl u,\ \curl\sigma -\grad\div u = f-p \quad\text{ in $\Omega$},
\\
u\x n =0, \ \div u =0 \quad\text{on $\partial\Omega$}.
\end{gathered}
\end{equation}
This is again a formulation of a boundary value problem for the vector
Laplacian $\curl\curl-\grad\div$, but with different boundary conditions than
for \eqref{k=1}, and this time stated in terms of two vector variables, rather
than one vector and one scalar.  This time uniqueness is obtained by imposing
orthogonality to $\Hfrak^2$, the space of vector fields satisfying
\begin{equation*}
\curl p = 0,\ \div p=0 \ \text{ in $\Omega$}, \quad
p\x n =0, \ \text{on $\partial\Omega$},
\end{equation*}
which has dimension equal to the second Betti number, i.e., the number of
voids in the domain.

The $\Bfrak^*_2$ problem arises for source functions of the form $f=\grad F$
for some $F\in\0H^1$ ($\0H^1$ is the function space corresponding to
$\0H^*\Lambda^3$).  We find $\sigma=0$, and $u$ solves
\begin{equation*}
-\div u=F, \ \curl u=0 \ \text{ in $\Omega$}, \quad
u\x n =0, \ \text{on $\partial\Omega$},
\end{equation*}
i.e., the same differential equation as for $\Bfrak^1$, but with different
boundary conditions, an extra assumption on $F$, and, of course, now
uniqueness is determined by orthogonality to $\Hfrak^2$.

If $\div f=0$ and $f\perp\Hfrak^2$, we get the $\Bfrak^2$ problem for which
the differential equations are $\sigma=\curl u$, $\curl\sigma =f$ and the
condition $\div u=0$ arising from the membership of $u$ in $\Bfrak^2$.  Thus
$u$ solves
\begin{equation*}
\curl\curl u = f,\ \div u=0 \ \text{ in $\Omega$}, \quad
u\x n =0, \ \text{on $\partial\Omega$},
\end{equation*}
the same differential equation as for $\Bfrak^*_1$, but with different
boundary conditions.

The eigenvalue problem \eqref{eigenv-pr} is the corresponding eigenvalue
problem for the vector Laplacian with the boundary conditions of \eqref{k=2},
and with the eigenspace $\Hfrak^2$ of the eigenvalue $\lambda = 0$ filtered
out, while the corresponding $\Bfrak^2$ problem, of the form \eqref{b-eigenv},
takes the form
\begin{equation*}
\curl\curl u = \lambda u,\ \div u=0 \ \text{ in $\Omega$}, \quad
u\x n =0, \ \text{on $\partial\Omega$}, \  u \perp \Hfrak^2.
\end{equation*}
Note that this is the same problem as the $\Bfrak^*_1$ eigenvalue
problem\eqref{b*eigenv-1}, but with different boundary conditions. However, if
we define $\sigma = \curl u$, then it is straightforward to check that the pair
$(\lambda,\sigma)$ will indeed solve the $\Bfrak^*_1$ problem
\eqref{b*eigenv-1}. This is an instance of the general equivalence between the
$\Bfrak^k$ problem, given by \eqref{b-eigenv}, and the corresponding
$\Bfrak^*_{k-1}$ problem, which was pointed out in
Section~\ref{related-eigenv}.

\subsubsection{The Hodge Laplacian for $k=3$}\label{subsubsec:hlk=3}
In this case the Hodge
Laplacian problem, which coincides with the $\Bfrak^3$ problem, is
\begin{equation*}
 \sigma=-\grad u,\ \div\sigma = f \text{ in $\Omega$},
\quad
u =0\ \text{on $\partial\Omega$},
\end{equation*}
which is the Dirichlet problem for Poisson's equation.  There are no non-zero
harmonic forms, and the problem has a unique solution.  Furthermore, the
eigenvalue problem \eqref{eigenv-pr} is the corresponding eigenvalue problem
for the scalar Laplacian with Dirichlet boundary conditions.

\section{Finite element approximation of the de Rham complex}
\label{sec:FEDF}
Our goal in this section is to discretize the de~Rham complex so that we may apply
the abstract results on approximation of Hilbert complexes from Section~\ref{sec:HCA}.
Hence we need to construct
finite dimensional subspaces $\Lambda^k_h$ of $H\Lambda^k(\Omega)$.  As we saw
in Section~\ref{subsec:well-p-h-hc}, the key properties these spaces must possess is, first, that
$d\Lambda^k_h\subset \Lambda^{k+1}_h$ so they form a subcomplex
$(\Lambda_h,d)$
of the de~Rham complex, second, that there exist uniformly bounded cochain projections
$\pi_h$ from $(L^2\Lambda,d)$ to $(\Lambda_h,d)$, and third, good approximation properties.
We may then use these spaces in
a Galerkin method based on the mixed formulation, as
described in Section~\ref{subsec:well-p-h-hc}, and the error estimates given
in Theorems~\ref{qo1cor} and \ref{impest} bound the error in the Galerkin solution in terms of the
approximation error afforded by the subspaces.  In this section
we will construct the spaces $\Lambda^k_h$ as spaces of \emph{finite element differential forms}
and show that they satisfy all these requirements, and can be efficiently implemented.

As discussed in Section~\ref{sec:femethod}, a finite element space is a
space of piecewise polynomials which is specified by the finite element assembly process,
i.e., by giving
a triangulation of the domain, a finite dimensional space of polynomial
functions (or, in our case, differential forms) on each element of the triangulation,
called the shape functions,
and a set of degrees of freedom for the shape functions associated with the faces of
various dimensions of the elements, which will be used to determine the degree of
interelement continuity. Spaces constructed in this way can be implemented very efficiently, since
they admit basis functions with small support and so lead to sparse algebraic systems.

In Section~\ref{subsec:pdf}, we will discuss the spaces of polynomial differential forms
which we will use as shape functions.  In Section~\ref{subsec:dof}, we specify degrees of freedom
for these spaces on simplices, and study the resulting finite element spaces of differential forms.
In particular, we show that they can be combined
in a variety of ways to form subcomplexes of the de~Rham complex.  In Section~\ref{subsec:bases},
we briefly describe recent work related to the implementation
of such finite elements with explicit local bases. In Sections~\ref{subsec:approx} and \ref{subsec:bcp},
we construct $L^2$-bounded cochain projections into these
spaces and obtain error estimates for them.  Finally Section~\ref{subsec:ahl} is simply a matter of
collecting our results in order to obtain error estimates for the resulting approximations
of the Hodge Laplacian.  Many of the results of this section have
appeared previously, primarily in \cite{acta}, and therefore many
proofs are omitted.

\subsection{Polynomial differential forms and the Koszul complex}
\label{subsec:pdf}
In this section, we consider spaces of polynomial differential forms, which
lead to a variety of subcomplexes of the de~Rham complex.  These will be used
in the next section to construct finite element spaces of differential forms.
The simplest spaces of polynomial differential $k$-forms are the spaces $\P_r\Lambda^k(\R^n)$
consisting of all differential $k$-forms on $\R^n$ whose coefficients are polynomials of
degree at most $r$.  In addition to these spaces, we will use another family
of polynomial form spaces, denoted $\P_r^-\Lambda^k(\R^n)$, which will be constructed
and analyzed using the Koszul differential and the associated Koszul complex.
Spaces taken from these two families can be combined into polynomial subcomplexes
of the de~Rham complex in numerous ways
(there are essentially $2^{n-1}$ such subcomplexes associated to each polynomial degree).  These
will lead to finite element de~Rham subcomplexes, presented in
Section~\ref{subsec:dof}.  Some of these have appeared in the
literature previously, with the systematic derivation of all of them first
appearing in \cite{arnold-falk-winther1}.

\subsubsection{Polynomial differential forms}
\label{subsubsec:polynomial-df}
Let $\P_r(\R^n)$ and $\H_r(\R^n)$ denote the spaces of
polynomials in $n$ variables of degree at most $r$ and of homogeneous
polynomial functions of degree $r$, respectively.
We interpret these spaces to be the zero space if $r<0$.
We can then define spaces of polynomial differential forms,
$\P_r\Lambda^k(\R^n)$, $\H_r\Lambda^k(\R^n)$, etc., as those
differential forms which, when applied to a constant vector field,
have the indicated polynomial dependence. For brevity, we
will at times suppress $\R^n$ from the notation and write simply
$\P_r$, $\H_r$, $\P_r\Lambda^k$, etc.

The dimensions of these spaces are easily calculated:
\begin{equation*}
\begin{aligned}
\dim\P_r\Lambda^k(\R^n) &= \dim \P_r(\R^n)\cdot\dim\Alt^k\R^n
\\*&=\binom{n+r}{n}\binom{n}{k} = \binom{r+k}{r}\binom{n+r}{n-k},
\end{aligned}
\end{equation*}
and $\dim\H_r\Lambda^k(\R^n)=\dim\P_r\Lambda^k(\R^{n-1})$.

For each polynomial degree $r\ge0$, we get a homogeneous
polynomial subcomplex of the de~Rham complex:
\begin{equation}\label{hdr}
0\to \H_r\Lambda^0 \xrightarrow{d} \H_{r-1}\Lambda^{1}
 \xrightarrow{d} \cdots \xrightarrow{d} \H_{r-n}\Lambda^n \to 0.
\end{equation}
We shall verify below the exactness of this sequence.  More precisely,
the cohomology vanishes if $r>0$ and also for $r=0$ except in the
lowest degree, where the cohomology space is $\R$ (reflecting
the fact that the constants are killed by the gradient).

Taking the direct sum of the homogeneous polynomial
de~Rham complexes over all polynomial
degrees gives the polynomial de~Rham complex:
\begin{equation}\label{pdr}
0\to \P_r\Lambda^0 \xrightarrow{d} \P_{r-1}\Lambda^{1}
 \xrightarrow{d} \cdots \xrightarrow{d} \P_{r-n}\Lambda^n \to 0,
\end{equation}
for which the cohomology space is $\R$ in the lowest degree, and vanishes
otherwise.


\subsubsection{The Koszul complex}
\label{subsubsec:koszul}
Let $x\in\R^n$.  
Since there is a natural identification of $\R^n$ with the tangent
space $T_0\R^n$ at the origin, there is a vector in $T_0\R^n$
corresponding to $x$.    (The origin is chosen for convenience
here, but we could use any other point instead.)  Then the translation map
$y\mapsto y+x$ induces an isomorphism from $T_0\R^n$ to $T_x\R^n$,
and so there is an element $X(x)\in T_x\R^n$ corresponding to $x$.
(Essentially $X(x)$ is the vector based at $x$ which points opposite
to the origin and whose length is $|x|$.)
Contraction with the vector
field $X$ defines a map $\kappa$ from $\Lambda^k(\R^n)$ to
$\Lambda^{k-1}(\R^n)$ called the \emph{Koszul differential}:
\begin{equation*}
(\kappa\omega)_x(v_1,\ldots,v_{k-1}) =
\omega_x\bigl(X(x),v_1,\ldots,v_{k-1}\bigr).
\end{equation*}
It is easy to see that $\kappa$ is a graded differential, i.e.,
\begin{equation*}
\kappa\circ\kappa=0
\end{equation*}
and
\begin{equation*}
\kappa(\omega\wedge\eta) = (\kappa\omega)\wedge\eta 
+ (-1)^k\omega\wedge(\kappa\eta),\quad
\omega\in\Lambda^k,\ \eta\in\Lambda^l.
\end{equation*}
In terms of coordinates, if $\omega_x = a(x)\,
d x_{\sigma_1}\wedge\cdots\wedge d x_{\sigma_k}$, then
\begin{equation*}
(\kappa\omega)_x = \sum_{i=1}^k (-1)^{i+1}
a(x)\,x_{\sigma(i)} \,d x_{\sigma_1}\wedge\cdots \wedge
\widehat{d x}_{\sigma(i)}\wedge\cdots\wedge d x_{\sigma_k},
\end{equation*}
where the notation $\widehat{d x}_{\sigma_i}$ means that the
term is omitted.  Note that $\kappa$ maps $\H_r\Lambda^k$ to
$\H_{r+1}\Lambda^{k-1}$, i.e., $\kappa$ increases polynomial degree and
decreases form degree, the exact opposite of the exterior derivative
$d$.

The Koszul differential gives rise to the homogeneous \emph{Koszul complex}
\cite[Chapter 3.4.6]{Loday},
\begin{equation}\label{koszul}
0\rightarrow \H_{r-n}\Lambda^n  \xrightarrow{\kappa} 
\H_{r-n+1}\Lambda^{n-1} \xrightarrow{\kappa} \cdots  \xrightarrow{\kappa} \H_r\Lambda^0
\to0.
\end{equation}
We show below that this complex is exact for $r>0$.
Adding over polynomial
degrees, we obtain the Koszul complex (for any $r\ge0$),
\begin{equation*}
0\rightarrow \P_{r-n}\Lambda^n  \xrightarrow{\kappa} 
\P_{r-n+1}\Lambda^{n-1} \xrightarrow{\kappa} \cdots  \xrightarrow{\kappa} \P_r\Lambda^0
\to0,
\end{equation*}
for which all the cohomology spaces vanish, except the rightmost,
which is equal to $\R$.

To prove the exactness of the homogeneous polynomial de~Rham and Koszul
complexes, we note a key connection between the exterior derivative and the
Koszul differential.  In the language of homological algebra, this says that
the Koszul differential is a contracting homotopy for the homogeneous polynomial
de~Rham complex.

\begin{thm}\label{contract}
\begin{equation}\label{dk+kd}
(d\kappa+\kappa d)\omega = (r+k)\omega,\quad
\omega\in\H_r\Lambda^k.
\end{equation}
\end{thm}

\begin{proof}  This result can be established by a direct computation
or by using the 
\emph{homotopy formula} of differential geometry (Cartan's magic
formula). See \cite{acta} for the details.
\end{proof}

As a simple consequence of \eqref{dk+kd}, we prove the injectivity
of $d$ on the range of $\kappa$ and {\it vice versa}.

\begin{thm}\label{injlemma}
If $d\kappa\omega=0$ for some $\omega\in\P\Lambda$, then
$\kappa\omega=0$.  If $\kappa d\omega=0$ for some $\omega\in\P\Lambda$, then
$d\omega=0$. 
\end{thm}

\begin{proof}
We may assume that $\omega\in\H_r\Lambda^k$ for some $r,k\ge0$.  If
$r=k=0$, the result is trivial, so we may assume that $r+k>0$.  Then
$(r+k)\kappa\omega=\kappa(d\kappa+\kappa d)\omega=0$,
if $d\kappa\omega=0$, so $\kappa\omega=0$ in this case.  Similarly,
$(r+k)d\omega=d(d\kappa+\kappa d)\omega=0$,
if $\kappa d\omega=0$.
\end{proof}

Another easy application of \eqref{dk+kd} is to establish the claimed
cohomology of the Koszul complex and polynomial de~Rham complex.
Suppose that $\omega\in\H_r\Lambda^k$ for some $r,k\ge0$
with $r+k>0$, and that $\kappa\omega=0$.  From \eqref{dk+kd}, we
see that $\omega=\kappa\eta$ with
$\eta=d\omega/(r+k)\in\H_{r-1}\Lambda^{k+1}$.  This establishes
the exactness of the homogeneous Koszul complex \eqref{koszul}
(except when $r=0$ and the sequence reduces to $0\to\R\to0$).
A similar argument establishes the exactness of \eqref{hdr}.

Another immediate but important consequence of \eqref{dk+kd}
is a direct sum decomposition
of $\H_r\Lambda^k$ for $r,k\ge0$ with $r+k>0$:
\begin{equation}\label{Hdirectsum}
\H_r\Lambda^k = \kappa\H_{r-1}\Lambda^{k+1} \oplus
d\H_{r+1}\Lambda^{k-1}.
\end{equation}
Indeed, if $\omega\in\H_r\Lambda^k$, then
$\eta=d\omega/(r+k)\in \H_{r-1}\Lambda^{k+1}$ and
$\mu=\kappa\omega/(r+k)\in \H_{r+1}\Lambda^{k-1}$ and
$\omega=\kappa\eta+d\mu$, so $\H_r\Lambda^k
 = \kappa\H_{r-1}\Lambda^{k+1} +
d\H_{r+1}\Lambda^{k-1}$.  Also, if $\omega\in
\kappa\H_{r-1}\Lambda^{k+1} \cap
d\H_{r+1}\Lambda^{k-1}$, then $d\omega=\kappa\omega=0$
(since $d\circ d=\kappa\circ\kappa=0$), and so, by \eqref{dk+kd},
$\omega=0$.  This shows that the sum is direct.  Since
$\P_r\Lambda^k = \bigoplus_{j=0}^r \H_j\Lambda^k$,
we also have
\begin{equation*}
\P_r\Lambda^k = \kappa\P_{r-1}\Lambda^{k+1} \oplus
d\P_{r+1}\Lambda^{k-1}.
\end{equation*}

The exactness of the Koszul complex can be used to compute the
dimension of the summands in \eqref{Hdirectsum} (c.f. \cite{acta}).
\begin{thm}
Let $r\ge0$, $1\le k\le n$, for integers $r$, $k$, and $n$.  Then
\begin{equation}\label{dim}
\dim\kappa\H_r\Lambda^k(\R^n) = \dim d\H_{r+1}\Lambda^{k-1}(\R^n)=
\binom{n+r}{n-k} \binom{r+k-1}{k-1}.
\end{equation}
\end{thm}

\subsubsection{The space $\P_r^-\Lambda^k$}
\label{subsubsec:Prminus}
Let $r\ge1$. Obviously, $\P_r\Lambda^k=\P_{r-1}\Lambda^k+\H_r\Lambda^k$.
In view of \eqref{Hdirectsum}, we may define a space of $k$-forms
intermediate between $\P_{r-1}\Lambda^k$ and $\P_r\Lambda^k$ by
\begin{equation*}
\P_r^-\Lambda^k=\P_{r-1}\Lambda^k+\kappa\H_{r-1}\Lambda^{k+1}
=\P_{r-1}\Lambda^k+\kappa\P_{r-1}\Lambda^{k+1}.
\end{equation*}
Note that the first sum is direct, while the second need not
be.
An equivalent definition is
\begin{equation*}
\P_r^-\Lambda^k=\{\,\omega\in\P_r\Lambda^k\,|\,\kappa\omega
\in\P_r\Lambda^{k-1}\,\}.
\end{equation*}
Note that $\P_r^-\Lambda^0=\P_r\Lambda^0$ and
$\P_r^-\Lambda^n=\P_{r-1}\Lambda^n$, but for $0<k<n$, $\P_r^-\Lambda^k$
is contained strictly between $\P_{r-1}\Lambda^k$ and $\P_r\Lambda^k$.
For $r\le0$, we set $\P_r^-\Lambda^k=0$.

From \eqref{dim}, we have
\begin{equation*}
\begin{aligned}
\dim \P_r^-\Lambda^k(\R^n) &= \dim\P_{r-1}\Lambda^k+\dim\kappa\H_{r-1}\Lambda^{k+1}
\\*&=\binom{n+r-1}{n}\binom{n}{k} + \binom{n+r-1}{n-k-1}\binom{r+k-1}{k}
\\*&= \binom{r+k-1}{k}\binom{n+r}{n-k},
\end{aligned}
\end{equation*}
where the last step is a simple identity.


We also note the following simple consequences of Lemma~\ref{injlemma}.
\begin{thm}\label{dvan}
If $\omega\in\P_r^-\Lambda^k$ and $d\omega=0$, then
$\omega\in\P_{r-1}\Lambda^k$.  Moreover, $d\P_r\Lambda^k=d\P_r^-\Lambda^k$.
\end{thm}

\begin{proof}
Write $\omega=\omega_1+\kappa\omega_2$ with $\omega_1\in\P_{r-1}\Lambda^k$
and $\omega_2\in\P_{r-1}\Lambda^{k+1}$.  Then
\begin{equation*}
d\omega=0 \implies d\kappa\omega_2=0 \implies \kappa\omega_2=0
\implies \omega\in\P_{r-1}\Lambda^k,
\end{equation*}
showing the first result.  For the second it suffices to note
that
$\P_r\Lambda^k=\P_r^-\Lambda^k+d\P_{r+1}\Lambda^{k-1}$.
\end{proof}

\begin{remark}
We defined the Koszul differential as contraction with the vector field $X$, where
$X(x)$ is the translation to $x$ of
the vector pointing from the origin in $\R^n$ to $x$.  The choice of
the origin as a base point is arbitrary---any point in $\R^n$ could be used.
That is, if $y\in\R^n$, we can define a vector field $X_y$ by
assigning to each point $x$ the translation to $x$ of the
vector pointing from $y$ to $x$, and then
define a Koszul differential $\kappa_y$ by contraction with $X_y$.  It is easy
to check that for $\omega\in \P_{r-1}\Lambda^{k+1}$ and
any two points $y,y'\in\R^n$, the difference
$\kappa_y\omega-\kappa_{y'}\omega\in \P_{r-1}\Lambda^k$.
Hence the space
\begin{equation*}
\P_r^-\Lambda^k = \P_{r-1}\Lambda^k +
\kappa_y\P_{r-1}\Lambda^{k+1}
\end{equation*}
does not depend on the particular choice of the point $y$.
This observation is important, because it allows us to define
$\P_r^-\Lambda^k(V)$ for any affine subspace $V$
of $\R^n$.  We simply set
\begin{equation*}
\P_r^-\Lambda^k(V) = \P_{r-1}\Lambda^k(V)
+\kappa_y\P_{r-1}\Lambda^{k+1}(V),
\end{equation*}
where $y$ is any point of $V$.  Note
that if $\omega\in\P_r^-\Lambda^k(\R^n)$, then the trace of
$\omega$ on $V$ belongs to $\P_r^-\Lambda^k(V)$.
\end{remark}

\begin{remark}
The spaces $\P_r^-\Lambda^k(\R^n)$ are \emph{affine-invariant}, i.e., if $\phi:\R^n\to\R^n$
is an affine map, then the pullback $\phi^*$ maps this space into itself.  Of course, the full
polynomial space $\P_r\Lambda^k(\R^n)$ is affine-invariant as well.  In \cite[Section~3.4]{acta},
all the finite dimensional affine-invariant spaces of polynomial differential forms are
determined.  These are precisely the spaces in the
$\P$ and $\P^-$ families together with one further
family of spaces which is of less interest to us.
\end{remark}

\subsubsection{Exact sequences of polynomial differential forms}
\label{subsubsec:seqs}
We have seen the polynomial de~Rham complex \eqref{pdr} is
a subcomplex of the de~Rham complex on $\R^n$ for which cohomology
vanishes except for the constants at the lowest order.  In other words,
for any $r\ge 0$ the sequence
\begin{equation}\label{pdr1}
\R\hookrightarrow \P_r\Lambda^0 \xrightarrow{d} \P_{r-1}\Lambda^{1}
 \xrightarrow{d} \cdots \xrightarrow{d} \P_{r-n}\Lambda^n \to 0
\end{equation}
is exact, i.e., is a resolution of $\R$.

As we shall soon verify, the complex
\begin{equation}\label{pdr2}
\R\hookrightarrow \P_r^-\Lambda^0 \xrightarrow{d} \P_r^-\Lambda^{1}
 \xrightarrow{d} \cdots \xrightarrow{d} \P_r^-\Lambda^n \to 0
\end{equation}
is another resolution of $\R$, for any $r>0$.  Note that in this
complex, involving the $\P_r^-\Lambda^k$  spaces, the polynomial
degree $r$ is held fixed, while in \eqref{pdr1}, the polynomial degree
decreases as the form order increases.  Recall that the $0$th order spaces in
these complexes, $\P_r\Lambda^0$ and
$\P_r^-\Lambda^0$, coincide. In fact, the complex \eqref{pdr1} is a
subcomplex of \eqref{pdr2}, and these two are the extreme cases of a
set of $2^{n-1}$ different resolutions of $\R$, each a subcomplex
of the next, and all of which have
the space $\P_r\Lambda^0$ in the $0$th order.
To exhibit them, we begin with the inclusion
$\R\hookrightarrow\P_r\Lambda^0$.  We may continue the complex
with either the map $d:\P_r\Lambda^0\to \P_{r-1}\Lambda^1$ or
$d:\P_r\Lambda^0\to\P_r^-\Lambda^1$, the former being a subcomplex
of the latter.  With either choice, the cohomology vanishes at the first
position.
Next, if we made the first choice, we can continue the complex
with either $d:\P_{r-1}\Lambda^1\to \P_{r-2}\Lambda^2$ or
$d:\P_{r-1}\Lambda^1\to\P_{r-1}^-\Lambda^2$.  Or, if we made the
second choice, we can continue with either $d:\P_r^-\Lambda^1\to
\P_{r-1}\Lambda^2$ or $d:\P_r^-\Lambda^1\to\P_r^-\Lambda^2$.  In the first
case, we may use the exactness of \eqref{pdr1} to see that
the second cohomology space vanishes.  In the second case, this
follows from Lemma~\ref{dvan}.  Continuing in this way
at each order, $k=1,\ldots,n-1$, we have two choices for the space
of $k$-forms (but only one choice for $k=n$, since $\P_{r-1}\Lambda^n$
coincides with $\P_r^-\Lambda^n$), and so we obtain $2^{n-1}$ 
complexes.  These form a totally ordered set with respect to
subcomplex inclusion.
For $r\ge n$ these are all distinct (but for small $r$ some coincide
because the later spaces vanish).

In the case $n=3$, the four complexes so obtained are:
\begin{equation*}
 \begin{CD}
\R\hookrightarrow\P_r\Lambda^0 @>d>> \P_{r-1}\Lambda^1 @>d>>
\P_{r-2}\Lambda^2 @>d>> \P_{r-3}\Lambda^3 \to 0,
\\
\R\hookrightarrow\P_r\Lambda^0 @>d>> \P_{r-1}\Lambda^1 @>d>>
\P_{r-1}^-\Lambda^2 @>d>> \P_{r-2}\Lambda^3 \to 0,
\\
\R\hookrightarrow\P_r\Lambda^0 @>d>> \P_r^-\Lambda^1 @>d>>
\P_{r-1}\Lambda^2 @>d>> \P_{r-2}\Lambda^3 \to 0,
\\
\R\hookrightarrow\P_r\Lambda^0 @>d>> \P_r^-\Lambda^1 @>d>>
\P_r^-\Lambda^2 @>d>> \P_{r-1}\Lambda^3 \to 0.
\end{CD}
\end{equation*}

\subsection{Degrees of freedom and finite element differential forms}
\label{subsec:dof}
Having introduced the spaces of polynomial differential forms
$\P_r\Lambda^k(\R^n)$ and $\P_r^-\Lambda^k(\R^n)$, we now wish to create
finite element spaces of differential forms.  We begin with notation for
spaces of polynomial differential forms on simplices.  If $f$ is a simplex (of any dimension)
in $\R^n$, we define
\begin{equation*}
 \P_r\Lambda^k(f)=\tr_{\R^n,f}\P_r\Lambda^k(\R^n), \quad
 \0\P_r\Lambda^k(f)=\{\,\omega\in\P_r\Lambda^k(f)\,|\, \tr_{f,\partial f}\omega=0\,\}.
\end{equation*}
The spaces $\P_r^-\Lambda^k(f)$ and $\0\P_r^-\Lambda^k(f)$ are defined similarly.

Now let $\Omega$ be a bounded polyhedral
domain which is triangulated, i.e., partitioned into a finite set $\T$ of $n$-simplices
determining a simplicial decomposition of $\Omega$.  This means that
the union of the elements of $\T$ is the closure of $\Omega$, and the intersection of any two
is either empty or a common subsimplex of each.
By way of notation, for any simplex $T$ we denote by $\Delta_d(T)$ the set of subsimplices
of $T$ of dimension $d$, and by $\Delta(T)$ the set of all subsimplices of $T$.  We also
set $\Delta_d(\T):=\bigcup_{T\in\T}\Delta_d(T)$ and $\Delta(\T):=\bigcup_{T\in\T}\Delta(T)$.

Corresponding to the $\P$ and $\P^-$ families of spaces of polynomial
differential forms, we will define two families of spaces of finite element
differential forms with respect to the triangulation $\T$, denoted
$\P_r\Lambda^k(\T)$ and $\P_r^-\Lambda^k(\T)$.  We shall show that these are
subspaces of $H \Lambda^k(\Omega)$ and can be collected, in various ways, into
subcomplexes of the de~Rham complex.

The spaces $\P_r\Lambda^k(\T)$ and $\P_r^-\Lambda^k(\T)$
will be obtained by the finite element assembly process.
For each $T\in\T$, we will choose the corresponding polynomial space
$\P_r\Lambda^k(T)$ or $\P_r^-\Lambda^k(T)$ to be used as
shape functions.  The other ingredient needed to define the
finite element space is a set of degrees of
freedom for the shape function spaces, that is, a basis for
the dual space, in which each degree of freedom is associated with a
particular subsimplex. When a subsimplex is shared by more than one simplex in
the triangulation, we will insist that the degrees of freedom associated with
that subsimplex be single-valued in a sense made precise below,
and this will determine the interelement
continuity.

The degrees of freedom for $\P_r\Lambda^k(\T)$ and $\P_r^-\Lambda^k(\T)$
which we shall associate to a $d$-dimensional subsimplex $f$ of $T$
will be of the following form: for some $(d-k)$-form $\eta$ on $f$, the functional will
be $\omega\mapsto \int_f\tr_{T,f}\omega\wedge\eta$.   The span of all the degrees of
freedom associated to $f$ is a subspace of the dual space of the shape function space,
and so we obtain a decomposition of the dual space of the shape functions on
$T$ into a direct sum of subspaces indexed by the subsimplices of $T$.  It
is really this \emph{geometric decomposition of the dual space} that determines the
interelement continuity rather than the particular choice of degrees of
freedom, since we may choose any convenient basis for each space in the
decomposition and obtain the same assembled finite element space.

The geometric
decompositions of the dual spaces of $\P_r\Lambda^k(T)$ or $\P_r^-\Lambda^k(T)$
are given specifically in the following theorem, which is proven in \cite[Sections~4.5 and 4.6]{acta}.

\begin{thm}\label{stardecomp}
Let $r$, $k$, and $n$ be integers with $0\le k\le n$ and $r>0$, and
let $T$ be an $n$-simplex in $\R^n$.  

1. To each $f \in \Delta(T)$, associate a space $W_r^k(T,f)\subset\P_r\Lambda^k(T)^*$:
\begin{equation*}
W_r^k(T,f)=\biggl\{\,\omega\mapsto\int_f \tr_{T,f}\omega\wedge\eta \,\big|\,\eta\in
   \P_{r+k-\dim f}^-\Lambda^{\dim f-k}(f)\,\biggr\}.
\end{equation*}
Then $W_r^k(T,f)\cong \P_{r+k-\dim f}^-\Lambda^{\dim f-k}(f)$ via the obvious correspondence, and
$$
\P_r\Lambda^k(T)^* = \bigoplus_{f\in\Delta(T)} W_r^k(T,f).
$$

2. To each $f \in \Delta(T)$, associate a space $W_r^{k-}(T,f)\subset\P^-_r\Lambda^k(T)^*$:
\begin{equation*}
W_r^{k-}(T,f)=\biggl\{\,\omega\mapsto\int_f \tr_{T,f}\omega\wedge\eta\,\big|\,
\eta\in \P_{r+k-\dim f-1}\Lambda^{\dim f-k}(f)\,\biggr\}.
\end{equation*}
Then $W_r^{k-}(T,f)\cong \P_{r+k-\dim f-1}\Lambda^{\dim f-k}(f)$ via the obvious correspondence, and
$$
\P^-_r\Lambda^k(T)^* = \bigoplus_{f\in\Delta(T)} W_r^{k-}(T,f).
$$
\end{thm}

Note that the spaces $W_r^k(T,f)$ and $W_r^{k-}(T,f)$ vanish if $\dim f<k$.
Note also that the dual space of $\P_r\Lambda^k(T)$ is expressed in terms
of spaces in the $\P^-$ family, and vice versa.
This intimate connection between the $\P$ and $\P^-$ families of spaces
of polynomial differential forms is most clearly seen in the following
algebraic proposition, which is closely related to Theorem~\ref{stardecomp},
and is also proved in \cite[Sections~4.5 and 4.6]{acta}.

\begin{lem}\label{trace2}
With $r$, $k$, $n$ as above,
\begin{equation*}
\0\P_r\Lambda^k(T)^* \cong \P_{r+k-n}^-\Lambda^{n-k}(T)
\text{\quad and\quad}
\0\P_r^-\Lambda^k(T)^* \cong \P_{r+k-n-1}\Lambda^{n-k}(T).
\end{equation*}
\end{lem}

With the decompositions in Theorem~\ref{stardecomp}, we can define the finite element
spaces.  Thus $\P_r\Lambda^k(\T)$ consists of all forms $\omega\in L^2\Lambda^k(\Omega)$
such that $\omega|_T$ belongs to the shape function space $\P_r\Lambda^k(T)$ for all
$T\in\T$, and for which the quantities $\int_f\tr_f\omega\wedge\eta$ are single-valued
for all $f\in\Delta(T)$ and all $\eta\in\P_{r+k-\dim f}^-\Lambda^{\dim f-k}(f)$.
More precisely, this means that if $f$ is a common face of $T_1,T_2\in\T$, then
\begin{equation*}
 \int_f\tr_{T_1,f}(\omega|_{T_1})\wedge\eta =  \int_f\tr_{T_2,f}(\omega|_{T_2})\wedge\eta
\end{equation*}
for all such $f$ and $\eta$.  The $\P^-$ family of spaces is defined analogously.

The degrees of freedom determine the amount of interelement continuity enforced
on the finite element space.  Of course we need to know that the assembled spaces
belong to $H\Lambda^k(\Omega)$.  In fact, the degrees of freedom we imposed enforce
exactly the continuity needed, as shown in the following theorem.
This is proved in \cite[Section~5.1]{acta}, where the equations below are taken as  definitions, and it
is shown that the assembly process leads to the same spaces.

\begin{thm}\label{smoothness}
 \begin{align*}
\P_r\Lambda^k(\T)
&= \{\omega \in H\Lambda^k(\Omega) \, | \ \omega|_T \in \P_r\Lambda^k,
\, T \in \T \},\\*
\P_r^-\Lambda^k(\T)
&= \{\omega \in H\Lambda^k(\Omega)\, | \ \omega|_T \in
\P_r^-\Lambda^k, 
\, T \in \T \}.
\end{align*}
\end{thm}

Next we note that these spaces of finite element differential forms can be collected
in subcomplexes of the de~Rham complex.  In view of Theorem~\ref{smoothness}, we have
$d\P_r\Lambda^k(\T)\subset\P_{r-1}\Lambda^{k+1}(\T)$
and $d\P_r^-\Lambda^k(\T)\subset\P_r^-\Lambda^{k+1}(\T)$.
Corresponding to the resolutions defined in Section~\ref{subsubsec:seqs}, we obtain
$2^{n-1}$ de~Rham subcomplexes for each value of $r$ and each mesh $\T$.
Each complex begins with the space $\P_r\Lambda^0(\T)=\P_r^-\Lambda^0(\T)$.
The maps making up the subcomplexes are all one of the following types:
\begin{equation}\label{maptypes}
\begin{CD}
 \left\{
\begin{matrix}
 \P_{s+1}\Lambda^{k-1}(\T)\\[1ex] \text{or} \\[1ex] \P_{s+1}^-\Lambda^{k-1}(\T)
\end{matrix}
\right\}
@>d>>
 \left\{
\begin{matrix}
 \P_{s+1}^-\Lambda^k(\T)\\[1ex] \text{or} \\[1ex] \P_s\Lambda^k(\T)
\end{matrix}
\right\}
\end{CD}
\end{equation}
for some $s$ and some $k$.  In Section~\ref{subsec:bcp}, we shall show that
all the subcomplexes admit bounded cochain projections.  These $2^{n-1}$ complexes,
which are distinct for $r\ge n$, are linearly ordered by inclusion.
The maximal complex is
\begin{equation*}
0\to \P_r^-\Lambda^0(\T) \xrightarrow{d} \P_r^-\Lambda^{1}(\T)
 \xrightarrow{d} \cdots \xrightarrow{d} \P_r^-\Lambda^n(\T) \to 0.
\end{equation*}
The spaces $\P_r^{-}\Lambda^k(\T)$ in this complex are referred to
as the higher order Whitney forms, since for $r=1$, this is exactly the
complex introduced by Whitney \cite{whitney}.  The minimal complex
with this starting space is the complex of polynomial differential forms
\begin{equation*}
0\to \P_r\Lambda^0(\T) \xrightarrow{d} \P_{r-1}\Lambda^{1}(\T)
 \xrightarrow{d} \cdots \xrightarrow{d} \P_{r-n}\Lambda^n(\T) \to 0.
\end{equation*}
This complex was used extensively by Sullivan \cite{sullivan73,sullivan77} and is sometimes
referred to as the complex of Sullivan--Whitney forms \cite{baker}.  It was
introduced into the finite element literature in \cite{Demkowicz-Monk-Rachowicz-Vardapetyan}.
The intermediate complexes
involve both higher order Whitney spaces and full polynomial spaces.

Finally, we note that the degrees of freedom used to define the space
determine a canonical projection $I_h:C\Lambda^k(\Omega)\to \P_r\Lambda^k(\T)$.
Namely, $I_h\omega\in \P_r\Lambda^k(\T)$
is determined by
\begin{equation*}
 \int_f \tr_{T,f}(\omega-I_h\omega)\wedge\eta=0,\quad \eta\in
   \P_{r+k-\dim f}^-\Lambda^{\dim f-k}(f), \ f\in \Delta(\T).
\end{equation*}
Similar considerations apply to $\P_r^-\Lambda^k(\T)$.  The canonical projection
can be viewed as a map from the (sufficiently) smooth de~Rham complex to
any one of the complexes built from the maps \eqref{maptypes}.  The degrees of
freedom were chosen exactly so that \emph{the
canonical projection is
a cochain map}, i.e., commutes with $d$.  This is verified using Stokes's theorem.
See \cite[Theorem 5.2]{acta}.

\begin{remark}
Given a simplicial triangulation $\T$, we have defined,
for every form degree $k$ and every polynomial degree $r$,
two finite element spaces of differential $k$-forms of degree at most $r$:
$\P_r\Lambda^k(\T)$ and $\P_r^-\Lambda^k(\T)$.
In $n=2$ and $n=3$ dimensions, we may use proxy fields to identify these spaces
of finite element differential forms with finite element spaces of
scalar and vector functions.  The space $\P_r \Lambda^0(\T)$ corresponds to
the Lagrange elements \cite{Ciarlet}, and the spaces $\P_r \Lambda^n(\T)$ and
$\P_r^- \Lambda^n(\T)$ correspond to the space of discontinuous piecewise
polynomials of degree $\le r$ and $\le r-1$, respectively. When $n=2$, the
spaces $\P_r \Lambda^1(\T)$ and $\P_r^-\Lambda^1(\T)$ correspond to the
Brezzi--Douglas--Marini $H(\div)$ elements of degree $\le r$, introduced in
\cite{Brezzi-Douglas-Marini} and the Raviart--Thomas $H(\div)$ elements
of degree $\le r-1$ introduced in \cite{Raviart-Thomas}.  These spaces
were generalized to three dimensions by N\'ed\'elec \cite{Nedelec1},
\cite{Nedelec2}.  The spaces $\P_r\Lambda^1(\T)$, $\P_r\Lambda^2(\T)$,
$\P_r^-\Lambda^1(\T)$, and $\P_r^-\Lambda^2(\T)$ then correspond to the
N\'ed\'elec 2nd kind $H(\curl)$ and $H(\div)$ elements of degree $\le r$ and
the N\'ed\'elec 1st kind $H(\curl)$ and $H(\div)$ elements of degree $\le
r-1$, respectively.
\end{remark}

\subsection{Computational bases}
\label{subsec:bases}
This subsection relates to the implementation of the $\P$ and $\P^-$ families
of finite element differential forms.  It is not essential to the rest of the paper.
Because the spaces $\P_r^-\Lambda^k(\T)$
and $\P_r\Lambda^k(\T)$ were constructed through a finite element assembly procedure
(shape functions and degrees of freedom), we have at hand a basis for their dual spaces.
Consider the space 
$\P_r^-\Lambda^k(\T)$, for example, with reference to the decomposition of the dual space
of the shape function space $\P_r^-\Lambda^k(T)$ given in Theorem~\ref{stardecomp}.
Choose any $f\in \Delta(\T)$
of dimension $d\ge k$, and choose any convenient basis for $\P_{r+k-d}\Lambda^{d-k}(f)$.
For each of the basis functions $\eta$, we obtain an element of $\P_r^-\Lambda^k(\T)^*$:
\begin{equation*}
 \omega\mapsto \int_f\tr_f\omega \wedge\eta,\quad \omega\in \P_r^-\Lambda^k(\T).
\end{equation*}
(This is meaningful since the integral is, by construction, single-valued.)
Taking the union over $f\in\Delta(\T)$ of the sets of elements of $\P_r^-\Lambda^k(\T)^*$
obtained in this way, gives a basis for that space.  An interesting case is that
of the Whitney forms $\P_1^-\Lambda^k(\T)$.  In this case there is exactly one dual basis
function for each $f\in\Delta_k(\T)$, namely $\omega\mapsto\int_f\tr_f\omega$.

For computation with finite elements, we also need a basis for the finite element
space itself, not only for the
dual space.  One possibility which is commonly used is to use the dual basis 
to the basis for the dual space just discussed.  Given an element of the basis of the dual
space, associated to some $f\in\Delta(\T)$, it is easy to check that the corresponding
basis function of the finite element space vanishes on all simplices $T\in \T$ which do not
contain $f$.  Thus we have a \emph{local basis}, which is very efficient for computation.

In the case of the Whitney space $\P_1^-\Lambda^k(\T)$, the dual basis can be written
down very easily.  We begin with the standard dual basis for $\P_1\Lambda^0(\T_h)$:
the piecewise linear function $\lambda_i$ associated to the vertex $x_i$ is determined
by $\lambda_i(x_j)=\delta_{ij}$ (after picking some numbering $x_1,\ldots,x_N$ of the
vertices).  Then given a
$k$-face $f$ with vertices $x_{\sigma(0)},\ldots,x_{\sigma(k)}$ we define the \emph{Whitney form}
\cite[p.~229]{whitney}
\begin{equation*}
 \phi_f = \sum_{i=0}^k (-1)^i\lambda_{\sigma(i)}\,
d\lambda_{\sigma(0)}\wedge\cdots\wedge\widehat{d\lambda_{\sigma(i)}}\wedge\cdots
\wedge d\lambda_{\sigma(k)}.
\end{equation*}

For higher degree finite element spaces it does not seem possible to write down the
dual basis explicitly, and it must be computed.  The computation comes down to
inverting a matrix of size $d\times d$ where
$d$ is the dimension of the space of shape functions.
This can be carried out once on a single reference simplex and the result
transferred to any simplex via an affine transformation.

An alternative to the dual basis, which is often preferred, is to use a basis for the finite element
space which can be written explicitly in terms of barycentric coordinates.
In particular, for the Lagrange finite element space $\P_r\Lambda^0(\T)$ consisting
of continuous piecewise polynomials of degree at most $r$, one often uses the
\emph{Bernstein basis}, defined piecewise by monomials in the barycentric coordinates,
instead of the dual or \emph{Lagrange basis}.  It turns out that explicit bases analogous
to the Bernstein basis can be given for all the finite element spaces in
the $\P$ and $\P^-$ families, as was shown in \cite{decomp}.  Here we content ourselves
with displaying a few typical  cases.

Bases for the spaces $\P^-_r \Lambda^1(\T)$ and $\P_r \Lambda^2(\T)$ are
summarized in Tables~\ref{tb:t2} and \ref{tb:t5}, respectively, for $n=3$ dimensions and polynomial
degrees
$r=1$, $2$, and $3$.  To explain the presentation, we interpret the second line
of Table~\ref{tb:t2}.  We are assuming that $\T$ is a triangulation of
a 3-dimensional polyhedron.   The table indicates that for the space $\P_2^-\Lambda^1(\T)$,
there are two basis functions associated to each edge of $\T$ and two basis functions
associated to each
$2$-dimensional face of $\T$.  If an edge has vertices $x_i$ and $x_j$ with $i<j$,
then on an simplex $T$ containing the edge, the corresponding basis functions
are given by $\lambda_i\phi_{ij}$ and $\lambda_j\phi_{ij}$, where $\lambda_i$ and $\lambda_j$ are the
barycentric coordinates functions on $T$ equal to $1$ at the vertices $x_i$ and $x_j$,
respectively, and
\begin{equation*}
\phi_{ij} = \lambda_i \,d \lambda_j - \lambda_j\, d \lambda_i
\end{equation*}
is the Whitney form associated to the edge.  Similarly, if $T$ contains a face
with vertices $x_i$, $x_j$, $x_k$, $i<j<k$, then on $T$ the two basis functions
associated to the face are given by $\lambda_k\phi_{ij}$ and $\lambda_j\phi_{ik}$.

\begin{table}[htb]
\caption{Basis for the spaces $\P^-_r \Lambda^1$, $n =3$.}
\label{tb:t2}
\footnotesize
\begin{center}
\begin{tabular}{c c c c c c c}
\hline
\\
$r$ & Edge $[x_i,x_j]$ & Face $[x_i,x_j,x_k]$ & Tet $[x_i,x_j,x_k,x_l]$
\\
\hline
$1$\rule{0pt}{15pt} & $\phi_{ij}$ &   & 
\\[2ex]
$2$ & $\{\lambda_i, \lambda_j\} \phi_{ij}$  &
$\lambda_k \phi_{ij}$,  \ $\lambda_j \phi_{ik}$ & 
\\[2ex]
$3$ & $\{\lambda_i^2, \lambda_j^2, \lambda_i \lambda_j\} \phi_{ij}$ &
$\{\lambda_i, \lambda_j, \lambda_k\} \lambda_k \phi_{ij}$
& $\lambda_k \lambda_l \phi_{ij}$, \ $\lambda_j \lambda_l \phi_{ik}$,
\\
&   & $\{\lambda_i, \lambda_j, \lambda_k\} \lambda_j \phi_{ik}$
& $\lambda_j \lambda_k \phi_{il}$
\\
\\
\hline
\end{tabular}
\end{center}
\end{table}

\begin{table}[htb]
\caption{Basis for the space $\P_r \Lambda^2$, $n = 3$.}
\label{tb:t5}
\footnotesize
\begin{center}
\begin{tabular}{c c c c c}
\hline
\\
$r$ &\hspace{.1in} & Face $[x_i,x_j,x_k]$ &\hspace{.1in} & Tet $[x_i,x_j,x_k,x_l]$ \\
\hline
$1$\rule{0pt}{15pt} && $\lambda_k d \lambda_i \wedge d \lambda_j$, \
$\lambda_j d \lambda_i \wedge d \lambda_k$, \
$\lambda_i d \lambda_j \wedge d \lambda_k$  &&
\\[2ex]
$2$ && $\lambda_k^2 d \lambda_i \wedge d \lambda_j$, \
$\lambda_j \lambda_k d \lambda_i \wedge  d(\lambda_k - \lambda_j)$
&& $\lambda_k \lambda_l d \lambda_i \wedge d \lambda_j$, \
$\lambda_j \lambda_l d \lambda_i \wedge d \lambda_k$
\\
&&  $\lambda_j^2 d \lambda_i \wedge d \lambda_k$, \
$\lambda_i \lambda_j d(\lambda_j - \lambda_i) \wedge d \lambda_k$
 && $\lambda_j \lambda_k d \lambda_i \wedge d \lambda_l$, \
$\lambda_i \lambda_l d \lambda_j \wedge d \lambda_k$
\\
&&  $\lambda_i^2 d \lambda_j \wedge d \lambda_k$, \
 $\lambda_i \lambda_k d \lambda_j \wedge d (\lambda_k - \lambda_i)$
&& $\lambda_i \lambda_k d \lambda_j \wedge d \lambda_l$, \
$\lambda_i \lambda_j d \lambda_k \wedge d \lambda_l$
\\[2ex]
$3$ && $\lambda_k^3 d \lambda_i \wedge d \lambda_j$, \
$\lambda_j^3 d \lambda_i \wedge d \lambda_k$, \
$\lambda_i^3 d \lambda_j \wedge d \lambda_k$
&& $\{\lambda_k, \lambda_l\} \lambda_k \lambda_l
d \lambda_i \wedge d \lambda_j$
\\
&&  $\lambda_j^2  \lambda_k d \lambda_i
\wedge d (2 \lambda_k - \lambda_j)$, \
$\lambda_j \lambda_k^2 d \lambda_i \wedge d (\lambda_k - 2 \lambda_j)$
&& $\{\lambda_j, \lambda_k, \lambda_l\} \lambda_j \lambda_l
d \lambda_i \wedge d \lambda_k$
\\
&& $\lambda_i^2 \lambda_j d(2 \lambda_j - \lambda_i) \wedge d \lambda_k$, \
$\lambda_i^2 \lambda_k d \lambda_j \wedge d(2\lambda_k - \lambda_i)$
&& $\{\lambda_j, \lambda_k, \lambda_l\} \lambda_j \lambda_k
d \lambda_i \wedge d \lambda_l$
\\
&& $\lambda_i \lambda_j^2 d(\lambda_j - 2 \lambda_i) \wedge d \lambda_k$, \
$\lambda_i \lambda_k^2 d \lambda_j \wedge d (\lambda_k - 2 \lambda_i)$
&&
$\{\lambda_i, \lambda_j, \lambda_k, \lambda_l\} \lambda_i \lambda_l
d \lambda_j \wedge d \lambda_k$
\\
&& $\lambda_i \lambda_j \lambda_k d (2 \lambda_j - \lambda_i - \lambda_k)
\wedge  d (2 \lambda_k - \lambda_i - \lambda_j)$ &&
$\{\lambda_i, \lambda_j, \lambda_k, \lambda_l\} \lambda_i \lambda_k
 d \lambda_j \wedge d \lambda_l$
\\
&& && 
$\{\lambda_i, \lambda_j, \lambda_k, \lambda_l\} \lambda_i \lambda_j
 d \lambda_k \wedge d \lambda_l$
\\
\\
\hline
\end{tabular}
\end{center}
\end{table}

\subsection{Approximation properties}
\label{subsec:approx}
To apply the general theory for approximation of Hilbert
complexes to the finite element exterior calculus, outlined above, we
need to construct bounded cochain projections from
$L^2\Lambda^k(\Omega)$ onto the various finite element spaces. The error estimates given
in Theorems~\ref{qo1cor} and \ref{impest} bound the error in the finite element solution in terms of the
approximation error afforded by the subspaces.  In this subsection, we show that the spaces 
$\P_r\Lambda^k(\T)$ and $\P^-_r\Lambda^k(\T)$  provide optimal order approximation
of differential $k$-forms as the mesh size tends to zero, where optimal order
means that the rate of convergence obtained
is the highest possible given the degree of the piecewise polynomials
and the smoothness of the form being approximated.
In the next subsection, we shall construct $L^2$-bounded cochain projections
which attain the same accuracy.

Let $\{\T_h\}$ be a family of simplicial triangulations of
$\Omega \subset \R^n$,
indexed by decreasing values of \emph{the mesh parameter} $h$ given as 
$h = \max_{T \in \T_h} \operatorname{diam} T$.
We will assume throughout that the family $\{ \T_h \}$ is \emph{shape
  regular},
i.e., that the ratio of the volume of the
circumscribed to the inscribed ball associated to any element $T$ is bounded uniformly
for all the simplices in all the triangulations of the family.
Parts of the construction below simplify if we assume, in addition, that the family $\{ \T_h \}$ is
\emph{quasi-uniform}, i.e., that the ratio $h/\operatorname{diam}T$ is
bounded for all $T\in\T_h$, uniformly over the family.
However, we do \emph{not} require quasi-uniformity,
only shape-regularity.

Throughout this and the following subsection, $\Lambda^{k-1}_h$ will denote a subspace of
$H\Lambda^{k-1}(\Omega)$ and $\Lambda^k_h$ a subspace of $H\Lambda^k(\Omega)$.
Motivated by Section~\ref{subsubsec:seqs}, for each $r=0,1,\dots$,
we consider the following possible pairs of spaces:
\begin{equation}\label{spaces}
 \Lambda^{k-1}_h= \left\{
\begin{matrix}
 \P_{r+1}\Lambda^{k-1}(\T)\\[1ex] \text{or} \\[1ex] \P_{r+1}^-\Lambda^{k-1}(\T)
\end{matrix}
\right\},\quad
\Lambda^k_h=
 \left\{
\begin{matrix}
 \P_{r+1}^-\Lambda^k(\T)\\[1ex] \text{or} \\[1ex] \P_r\Lambda^k(\T) \text{ (if $r>0$)}
\end{matrix}
\right\}.
\end{equation}
Note that $\Lambda^k_h$ contains all polynomials of degree $r$, but not $r+1$ in
each case.
As we have seen in Section~\ref{subsec:dof}, 
the degrees of freedom define canonical projection operators
$I_h = I_h^k$ mapping $C\Lambda^k(\Omega)$ boundedly to $\Lambda^k_h$ and commuting with the
exterior derivative. However, the operators $I_h$ do not extend boundedly
to all of $L^2\Lambda^k(\Omega)$, or even to
$H\Lambda^k(\Omega)$ and so do not fulfill the requirements of
bounded cochain projections of the abstract theory developed in Section~\ref{sec:HCA}.
In this subsection, we will describe the Cl\'{e}ment interpolant, which
is a modification of the
canonical projection which is bounded on $L^2\Lambda^k$ and can be used to establish the
approximation properties of the finite element space.  However, the Cl\'{e}ment interpolant is
not a projection and does not commute with the exterior derivative.
So in the next subsection we construct
a further modification which regains these properties and maintains the approximation 
properties of the Cl\'ement interpolant.

The difficulty with the canonical interpolation operators $I_h$
is that they make use of traces onto lower dimensional simplexes,
and as a consequence, they cannot be extended boundedly to $L^2\Lambda^k(\Omega)$
(except for $k=n$).
The Cl\'{e}ment interpolant \cite{clement} is a classical tool of finite element theory
developed to overcome this problem in the case of $0$-forms.
To define this operator (for general $k$), we need some additional
notation. For any $f \in\Delta(\T_h)$,
we let $\Omega_f \subset \Omega$ be the union of elements containing $f$:
\[
\Omega_f = \bigcup \{T \, | \, T \in \T_h, \, f \in \Delta(T) \, \},
\]
and let $P_f:L^2\Lambda^k(\Omega_f)\to \P_r\Lambda^k(\Omega_f)$ be the $L^2$ projection
onto polynomial $k$-forms of degree at most $r$.  For $\omega\in L^2\Lambda^k(\Omega)$,
we determine the Cl\'ement interpolant
$\tilde I_h\omega:=\tilde I_h^k \omega \in \Lambda^k_h$ by specifying $\phi(\tilde I_h\omega)$
for each degree of freedom
$\phi$ of the space $\Lambda^k_h$; see Section~\ref{subsec:dof}.
Namely, if $\phi$ is a degree of freedom for $\Lambda^k_h$ associated to $f\in\Delta(\T_h)$, we take
\[
\phi(\tilde I_h \omega) = \phi(P_f\omega_f),
\]
where $\omega_f = \omega|_{\Omega_f}$.

The Cl\'{e}ment interpolant is local in the sense that
for any $T \in T_h$,
$\tilde I_h\omega|_T$ is determined by
$\omega|_{T^*}$, where $T^* = \bigcup \{\Omega_f \,|\, f \in \Delta(T) \, \}$.
In fact, by scaling, as in \cite[Section 5.3]{acta}, for example, it can be seen that
$$
\|\tilde I_h\omega\|_{L^2\Lambda^k(T)}\le
c_0\sum_{f\in\Delta(T)}\| P_h\omega \|_{L^2\Lambda^k(\Omega_f)}\le c_1 \| \omega \|_{L^2\Lambda^k(T^*)},
$$
where the constants $c_0,c_1$ may depend on the polynomial degree $r$ and the dimension $n$,
but are independent of $T\in\T_h$ and $h$, thanks to the shape regularity assumption.
Therefore, the Cl\'{e}ment interpolant is uniformly bounded in 
$\Lin(L^2\Lambda^k(\Omega),L^2\Lambda^k(\Omega))$:
\begin{align*}
\|\tilde I_h \omega\|_{L^2\Lambda^k(\Omega)}  \le c_2\| \omega \|_{L^2\Lambda^k(\Omega)}
\end{align*}
(with the constant independent of $h$).

Another key property of the Cl\'{e}ment interpolant is that it
preserves polynomials locally in the sense that $\tilde I_h \omega|_T
= \omega|_T$ if $\omega \in \P_r\Lambda^k(T^*)$.  It is then standard,
using the Bramble--Hilbert lemma \cite{bramble-hilbert}
and scaling, to obtain the following error estimate.
\begin{thm}\label{clement}
Assume that $\Lambda^k_h$ is either $\P_{r+1}^-\Lambda^k(\Th)$, or, if $r\ge 1$,
$\P_r\Lambda^k(\Th)$.
Then there is a constant $c$, independent of $h$, such that the 
Cl\'{e}ment interpolant $\tilde I_h^k :L^2\Lambda^k(\Omega) \to \Lambda^k_h$
satisfies the bound
\begin{equation*}
\|\omega-\tilde I_h^k\omega\|_{L^2\Lambda^k(\Omega)} \le
ch^s|\omega|_{H^s\Lambda^k(\Omega)}, \quad
\omega\in H^s\Lambda^k(\Omega),
\end{equation*}
for $0 \le s \le r+1$.
\end{thm}
Note that the estimate implies that any sufficiently smooth $k$-form is approximated
by elements $\Lambda^k_h$ with order $O(h^{r+1})$ in $L^2$.  Since the polynomial spaces
used to construct $\Lambda^k_h$ contain $\P_r\Lambda^k$ but not $\P_{r+1}\Lambda^k$,
this is the optimal order of approximation.

\subsection{Bounded cochain projections}
\label{subsec:bcp}
The Cl\'{e}ment interpolant $\tilde I_h$
is both uniformly bounded in $L^2$ and gives 
optimal error bounds for smooth functions. However, it is not a bounded cochain
projection in the sense of the theory of Section 3. Indeed, it is neither a
projection operator---it does not leave $\Lambda^k_h$ invariant---nor does it
commute with the exterior derivative.  Therefore, to construct bounded
cochain projections, we consider another modification of the canonical
projection $I_h$, in which the operator $I_h$ is
combined with a smoothing operator.

This construction---key ingredients of which were contributed by Sch\"oberl
\cite{schoberl05} and Christiansen \cite{christiansen05}---was discussed in
detail in \cite{acta} (where it was called a \emph{smoothed projection}),
under the additional assumption that the family of triangulations $\{\T_h\}$
is quasi-uniform, and then in \cite{winther-c} in the general shape regular
case.  Therefore, we will just give a brief outline of this construction here.

Let $\Lambda^k_h$ be one of the spaces $\P_{r+1}^-\Lambda^k(\T_h)$ or, if $r\ge1$,
$\P_r\Lambda^k(\T_h)$, and let $I_h:C\Lambda^k(\Omega)\to \Lambda^k_h$
be the corresponding canonical interpolant.
To define an appropriate smoothing operator, we let $\rho: \R^n \to \R$
be a nonnegative smooth function with support in the unit ball and
with integral equal to one. In the quasi-uniform case, we utilize 
a standard convolution operator of the form
\[
\omega \mapsto \int_{|y|\le 1} \rho(y)\omega(x + \delta y) \,
dy,
\]
mapping $L^2\Lambda^k(\tilde \Omega)$ into
$C^\infty\Lambda^k(\Omega)$, where the domain $\tilde \Omega \supset
\Omega$ is sufficiently large so that the convolution operator is well
defined.  We set the smoothing parameter $\delta =\epsilon h$, where
the proportionality constant $\epsilon > 0$ is a parameter to be
chosen. In the more general shape regular case, we need to generalize
this operator slightly.  As in \cite{winther-c}, we introduce a
Lipschitz continuous function $g_h : \Omega \to \R^+$, with Lipschitz
constant bounded uniformly in $h$, such that $g_h(x)$ is uniformly
equivalent to $\operatorname{diam} T$ for $T \in \T_h$ and $x \in
T$. Hence, $g_h|_T$ approximates $\operatorname{diam} T$.  The
appropriate smoothing operator in the general case, mapping
$L^2\Lambda^k(\tilde\Omega)$ into $C\Lambda^k(\Omega)$, is now given
by
\[
\omega \mapsto \int_{|y|\le 1} \rho(y)((\Phi_h^{\epsilon y})^*\omega)_x \,
dy,
\]
where the map $\Phi_h^{\epsilon y}: \R^n \to \R^n$ is defined by 
\[
\Phi_h^{\epsilon y}(x) = x + \epsilon g_h(x)y.
\]
Since this smoothing operator is defined as an average of
pullbacks,
it will indeed commute with the exterior derivative.
By combining it with an appropriate extension operator
$E : L^2\Lambda^k(\Omega) \to L^2\Lambda^k(\tilde\Omega)$,
constructed such that it commutes with the exterior derivative,
we obtain an operator $R_h^\epsilon : L^2\Lambda^k(\Omega) \to
C\Lambda^k(\Omega)$. The operator $Q_h^\epsilon = I_h \circ
R_h^\epsilon$ maps $L^2\Lambda^k(\Omega)$ into $\Lambda_h^k$,
and commutes with the exterior derivative. Furthermore, for each
$\epsilon > 0$, the operators $Q_h^\epsilon$ are bounded in $L^2$,
uniformly in $h$.

However, the operator $Q_h^\epsilon$ is not invariant on the
subspace $\Lambda_h^k$. The remedy to fix this is to establish that
the operators $Q_h^\epsilon|_{\Lambda_h^k} : \Lambda_h^k \to
\Lambda_h^k$ converge to the identity in the $L^2$ operator norm as
$\epsilon$ tends to zero,
uniformly in $h$. Therefore, for a fixed $\epsilon > 0$, taken
sufficiently small,  this operator
has an inverse $J_h^\epsilon : \Lambda_h^k \to
\Lambda_h^k$. 
The desired projection operator, $\pi_h=\pi_h^k$, is now given
as $\pi_h = J_h^\epsilon\circ Q_h^\epsilon$. The operators
$\pi_h$ are uniformly bounded with respect to $h$ as operators
in $\Lin(L^2\Lambda^k(\Omega), L^2\Lambda^k(\Omega))$. Furthermore, 
since they are projections onto $\Lambda_h^k$, we obtain 
\begin{equation}\label{x}
\|\omega-\pi_h\omega\|_{L^2\Lambda^k(\Omega)}\le 
\inf_{\mu \in \Lambda_h^k} \| (I - \pi_h)(\omega - \mu)
\|_{L^2\Lambda^k(\Omega)}.
\end{equation}
Based on these considerations, we obtain the
the following theorem. Cf.~\cite[Theorem
5.6]{acta} and \cite[Corollary 5.3]{winther-c}.
\begin{thm}\label{tilde_pi}
1.~Let $\Lambda^k_h$ be one of the spaces $\P_{r+1}^-\Lambda^k(\T_h)$ or, if $r\ge1$,
$\P_r\Lambda^k(\T_h)$ and $\pi^k_h:L^2\Lambda^k(\Omega)\to \Lambda^k_h$ the
smoothed projection operator constructed above.  Then $\pi^k_h$ is a projection onto  $\Lambda^k_h$
and satisfies
\begin{equation*}
\|\omega-\pi_h^k\omega\|_{L^2\Lambda^k(\Omega)}\le ch^s
\|\omega\|_{H^s\Lambda^k(\Omega)}, \quad \omega\in H^s\Lambda^k(\Omega),
\end{equation*}
for $0\le s\le r+1$. Moreover, for all $\omega\in L^2\Lambda^k(\Omega)$,
 $\pi_h^k\omega\to\omega$ in $L^2$ as $h\to0$.

2.~Let $\Lambda^k_h$ be one of the spaces $\P_{r}\Lambda^k(\T_h)$ or
$\P_r^-\Lambda^k(\T_h)$ with $r\ge1$.  Then
\begin{equation*}
\|d(\omega-\pi_h^k\omega)\|_{L^2\Lambda^k(\Omega)}\le ch^s
\|d\omega\|_{H^s\Lambda^k(\Omega)}, \quad \omega\in H^s\Lambda^k(\Omega),
\end{equation*}
for $0\le s\le r$.

3.~Let $\Lambda^{k-1}_h$ and $\Lambda^k_h$ be taken as in \eqref{spaces}
and $\pi_h^{k-1}$ and $\pi_h^k$ the corresponding smoothed projections.
Then
$d\pi^{k-1}_h=\pi^k_hd$.
\end{thm}
\begin{proof}
We have already established that the $\pi^k_h$ are uniformly bounded projections
which commute with $d$. The error estimate in the first statement of the theorem
then follows from \eqref{x}.  Statement 2 follows from 1 and 3.
\end{proof}

\subsection{Approximation of the de Rham complex and the Hodge Laplacian}
\label{subsec:ahl}
Let $\Lambda^n_h$ be the space of piecewise polynomial $n$-forms of degree
at most $s$ for some nonnegative integer $s$.  We showed
in Section~\ref{subsec:dof} that there are $2^{n-1}$ distinct
discrete de Rham complexes
\begin{equation}
\label{disc-derham1}
0 \to \Lambda_h^0 \xrightarrow{d} \Lambda_h^0 \xrightarrow{d} \cdots
\xrightarrow{d} \Lambda_h^n \to 0,
\end{equation}
with each of the mapping $\Lambda_h^{k-1} \xrightarrow{d} \Lambda_h^{k}$
being of the form \eqref{maptypes}
for an appropriate choice of $s$.
Making use of the bounded projections, we obtain in each case
a commuting diagram
\[
\begin{CD}
0\to\,@. H\Lambda^0(\Omega) @>\D>> H\Lambda^1 (\Omega)@>\D>>
\cdots @>\D>> H\Lambda^n(\Omega)@.\,\to0\\
 @. @VV\pi_h V @VV\pi_hV @.  @VV\pi_hV\\
0\to\,@. \Lambda^0_h @>\D>> 
\Lambda^1_h  @>\D>>
\cdots @>\D>> \Lambda^n_h @.\,\to 0.
\end{CD}
\] 
Thus the projections give a cochain projection from the de~Rham complex to
the discrete de~Rham complex, and so induce a surjection on cohomology.
Since $\lim_{h \rightarrow 0} \|\omega - \pi_h w\|_{H \Lambda^k(\Omega)}=0$,
we can use Theorem~\ref{isom-cohom} (with $V=H\Lambda$ and $V_h=\Lambda_h$) to see immediately that for $h$
sufficiently small, this is in each case an isomorphism on cohomology.

The simplest finite element de~Rham complex is  the complex of Whitney forms,
\begin{equation}\label{whitf}
0\to \P_1^-\Lambda^0(\Th)
\xrightarrow{d} \P_1^-\Lambda^1(\Th) \xrightarrow{d} \cdots \xrightarrow{d}
\P_1^-\Lambda^n(\Th)\to 0.
\end{equation}
Via the basic Whitney forms $\phi_f$,
the space $\P_1^-\Lambda^k(\T_h)$ is isomorphic to the space of simplicial cochains of
dimension $k$ associated to the triangulation $\T_h$.  
In this way, the complex \eqref{whitf} of Whitney forms is isomorphic, as a cochain complex, to
the simplicial cochain complex, and its cohomology is isomorphic to the
simplicial cohomology of $\Omega$.  Thus Theorem~\ref{isom-cohom}, when applied to the Whitney forms,
gives an isomorphism between the de Rham cohomology and simplicial cohomology for a sufficiently
fine triangulation.  Since the simplicial cohomology
is independent of the triangulation (being isomorphic to the singular cohomology),
for the complex of Whitney
forms, the isomorphism on cohomology
given by Theorem~\ref{isom-cohom} holds for any triangulation, not just  one sufficiently fine.
This proves de~Rham's theorem on the equality of de~Rham and simplicial cohomology.

For all of the discrete complexes \eqref{disc-derham1}, as for the Whitney forms,
the cohomology is independent of the triangulation (and equal to the de~Rham cohomology).
To see this, 
following \cite{christiansen05}, we
consider the Whitney forms complex \eqref{whitf} as a subcomplex
of \eqref{disc-derham1}.  The canonical projections
$I_h$ define cochain projections.  Note that the
$I_h$ are defined on the finite element spaces $\Lambda^k_h$,
because all of the trace moments they require are single-valued
on $\Lambda^k_h$.  From the commuting diagram
\begin{equation*}
\begin{CD}
0\to\,@. \Lambda_h^0 @>\D>> \Lambda_h^1 @>\D>>
\cdots @>\D>> \Lambda_h^n @.\,\to0\\
 @. @VV I_hV @VV I_hV @.  @VV I_hV\\
0\to \,@. \P_1^-\Lambda^0(\Th) @>\D>> 
\P_1^-\Lambda^1(\Th) @>\D>>
\cdots @>\D>> \P_1^-\Lambda^n(\Th) @.\,\to0,
\end{CD}
\end{equation*}
we conclude that the cohomology of the top row, which we have
already seen to be an image of the de~Rham cohomology, maps onto
the cohomology of the bottom row, which is isomorphic
to the de~Rham cohomology.  Hence the dimension of all
the corresponding cohomology groups are equal and both cochain
projections induce an isomorphism on cohomology.

Consider now the numerical solution of the Hodge Laplacian problem
\eqref{pdes}.  We approximate this by a Galerkin method: Find
$\sigma_h\in\Lambda^{k-1}_h$, $u_h\in\Lambda^k_h$, and $p_h
\in \Hfrak^k_h$ such that
\begin{equation*}
\begin{aligned}
\<\sigma_h,\tau\> - \<
 d \tau,u_h\> &= 0,  & \tau&\in \Lambda^{k-1}_h,
\\*
\<  d \sigma_h,v\>
+ \< d  u_h, d  v\> + \< v, p_h\>
&= \< f,v\>, &
v&\in \Lambda^k_h,
\\*
\< u_h, q \>  &=0,
&q&\in \Hfrak^k_h.
\end{aligned}
\end{equation*}
For the finite element spaces, we may choose $\Lambda^{k-1}_h$ and $\Lambda^k_h$
to be any of the pairs given in \eqref{spaces}.
We have verified that $d$ maps $\Lambda^{k-1}_h$ into $\Lambda^k_h$ and this
map can be extended to a full subcomplex of the de~Rham complex admitting a bounded
cochain projection.  Therefore we may combine the error estimates derived in the abstract
setting in Sections~\ref{subsec:improved} and \ref{subsec:eigenv}
and approximation estimates from Section~\ref{subsec:bcp} to obtain convergence and rates of
convergence.  The rates of convergence are limited by three factors: (1)~the smoothness of
the data $f$, (2)~the amount of elliptic regularity (determined by the smoothness of the domain),
and (3)~the degree of complete polynomials contained in the finite element shape
functions. Specifically, from Theorem~\ref{impest} we get the following result.
Assume the regularity estimate
\begin{equation*}
\|u\|_{H^{t+2}(\Omega)} + \|p\|_{H^{t+2}(\Omega)} + \|du\|_{H^{t+1}(\Omega)}
+ \|\sigma\|_{H^{t+1}(\Omega)} + \|d\sigma\|_{H^{t}(\Omega)}
\le C \|f\|_{H^{t}(\Omega)},
\end{equation*}
holds for $0 \le t \le t_{\text{max}}$.  Then for $0\le s\le t_{\text{max}}$
the following estimates hold:
\begin{equation*}
 \|d(\sigma-\sigma_h)\| \le Ch^s \|f\|_{H^s(\Omega)}, \text{ if $s\le r+1$},
\end{equation*}
\begin{equation*}
\|\sigma-\sigma_h\| \le 
Ch^{s+1}\|f\|_{H^{s}(\Omega)}, \text{ if }\begin{cases} s \le r+1, & \Lambda_h^{k-1} =
\P_{r+1}\Lambda^{k-1}(\T), \\ s \le r, & \Lambda_h^{k-1} =
\P_{r+1}^-\Lambda^{k-1}(\T), \end{cases}
\end{equation*}
\begin{equation*}
 \|d(u-u_h)\| \le Ch^{s+1}\|f\|_{H^{s}(\Omega)},  \text{ if }\begin{cases}
s \le r,& \Lambda_h^{k} = \P_{r+1}^-\Lambda^{k}(\T), \\
s \le r-1,&\Lambda_h^{k} = \P_{r}\Lambda^{k}(\T),
                             \end{cases}
\end{equation*}
\begin{equation*}
 \|u-u_h\| + \|p - p_h\| \le \begin{cases}
Ch\|f\|, & \Lambda^k_h=\P_1^-\Lambda^{k}(\T), \\
Ch^{s+2}\|f\|_{H^{s}(\Omega)} \text{ if $s\le r-1$},& \text{otherwise}.
                             \end{cases}
\end{equation*}
Thus, the error in each case is the optimal
order allowed by the subspace if we have sufficient elliptic
regularity.

Concerning the eigenvalue problem, applying Theorem~\ref{ev-basicest}, we immediately
obtain the following error estimate:
\begin{equation*}
|\lambda - \lambda_h| \le C h^{2s+2} \|u\|_{H^{s}(\Omega)}^2, \quad
0 \le s \le r-1.
\end{equation*}
This estimate does not apply in case $\Lambda^k_h
=\P_1^-\Lambda^k(\T)$, i.e., the Whitney forms.  In that case, we get
the estimate:
\begin{equation*}
|\lambda - \lambda_h| \le \begin{cases}
C h \|u\|^2, \\
C h^2 \|u\| \|u\|_{H^{1}(\Omega)}.
\end{cases}
\end{equation*}

\section{Variations of the de Rham complex}
\label{sec:variations}
In Sections~\ref{sec:deRham} and \ref{sec:FEDF}, we applied the abstract theory of Section~\ref{sec:HCA}
to the Hilbert complex obtained by choosing
$W^k$ to be $L^2\Lambda^k(\Omega)$ and $d^k$ to be the exterior derivative with domain
$V^k=H\Lambda^k(\Omega)$.  This led to finite element approximations of certain boundary
value problems for the Hodge Laplacian, and a variety of related problems
($\Bfrak$ and $\Bfrak^*$ problems, eigenvalue problems).  In this section,
we show how some small variations of these choices lead to methods for important related
problems, namely generalizations of the Hodge Laplacian with variable coefficients
and problems with essential boundary conditions.

\subsection{Variable coefficients}
\label{subsec:vc}
In this section, $d^k$ is again taken to be the
exterior derivative with domain $V^k=H\Lambda^k(\Omega)$ and
$W^k$ is again taken to be the space $L^2\Lambda^k$, but $W^k$ is
furnished with an inner product
which is equivalent to, but not equal to, the standard $L^2\Lambda^k$ inner product.  Specifically,
we let $a^k:W^k\to W^k$ be a bounded, symmetric, positive definite operator
and choose
\begin{equation*}
 \<u,v\>_{W^k}=\<a^ku,v\>_{L^2\Lambda^k}.
\end{equation*}
With this choice, we can check that $d^*_k=(d^{k-1})^*$ is given by
\begin{equation*}
 d^*_k u=(a^{k-1})^{-1}\delta_k a^k u,
\end{equation*}
with the domain of $d^*_k$ given by
\begin{equation*}
 V^*_k=\{\,u\in L^2\Lambda^k(\Omega)\,|\, a^k u\in \0H^*\Lambda^k(\Omega)\,\}
 = (a^k)^{-1}\0H^*\Lambda^k(\Omega).
\end{equation*}
The harmonic forms are then
\begin{equation*}
 \Hfrak^k=\Zfrak^k\cap\Zfrak^*_k=
\{\,u\in H\Lambda^k\cap(a^k)^{-1}\0H^*\Lambda^k\,|\, du=0,\ \delta a^ku=0\,\}.
\end{equation*}
The Hodge Laplacian $d^*du+dd^*u$ is now realized as
\begin{equation*}
d(a^{k-1})^{-1}\delta a^k u +(a^k)^{-1}\delta a^{k+1} d u.
\end{equation*}

As an example of the utility of this generalization, we consider the $\Bfrak^*_1$
eigenvalue problem in $3$ dimensions.  We take $a^1=\varepsilon$
and $a^2=\mu^{-1}$ where $\varepsilon$ and $\mu$ are
symmetric positive definite $3\x 3$ matrix fields.
Then the $\Bfrak^*_1$ eigenvalue problem is
\begin{equation*}
 \curl \mu^{-1} \curl u = \lambda\varepsilon u.
\end{equation*}
This is the standard Maxwell eigenvalue problem where $\varepsilon$ is the
dielectric tensor of the material and $\mu$ is its magnetic permeability.
These are scalar for an isotropic material, but given by matrices in general.
They are constant for a homogeneous material, but vary if the material is
not constant across the domain.

We notice that the Hilbert complex in this case is exactly the de~Rham complex
\eqref{drl2}, except that the spaces are furnished with different, but equivalent,
inner products.  Thus the finite element de~Rham subcomplexes we constructed in Section~\ref{sec:FEDF}
apply for this problem, and satisfy the subcomplex and bounded cochain projection properties.
In other words, the same finite element spaces which were developed for the constant
coefficient Hodge Laplacian, can be applied equally
well for variable coefficient problems.

\subsection{The de Rham complex with boundary conditions}
\label{subsec:ebc}
For another important application of the abstract theory of Section~\ref{sec:HCA},
we again choose $W^k=L^2\Lambda^k(\Omega)$ and $d^k$ to be the exterior derivative,
but this time we take its domain $V^k$ to be the space $\0H\Lambda^k(\Omega)$.
This space is dense in $W^k$, and is a closed subspace of $H\Lambda^k$, so is complete,
and hence $d$ is again a closed operator with this domain.  Thus we obtain
the following complex of Hilbert spaces, called
the de Rham complex with boundary conditions:
\begin{equation}
\label{drl20}
0\to \0H\Lambda^0(\Omega) \xrightarrow{d} \0H\Lambda^1(\Omega) 
\xrightarrow{d} \cdots \xrightarrow{d}
\0H\Lambda^n(\Omega)\to 0.
\end{equation}
In Section~\ref{subsec:deRcHc},
we used the fact that $H\Lambda^k \cap \0H^* \Lambda^k$ is compactly included in $L^2 \Lambda^k$.
Replacing $k$ by $n-k$ and applying the Hodge star operator, we conclude
the analogous result, that $\0H\Lambda^k\cap H^*\Lambda^k$ is compactly included in $L^2\Lambda^k$.
Then, just as in Section~\ref{subsec:deRcHc}, 
all the results of Section~\ref{sec:HCA} apply to the de~Rham complex
with boundary conditions. 
In particular, we have the Hodge
decomposition of $L^2\Lambda^k$ and of $\0H\Lambda^k$, the Poincar\'e
inequality, well-posedness of the mixed formulation of the Hodge
Laplacian, and all the approximation results established in
Sections~\ref{subsec:approxhc}--\ref{subsec:eigenv}.

We denote the spaces of $k$-cocycles, $k$-coboundaries, and harmonic
$k$-forms as
\begin{gather*}
\0\Zfrak^k=\{\,\omega\in \0H\Lambda^k\,|\,d\omega=0\,\},\
\0\Bfrak^k=\{\,d\eta\,|\,\eta\in \0H\Lambda^{k-1}\,\},\\
\0\Hfrak^k=\{\,\omega\in \0\Zfrak^k\,|\,\langle\omega,\mu\rangle=0,
\ \mu\in\0\Bfrak^k\,\}.
\end{gather*}
Now, from \eqref{ibp2a}, we have that
\begin{equation}\label{ibp30}
 \langle d\omega,\mu\rangle = \langle \omega,\delta\mu\rangle,
\quad \omega\in \0H\Lambda^{k-1},\ \mu\in H^*\Lambda^k.
\end{equation}
Therefore
\begin{equation*}
 \0\Bfrak^{k\perp_{L^2}} := \{\,\omega\in L^2\Lambda^k\,|\,\<\omega,\D\eta\>
=0 \ \forall\eta\in \0H\Lambda^{k-1}\,\}
= \{\,\omega\in H^*\Lambda^k\,|\,\delta\omega=0\,\}.
\end{equation*}
This gives a concrete characterization of the harmonic forms analogous to \eqref{hfrak}:
\begin{equation}\label{hla0}
\0\Hfrak^k = \{\omega \in \0H\Lambda^k(\Omega) \cap H^*\Lambda^k(\Omega) \,
| \, d \omega =0, \delta \omega =0\}.
\end{equation}
Now, from \eqref{Hstar}, \eqref{H0star}, and \eqref{defdelta}, the Hodge star operator maps
the set of $\omega\in H\Lambda^{n-k}$
with $d\omega=0$ isomorphically onto the set of $\mu\in H^*\Lambda^k$ with $\delta\mu=0$,
and maps the set of $\omega\in\0\H^*\Lambda^{n-k}$ with $\delta\omega=0$ isomorphically
onto the set of $\mu\in \0H\Lambda^k$ with $d\mu=0$.  Comparing
\eqref{hfrak} and \eqref{hla0}, we conclude that the
\emph{Hodge star operator maps $\Hfrak^{n-k}$ isomorphically onto $\0\Hfrak^k$}.
This isomorphism is called \emph{Poincar\'e} duality.
In particular, $\dim\0\Hfrak^k$ is the $(n-k)$th Betti number of $\Omega$.

Finally, from \eqref{ibp30}, we can characterize $d^*$, the adjoint of the exterior
derivative $d$ viewed as an unbounded operator
$L^2\Lambda^{k-1}\to L^2\Lambda^k$ with domain $\0H\Lambda^k$.  Namely, in this case
$d^*$ has  domain $H^*\Lambda^k$ and coincides with the operator $\delta$ defined
in \eqref{defdelta}.  By contrast, when we took the domain of $d$ to be all of $H\Lambda^k$,
the domain of $d^*$ turned out to be the smaller space $\0H^*\Lambda^k$.

We now consider the Hodge Laplacian problem given in the abstract case by
\eqref{wfhc}. In this case, we get that $(\sigma,u,p)\in \0H\Lambda^{k-1}\x
\0H\Lambda^k\x\0\Hfrak^k$ is a solution if and only if
\begin{gather}\label{pdes0}
 \sigma=\delta u,\ d\sigma +\delta d u = f-p \quad\text{ in $\Omega$},
\\\label{bcs0}
\tr \sigma =0, \ \tr u =0 \quad\text{on $\partial\Omega$},
\\\label{sc0}
u\perp \0\Hfrak^k.
\end{gather}
Note that now both boundary conditions are \emph{essential} in the sense that
they are implied by membership in the spaces $\0H\Lambda^{k-1}$ and $\0H\Lambda^k$
where the solution is sought, rather than by the variational formulation.

To make things more concrete, we restrict to a domain
$\Omega\subset\R^3$, and consider the Hodge Laplacian for $k$-forms,
$k=0,1,2$, and $3$.  We also discuss the $\Bfrak^*$ and $\Bfrak$ problems
given by \eqref{bfrak*} and \eqref{bfrak} for each $k$. Again we have

\begin{equation*}
 d^0=\grad, \ d^1=\curl,\ d^2 = \div, 
\end{equation*}
and the de~Rham complex \eqref{drl20} is now realized as
\begin{equation*}
0\to \0H^1(\Omega) \xrightarrow{\grad} \0H(\curl;\Omega)
\xrightarrow{\curl} \0H(\div;\Omega) \xrightarrow{\div} L^2(\Omega)
\to 0,
\end{equation*}
where
\begin{align*}
 \0H(\curl;\Omega)&= \{\, u:\Omega\to\R^3\,|\, u\in L^2,\ \curl u\in L^2,\
u \times n =0 \text{ on } \partial \Omega \, \},\\
 \0H(\div;\Omega)&= \{\, u:\Omega\to\R^3\,|\, u\in L^2,\ \div u\in L^2,\
u \cdot n =0 \text{ on } \partial \Omega\,\}.
\end{align*}

\subsubsection{The Hodge Laplacian for $k=0$}  For $k=0$, the boundary
value problem \eqref{pdes0}--\eqref{sc0} is the Dirichlet problem for the
ordinary scalar Laplacian. The space $H\Lambda^{-1}$ is understood to be $0$,
so $\sigma$ vanishes.  The space
$\0\Hfrak^0$ of harmonic $0$-forms vanishes. The first
differential equation of \eqref{pdes0} vanishes, and the second gives Poisson's
equation
\begin{equation*}
 -\div \grad u = f \quad\text{in $\Omega$}.
\end{equation*}
Similarly, the first boundary condition in \eqref{bcs0} vanishes, while the
second is the Dirichlet condition $u= 0$ on $\partial\Omega$.
The side condition \eqref{sc0} is automatically satisfied.
Nothing additional is obtained by considering the split into the $\Bfrak^*$
and $\Bfrak$ subproblems, since the latter is trivial.

\subsubsection{The Hodge Laplacian for $k=1$}  In this case the 
differential equations and boundary conditions are
\begin{equation}\label{k=10}
\begin{gathered}
 \sigma=-\div u,\ \grad\sigma +\curl\curl u = f-p \quad\text{ in $\Omega$},
\\
\sigma =0, \  u\x n =0 \quad\text{on $\partial\Omega$},
\end{gathered}
\end{equation}
which is a formulation of a boundary value problem for the vector Laplacian
$\curl\curl-\grad\div$. The solution is determined uniquely by the
additional condition that $u$ be orthogonal to $\0\Hfrak^1 = \Hfrak^2$, i.e.,
those vector fields satisfying
\begin{equation*}
\curl p = 0,\ \div p=0 \ \text{ in $\Omega$}, \quad
p\times n =0, \ \text{on $\partial\Omega$}.
\end{equation*}

The $\Bfrak^*_1$ problem \eqref{bfrak*} is defined for $L^2$ vector fields $f$
which are orthogonal to both gradients of functions in $\0H^1(\Omega)$ and the
vector fields in $\0\Hfrak^1$.  In that case, the solution to \eqref{k=10} has
$\sigma=0$ and $p=0$, while $u$ satisfies
\begin{equation*}
\curl\curl u = f,\ \div u=0 \ \text{ in $\Omega$}, \quad
u\x n =0 \ \text{on $\partial\Omega$}.
\end{equation*}
The orthogonality condition $u\perp \0\Hfrak^1$ again determines the solution
uniquely.

Next we turn to the $\Bfrak^1$ problem.  For source functions of the form
$f=\grad F$ for some $F\in \0H^1$, \eqref{k=10} reduces to the problem of
finding $\sigma\in \0H^1$ and $u\in\Bfrak^1=\grad \0H^1$ such that:
\begin{gather*}
\sigma=-\div u,\ \grad\sigma = f \ \text{ in $\Omega$}. 
\end{gather*}
The differential equations may be simplified to $-\grad\div u = f$ and the
condition that $u\in\Bfrak^1$ can be replaced by the differential equation
$\curl u=0$ and the boundary condition $u\times n =0$, together with
orthogonality to $\0\Hfrak^1$.  Now $\grad(\sigma-F)=0$ and $\sigma-F =0$ on
$\partial \Omega$, so $\sigma=F$. Hence, we may rewrite the system as
\begin{equation*}
-\div u=F, \ \curl u=0 \ \text{ in $\Omega$}, \quad
u\times n =0, \ \text{on $\partial\Omega$},
\end{equation*}
which, again, has a unique solution subject to orthogonality to $\0\Hfrak^1$.

\subsubsection{The Hodge Laplacian for $k=2$}  The differential equations
and boundary conditions are
\begin{equation*}
\begin{gathered}
 \sigma=\curl u,\ \curl\sigma -\grad\div u = f-p \quad\text{ in $\Omega$},
\\
\sigma \x n =0, \ u \cdot n =0 \quad\text{on $\partial\Omega$}.
\end{gathered}
\end{equation*}
This is again a formulation of a boundary value problem for the vector
Laplacian $\curl\curl-\grad\div$, but with different boundary conditions than
for \eqref{k=10}, and this time stated in terms of two vector variables, rather
than one vector and one scalar.  This time uniqueness is obtained by imposing
orthogonality to $\0\Hfrak^2 = \Hfrak^1$, the space of vector fields
satisfying
\begin{equation*}
\curl p = 0,\ \div p=0 \ \text{ in $\Omega$}, \quad
p \cdot n =0, \ \text{on $\partial\Omega$}.
\end{equation*}

The $\Bfrak^*_2$ problem arises for source functions of the form $f=\grad F$
for some $F\in H^1$ which we may take to have integral $0$.  We find $\sigma=0$, and $u$ solves
\begin{equation*}
-\div u=F, \ \curl u=0 \ \text{ in $\Omega$}, \quad
u \cdot n =0, \ \text{on $\partial\Omega$},
\end{equation*}
i.e., the same differential equation as for $\Bfrak^1$, but with different
boundary conditions and, of course, now uniqueness is determined by
orthogonality to $\0\Hfrak^2$.

If $\div f=0$, $f \cdot n =0$ on $\partial \Omega$, and $f\perp\0\Hfrak^2$, we
get the $\Bfrak^2$ problem for which the differential equations are
$\sigma=\curl u$, $\curl\sigma =f$ and the condition $\div u=0$ arising from
the membership of $u$ in $\Bfrak^2$.  Thus $u$ solves
\begin{equation*}
\curl\curl u = f,\ \div u=0 \ \text{ in $\Omega$}, \quad
u \cdot n =0, \ \text{on $\partial\Omega$},
\end{equation*}
the same differential equation as for $\Bfrak^*_1$, but with different
boundary conditions.  

\subsubsection{The Hodge Laplacian for $k=3$} The space 
$\0\Hfrak^3$ consists of
the constants (or, if $\Omega$ is not connected, of functions which are
constant on each connected component).  The Hodge Laplacian problem is
\begin{equation*}
 \sigma=-\grad u,\ \div\sigma = f -p \text{ in $\Omega$}, \quad \sigma \cdot n
=0\ \text{on $\partial\Omega$}, \quad \int_{\Omega} u \, q\, dx =0,\ q\in\0\Hfrak^3,
\end{equation*}
which is the Neumann problem for Poisson's equation, where the (piecewise)
constant $p$ is required for there to exist a solution 
and the final condition fixes a unique solution.

\section{The elasticity complex}
\label{sec:elasticity}
In this section, we present another complex,
which we call the \emph{elasticity complex},
and apply it to the development of numerical methods 
for linear elasticity. In contrast to the other 
examples presented above, 
the elasticity complex is 
not a simple variant of the de Rham complex, and, in particular, it
involves a differential operator of second order.
However, the two complexes are related in a subtle manner, via a construction
known as the
Bernstein--Bernstein--Gelfand resolution, explained for example in 
\cite{acta,weak-sym,Eastwood} and references given there.

The equations of elasticity are of great importance in modeling solid structures,
and the need to solve them in engineering applications was probably the primary
reason for the development of the finite element method in the 1960s.  The simplest
finite element methods for elasticity are based on \emph{displacement approaches},
in which the displacement vector field is characterized as the minimizer of an
elastic energy functional.  The design and analysis of displacement finite element
methods is standard and discussed
in many textbooks, (e.g., \cite{Ciarlet}). However, for more general material models,
arising, for example, for incompressible materials or some viscoelastic or plastic
materials, the displacement method is either not feasible or performs poorly, and a mixed approach,
in which the elastic stress and the displacement are taken as unknowns, is the natural
alternative.  In fact, mixed finite element methods for elasticity were proposed in
the earliest days of finite elements \cite{fraeijsdv}, and stable finite elements for the mixed
formulation of elasticity have been sought for over four decades.  These proved very elusive.
Indeed, one of the motivations of the pioneering work of Raviart and Thomas \cite{Raviart-Thomas}
on mixed finite
elements for the Laplacian, was the hope that the solution to this easier problem
would pave the way to such elements
for elasticity, and there were many attempts to generalize their elements to the
elasticity system
\cite{Arnold-Brezzi-Douglas,Arnold-Douglas-Gupta,Johnson-Mercier,Stenberg86,Stenberg88,Stenberg-mafelap}.
However, it was not until 2002 that the first stable mixed finite elements for elasticity
using polynomial shape functions were discovered by two of the present authors \cite{arnold-winther02},
based on techniques and insights from the then nascent finite element exterior calculus.
See also \cite{arnold-falk-winther2}.  They
developed and analyzed a family of methods for plane (2-dimensional) elasticity.  The basic elements
in the family involve a space of shape functions intermediate between quadratics and cubics
for the stress, with 24 degrees of freedom, while the displacement field consists of
the complete space of piecewise linear vector fields, without continuity constraints.  It
is these elements we will describe here.

The plane elasticity elements of \cite{arnold-winther02} were generalized to three dimensions
in \cite{arnold-awanou-winther}.  See also \cite{adams-cockburn04}. These elements are, however,
quite complicated (162 degrees of freedom per tetrahedron for the stress).  A much more promising
approach involves a variant Hilbert complex for elasticity, the elasticity complex with
weak symmetry.  In \cite{acta} and \cite{weak-sym}, stable families of finite elements are
derived for this formulation in $n$-dimensions, again using techniques from the
finite element exterior calculus.  For the lowest order stable
elements, the stress space is discretized by
piecewise linears with just 12 degrees of freedom per triangle in two
dimensions, 36 per tetrahedron in three dimensions, while the other variables
are approximated by piecewise constants.

\subsection{The elasticity system}
The equations of linear elasticity arise as a system consisting of
a constitutive equation and an equilibrium equation:
\begin{equation}\label{mixed-system}
A\sigma = \eps u, \qquad \div \sigma = f \quad \text{in }\Omega.
\end{equation}
Here the unknowns $\sigma$ and $u$ denote the stress and displacement
fields engendered by a body force $f$ acting on a linearly elastic
body which occupies a region $\Omega \subset \R^n$. The constitutive equation
posits a linear relationship between the linearized deformation or strains due to
the displacement and the stresses.  The equilibrium equation states that the
stresses, which measure the internal forces in the body, are in equilibrium with
the externally applied force.  The stress field $\sigma$
takes values in the space $\Sym:=\R^{n\x n}_{\text{sym}}$ of symmetric
matrices and the displacement field $u$ takes values in $\V:=\R^n$. The differential operator
$\eps$ is the symmetric part of the gradient, the div operator is
applied row-wise to a matrix, and the compliance tensor
$A=A(x):\Sym\to\Sym$ is a bounded and symmetric, uniformly positive
definite operator reflecting the material properties of the body.  If the body
is clamped on the boundary $\partial\Omega$ of $\Omega$, then the
appropriate boundary condition for the system \eqref{mixed-system} is $u =
0$ on $\partial\Omega$. For simplicity, this boundary condition will
be assumed here.  

The pair $(\sigma,u)$ can alternatively be characterized as
the unique critical point of the Hellinger--Reissner functional
\begin{equation*}
I(\tau,v)=\int_\Omega\bigl(\frac12
A\tau:\tau+\div\tau\cdot v-f\cdot v\bigr)\,dx.
\end{equation*}
The critical point is sought 
among $\tau\in\Hdiv$,
the space of square-integrable symmetric matrix fields with
square-integrable divergence, and $v\in L^2(\Omega;\V)$, the
space of square-integrable vector fields.
Equivalently, the pair $(\sigma,u)\in \Hdiv$ $\x L^2(\Omega;\V)$
is the unique solution to the weak formulation
of the system \eqref{mixed-system}:
\begin{equation}\label{mixed-system_w}
\begin{array}{lrll}
\int_\Omega(A\sigma:\tau + \div\tau\cdot u)\,dx &=& 0,
\quad
&\tau \in \Hdiv,\\[0.2cm]
\int_\Omega \div\sigma\cdot v \,dx  &=& 
\int_\Omega f \cdot v\,dx, \quad
&v \in L^2(\Omega;\V).
\end{array}
\end{equation}

\subsection{The elasticity complex and its discretization}
The discretization
of the problem \eqref{mixed-system_w} is closely tied to
discretization of the elasticity complex.
For a two dimensional domain $\Omega$, the elasticity complex takes the form
\begin{equation}\label{2d-elasticity-complex}
\begin{CD}
0 \to\, @. H^2(\Omega)  @> J >> \Hdiv @>\div >>
L^2(\Omega;\R^2) @. \,\to 0.
\end{CD}
\end{equation}
The operator $J$ is a second-order differential operator
mapping scalar fields into symmetric matrix fields, namely
the rotated Hessian
\[
Jq= \begin{pmatrix} \partial^2_{y} q &
  -\partial_{x}\partial_{y} q \\
  -\partial_{x}\partial_{y} q &  \partial^2_{x} q  \end{pmatrix} = O (\grad \grad q) O^T, \quad O = 
\begin{pmatrix} 0 & 1 \\ -1 & 0 \end{pmatrix}.
\]
We now show that this complex is closed, i.e., that the
range of the two operators $J$ and $\div$ are closed subspaces.  The null space of $J$ is just
the space $\P_1$ of linear polynomials and we have $\|q\|_{H^2}\le c\|Jq\|_{L^2}$
for all $q\in H^2(\Omega)$ with $q\perp\P_1$, which implies that the range of
$J$ is closed.  We now show that $\div$ maps $\Hdiv$ onto $L^2(\Omega;\R^2)$, using the
fact that the scalar-valued divergence maps $H^1(\Omega;\R^2)$ onto $L^2(\Omega)$.  Given
$f\in L^2(\Omega;\R^2)$ the latter result implies the existence
of $\tau\in H^1(\Omega;\R^{2\x2})$ with $\div\tau=f$.  This $\tau$ may not
be symmetric.  However, for any $v\in H^1(\Omega;\R^2)$, we let $\curl v$ denote the matrix
with $i$th row equal to
$\curl v_i=(-\partial v_i/\partial x_2,\partial v_i/\partial x_1)$.  Then
$\div(\tau+\curl v)=f$ and $\tau+\curl v\in\Hdiv$ if $\div v=\tau_{21}-\tau_{12}$.
Thus $\div:\Hdiv\to L^2(\Omega;\R^2)$ is indeed surjective.
Setting
$W^0 =L^2(\Omega)$, $W^1 = L^2(\Omega;\Sym)$, and $W^2 =
L^2(\Omega;\R^2)$, we obtain a closed Hilbert complex with domain complex \eqref{2d-elasticity-complex}.

Now we describe only the simplest discretization from \cite{arnold-winther02}.
The space $V^0_h \subset V^0 =
H^2(\Omega)$ is the finite element space whose shape functions on each triangle $T$ are the
quintic polynomials $\P_5(T)$ and whose degrees of freedom are
the values, first, and second derivatives at each vertex,
together with the mean value of the normal derivatives along each edge,
which gives $21=\dim\P_5(T)$ degrees of freedom per triangle.  The resulting finite
element space $V^0_h$ is 
referred to in the finite element literature as the
Argyris space, and is easily seen to be contained in $C^1(\Omega)$, and hence
in $H^2(\Omega)$.  In fact, the Argyris space is in a certain sense the simplest finite
element space contained in $C^1(\Omega)$.  However the Argyris space does not coincide
with the space of all $C^1$ piecewise quintics (which is not a finite element space,
as it cannot be defined via degrees of freedom).  Elements of the Argyris space
have extra smoothness at the vertices, where they are $C^2$.

The space 
$V_h^1 \subset V^1 = \Hdiv$ consists 
of those continuous piecewise cubic symmetric matrix fields whose divergence
is piecewise linear. The degrees of freedom determining an
element $\sigma \in V_h^1$ are 
\begin{itemize}
\item the value of $\sigma$ at each vertex,
\item the moments of degree $0$ and $1$ of the two normal
  components of $\sigma$ on each edge, and
\item the moments of degree $0$ of $\sigma$ on each triangle.
\end{itemize}
Hence, the restriction of $\sigma$ to a triangle $T \in \T_h$
is uniquely determined by $24$ degrees of freedom.
Finally, the space $V_h^2 \subset V^2 =L^2(\Omega;\R^2)$
consists of all piecewise linear vector fields with respect to $\T_h$,
without any imposed continuity.  For degrees of freedom, we choose the moments of degrees
$0$ and $1$ for each component on each triangle $T$.  Using the degrees of
freedom, we also define projections $I^k_hv\in V^k_h$ defined for sufficiently smooth
$v\in V^k$ by $\phi(I^k_hv)=\phi(v)$ for all degrees of freedom $\phi$.

It is straightforward to check that we have the subcomplex property
$J V^0_h\subset V^1_h$ and $\div V^1_h\subset V^2_h$ and that the interpolation operators
$I_h:=I^k_h$ give a commuting diagram
\begin{equation}\label{full-2d-diagram}
\begin{CD}
0 \to\, @. V^0 \cap C^2 @>J >> V^1\cap C^0 @>\div >> V^2 @. \,\to 0\\
@. @V{I^0_{h}}VV @V{I^1_{h}}VV  @V{I^2_{h}}VV 
\\
0 \to\, @.V^0_h @>J >> V^1_h @>\div >> V^2_h @. \,\to 0
\end{CD}
\end{equation}
However, since $I_h^0$ involves point values of the second derivative and
$I_h^1$ involves point values, these operators do not extend to bounded operators
on the Hilbert spaces $V^0$ and $V^1$.

In order to apply the general theory of approximation of 
Hilbert complexes developed in Section~\ref{sec:HCA}, we will modify
the interpolation operators $I_h^k$ in \eqref{full-2d-diagram}
to obtain bounded cochain projections $\pi_h^k:W^k\to V^k_h$ for $k =0,1,2$.
As in Section~\ref{subsec:bcp}, the $\pi^k_h$ will be obtained from $I^k_h$
via smoothing.

\subsection{Bounded cochain projections for the elasticity complex}
First, we define an appropriate pullback of an \emph{affine map}, and
show that it defines a cochain map for the elasticity sequence.  Let
$F : \Omega \to \Omega' \subset \R^2$ be an affine map, so
that $DF = DF(x)$ is a constant $2\times 2$ matrix.
Let  $B$ be the matrix $B =\det(DF)(DF)^{-1} = O(DF)^T O^T$ and
define pullbacks $F^* = F^*_k : W^k(\Omega') \to W^k(\Omega)$ by
\begin{alignat*}{2}
F^*v &=v\circ F, \quad && v \in W^0(\Omega'),\\
F^* v  &=B (v\circ F)B^T, \quad&& v \in W^1(\Omega'),\\
F^*v &=\det(DF) B (v\circ F), \quad&& v \in W^2(\Omega').
\end{alignat*}
It is straightforward to check that, as long as $F$ is affine, $F^*$ commutes
with the differential operators $J$ and $\div$. In other words,
we obtain the commuting diagram
\begin{equation*}
\begin{CD}
0 \to\, @. V^0(\Omega') @>J >> V^1(\Omega') @>\div >> V^2(\Omega') @. \,\to 0\\
@. @V{F^*}VV @V{F^*}VV  @V{F^*}VV 
\\
0 \to\, @. V^0(\Omega) @>J >> V^1(\Omega) @>\div >> V^2(\Omega) @. \,\to 0.
\end{CD}
\end{equation*}

To define a smoothing operator which maps $W^k$ into itself, we need to
make the assumption that the domain \emph{$\Omega$ is star-shaped with
respect to some point in its interior} (which we assume, without further
loss of generality, to be the origin).  We then let
$a=\max_{x\in\partial\Omega}|x|^{-1}$ and define the dilation
$F_\delta :\R^2 \to \R^2$ by $F_\delta(x) = x/(1+ a \delta)$,
$\delta>0$. Then $F_\delta$ maps the $\delta$-neighborhood of $\Omega$,
$\Omega_\delta$, into $\Omega$ and so
$F^*_\delta$ maps $L^2(\Omega)$ to $L^2(\Omega_\delta)$.  Composing with
a standard mollification
\[
v \mapsto  \int_{|z|<1} \rho(z) v(x + \delta
z)\, dz,
\]
we obtain a smoothing operator $R_{\delta}$ mapping
$W^k$ into $W^k$.  Here $\rho : \R^2 \to \R$ is a smooth, nonnegative function supported
in the unit ball and with integral equal to one.  This smoothing
commutes with any differential operator with constant coefficients, in
particular with $J$ and $\div$.

The construction now proceeds as in Section~\ref{subsec:bcp}, and
we just give an outline of it.  For simplicity, we assume a quasi-uniform family
of meshes,
although extension to general shape regular meshes could be made as in Section~\ref{subsec:bcp}.
Let $\epsilon>0$ be fixed.  Define operators
$Q^k_h:W^k\to W^k_h$
by $Q_h^k = I_h^k \circ R_{\epsilon h}$. The operator $Q_h^k$
maps $W^k$ into $V_h^k$. In fact, it can be seen 
by a scaling argument,  that the operators 
$Q_h^k$ are bounded as operators in $\Lin(W^k,W^k)$, uniformly in
$h$. Furthermore, $Q_h^k$
commutes with
$J$ and $\div$, i.e.,
\begin{equation}\label{Q-commute}
Q_h^1 \circ J = J \circ Q_h^0 \quad \text{and} \quad 
Q_h^2 \circ \div = \div \circ \, Q_h^1.
\end{equation}
However, the operators $Q_h^k$ are not projections, since they do not
restrict to the identity on $V_h^k$. Therefore, we define the
desired projections $\pi_h^k : W^k \to V_h^k$ as $\pi^k_h =
(Q^k_h|_{V_h^k})^{-1} Q^k_h$.
To justify this definition, we need to argue that
$(Q^k_h|_{V_h^k})^{-1}$ exists. In fact, for $\epsilon$ sufficiently small,
but fixed, $(Q^k_h|_{V_h^k})^{-1}$
is even uniformly bounded in $\Lin(W_h^k, W_h^k)$, where $W_h^k =
V_h^k$,
but equipped with the norm of $W^k$.  This follows from
the fact that the operators $Q^k_h|_{V^k_h}$ converge to the identity
with $\epsilon$ in $\Lin(W_h^k, W_h^k)$, uniformly in $h$, established by
a scaling argument.

We conclude that if the parameter $\epsilon > 0$ is fixed, but sufficiently
small, then the operators $\pi_h^k$ are
well-defined and uniformly bounded with respect to $h$ in $\Lin(W^k,W^k)$.
Furthermore, the
diagram 
\[
\begin{CD}
0\to\, @. V^0 @>J >> V^1 @>\div >> V^2 @. \,\to 0\\
@. @V{\pi_{h}}VV @V{\pi_{h}}VV  @V{\pi_{h}}VV 
\\
0 \to\, @.V^0_h @>J >> V^1_h @>\div >> V^2_h @. \,\to 0,
\end{CD}
\]
commutes, as a consequence of \eqref{Q-commute}.

Thus, under the assumption of a star-shaped domain, we have established
the existence of bounded cochain projections and so the results of
Section~\ref{subsec:well-p-h-hc} apply to the elasticity complex. This means that the
corresponding Hodge Laplacians are well-posed and that the
discretizations
are stable.  In particular, using the spaces $V^1_h$ and $V^2_h$, we obtain
a stable, convergent mixed discretization 
of the elasticity equations.

\begin{remark}
Stability and convergence of the mixed finite element method for elasticity
using the spaces $V^1_h$ and $V^2_h$ described above is proven in \cite{arnold-winther02}
without restricting to star-shaped domains.  However we do not
know how to construct bounded cochain projections without this restriction.
\end{remark}

\bibliographystyle{amsplain}
\bibliography{bulletin}

\end{document}